\documentclass[12pt, a4paper]{amsart}
\usepackage{amsfonts,latexsym,graphicx,
amssymb,amsthm,amsmath}
\newcounter{f1}
\newcounter{f2}
\newcounter{f3}
\newcounter{f4}
\newcounter{f5}
\newcounter{f6}
\newcounter{f7}
\newcounter{f8}
\newcounter{f9}
\newcounter{f10}
\newcounter{f11}
\newcounter{f12}
\newtheorem{theorem}{Theorem}[section]
\newtheorem{proposition}[theorem]{Proposition}
\newtheorem{lemma}[theorem]{Lemma}
\newtheorem{corollary}[theorem]{Corollary}
\newtheorem{remark}[theorem]{Remark}
\newtheorem{main-theorem}{Main Theorem}

\title[The orbit decomposition and orbit type]
{The classification of orbits
on certain exceptional Jordan algebra
under the automorphism group.}
\author{Akihiro Nishio}
\date{2012/4/21}
\begin{document}

\begin{abstract}
Let $\mathcal{J}^1$ be the real form of 
complex simple Jordan algebra
with the automorphism group $G$ of type
$\mathrm{F}_{4(-20)}$.
Explicitly, we give
the orbit decomposition of $\mathcal{J}^1$
under the action of $G$ 
and
determine the Lie group structure of 
stabilizer for each $G$-orbit 
on $\mathcal{J}^1$. 
\end{abstract}


\address{
Graduate~school~of~Engineering,
University~of~Fukui,
Fukui-shi, 910-8507, Japan}

\email{nishio@quantum.apphy.u-fukui.ac.jp}

\maketitle

\addtocounter{section}{-1}
\section{Introduction.}\label{itd}
Let $G$ be an exceptional linear Lie group of type
$\mathrm{F}_4$ defined by the automorphism group
of an exceptional Jordan algebra.
The objective of this article is for $G=\mathrm{F}_{4(-20)},$  
to solve the following problem:

A classification of $G$-orbits:
\begin{quote}
(A)
the decomposition of the space of elements
in which $G$ is represented,
into equivalence classes or "orbits".
\smallskip

\noindent
(B)
the determining the Lie group structure of the
stabilizer for each $G$-orbit.
\end{quote}
The orbit decompositions
are given  for $G=\mathrm{F}_4^{\mathbb{C}}$ and $\mathrm{F}_{4(4)}$
in \cite{NY2010}. 
\medskip

The definition of an exceptional Jordan algebra $\mathcal{J}^1$ 
with identity $E$ is given in \S 1.
For $X,Y\in\mathcal{J}^1$,
the Jordan product of $X$ and $Y$ is denoted by $X\circ Y$.
$\mathcal{J}^1$ has the trace $\mathrm{tr}(X)$ as usual,
and  
$(X|Y)=\mathrm{tr}(X\circ Y)$ gives
a non-degenerate indefinite inner product.  
Moreover, $\mathcal{J}^1$ has the cross product 
of H. Freudenthal $X\times Y$.
Then
the determinant $\mathrm{det}(X)$,
the characteristic polynomial $\Phi_X(\lambda)$ degree $3$
and the characteristic roots of $X$ are defined.
The linear Lie group $\mathrm{F}_{4(-20)}$ is 
defined to be the automorphism group of  
$\mathcal{J}^1$
with the Jordan product.
The action of 
$\mathrm{F}_{4(-20)}$ preserves 
the identity element $E$,
the trace,
the inner product, 
the cross product,
the determinant,
the characteristic polynomial
and the characteristic roots
with multiplicity.

We give the elements $E_1,~E_2,~E_3,~P^+,~P^-,~
Q^+(1) \in \mathcal{J}^1$ in \S 1
and
determine typical orbits
in \S 5.
These orbits are
the {\it exceptional hyperbolic planes}
$\mathcal{H}({\bf O})$, $\mathcal{H}'({\bf O})$,
and the {\it exceptional null cones} $\mathcal{N}_1^+({\bf O})$,
$\mathcal{N}_1^-({\bf O})$, $\mathcal{N}_2({\bf O})$.

\begin{proposition}\label{hp-np}
The following equations hold.
\begin{align*}
\tag{\ref{hp-np}.a}
\mathcal{H}({\bf O})&=Orb_{\mathrm{F}_{4(-20)}}(E_1).\\
\tag{\ref{hp-np}.b}
\mathcal{H}'({\bf O})&=Orb_{\mathrm{F}_{4(-20)}}(E_2)
=Orb_{\mathrm{F}_{4(-20)}}(E_3).\\
\tag{\ref{hp-np}.c}
\mathcal{N}_1^+({\bf O})&=Orb_{\mathrm{F}_{4(-20)}}(P^+).\\
\tag{\ref{hp-np}.d}
\mathcal{N}_1^-({\bf O})&=Orb_{\mathrm{F}_{4(-20)}}(P^-).\\
\tag{\ref{hp-np}.e}
\mathcal{N}_2({\bf O})&=Orb_{\mathrm{F}_{4(-20)}}(Q^+(1)).
\end{align*}
\end{proposition}
\medskip

Let $X\in\mathcal{J}^1$.
In \S 1, we define the minimal space $V_X$ which 
is the minimal dimensional linear
subspace of $\mathcal{J}^1$ being closed under the
cross product and
containing the elements $E,X\in \mathcal{J}^1$.
In order to present the intersection of $V_X$ and typical
orbits, we define the elements 
$E_{X,\lambda_1},~W_{X,\lambda_1} \in V_X$ 
with $\lambda_1\in\mathbb{R}$ in \S 1
and the traceless component of $X$: $p(X)\in V_X$.
Using
the set of all characteristic roots of $X$
with multiplicities
and
the intersection of $V_X$ and typical
orbits,
the $\mathrm{F}_{4(-20)}$-orbit of $X$ can be described. 
\medskip

\begin{main-theorem}\label{orb-decomposition}
$\mathrm{F}_{4(-20)}$-orbits on $\mathcal{J}^1$ 
are given as follows.
\par
{\rm (1)} Assume that $X\in\mathcal{J}^1$ admits 
characteristic roots $\lambda_1>\lambda_2>\lambda_3.$
Then 
there exists a unique $i\in\{1,2,3\}$ 
such that 
\begin{align*}
{\rm (i)}&~
\mathcal{H}({\bf O})\cap V_X
=\{E_{X,\lambda_i}\},\\
{\rm (ii)}&~
\mathcal{H}'({\bf O})\cap V_X
=\{E_{X,\lambda_{i+1}},~E_{X,\lambda_{i+2}}\}
~~\text{with}~E_{X,\lambda_{i+1}}\ne E_{X,\lambda_{i+2}}
\end{align*}
where indexes $i,$ $i+1,$ $i+2$ are counted modulo $3$. 
In this case, $X$ can be transformed to one of 
the following canonical forms by the action of $\mathrm{F}_{4(-20)}$.
\begin{center}
\begin{tabular}{ll}
\hline
Cases & Canonical forms of $X$\\
\hline\hline
{\rm 1.} $E_{X,\lambda_1}\in\mathcal{H}({\bf O})$
     & ${\rm diag}(\lambda_1,\lambda_2,\lambda_3)$\\
\smallskip
{\rm 2.} $E_{X,\lambda_2}\in\mathcal{H}({\bf O})$
     & ${\rm diag}(\lambda_2,\lambda_3,\lambda_1)$\\
\smallskip
{\rm 3.} $E_{X,\lambda_3}\in\mathcal{H}({\bf O})$
     & ${\rm diag}(\lambda_3,\lambda_1,\lambda_2)$\\
\hline
\end{tabular}\\
\end{center}
\par
{\rm (2)} Assume that $X\in\mathcal{J}^1$ admits 
characteristic roots $\lambda_1\in\mathbb{R},$ 
$p\pm \sqrt{-1}q$ with $p\in\mathbb{R}$ and $q>0$.
Then $X$ can be transformed to 
the following canonical form by the action of $\mathrm{F}_{4(-20)}$.
\begin{center}
\begin{tabular}{lll}
\hline
\quad\quad\quad & 
The canonical form of $X$\\
\hline\hline
{\rm 4.} \quad\quad\quad & 
${\rm diag}(p,p,\lambda_1)+F_3^1(q)$
\\
\hline
\end{tabular}\\
\end{center}
\par
{\rm (3)} Assume that $X\in\mathcal{J}^1$ 
admits characteristic roots $\lambda_1$ 
of multiplicity $1$ and $\lambda_2$ of multiplicity $2.$
Then 
\begin{align*}
{\rm (i)}&~E_{X,\lambda_1}\in\mathcal{H}({\bf O})
\coprod \mathcal{H}'({\bf O}),~~
{\rm (ii)}~
W_{X,\lambda_1}\in \{0\}\coprod \mathcal{N}_1^+({\bf O})
\coprod \mathcal{N}_1^-({\bf O}),\\
{\rm (iii)}&~E_{X,\lambda_1}\in\mathcal{H}({\bf O})
\Rightarrow W_{X,\lambda_1}=0,
\quad 
{\rm (iv)}~
W_{X,\lambda_1}\ne 0 
\Rightarrow E_{X,\lambda_1}\in\mathcal{H}'({\bf O})
\end{align*}
In this case, $X$ can be transformed to 
one of the following canonical forms by the action of
 $\mathrm{F}_{4(-20)}$.
\begin{center}
\begin{tabular}{lll}
\hline
Cases & Canonical forms of $X$\\
\hline\hline
{\rm 5.} $E_{X,\lambda_1}\in \mathcal{H}({\bf O})$  & 
    ${\rm diag}(\lambda_1,\lambda_2,\lambda_2)$\\
\smallskip
{\rm 6.} $E_{X,\lambda_1}\in \mathcal{H}'({\bf O}),~
W_{X,\lambda_1}=0$ & 
    ${\rm diag}(\lambda_2,\lambda_2,\lambda_1)$\\
\smallskip
{\rm 7.} $W_{X,\lambda_1}\in \mathcal{N}_1^+({\bf O})$
    & 
    ${\rm diag}(\lambda_2,\lambda_2,\lambda_1)+P^+$\\
\smallskip
{\rm 8.} $W_{X,\lambda_1}\in \mathcal{N}_1^-({\bf O})$
    & 
    ${\rm diag}(\lambda_2,\lambda_2,\lambda_1)+P^-$\\
\hline
\end{tabular}\\
\end{center}
\par
{\rm (4)} Assume that $X\in\mathcal{J}^1$ 
admits a characteristic root of multiplicity $3.$
Then 
\[p(X)\in\{0\}\coprod
\mathcal{N}_1^+({\bf O})\coprod\mathcal{N}_1^-({\bf O})
\coprod\mathcal{N}_2({\bf O}).\]
In this case, 
$X$ can be transformed to 
one of the following canonical forms by the action of
 $\mathrm{F}_{4(-20)}$.
\begin{center}
\begin{tabular}{lll}
\hline
Cases & Canonical forms of $X$\\
\hline\hline
{\rm ~9.} $p(X)=0$  & 
  $3^{-1}\mathrm{tr}(X)E$\\
\smallskip
{\rm 10.} $p(X)\in \mathcal{N}_1^+({\bf O})$  &
  $3^{-1}\mathrm{tr}(X)E+P^+$\\
  \smallskip
{\rm 11.} $p(X)\in \mathcal{N}_1^-({\bf O})$  &
  $3^{-1}\mathrm{tr}(X)E+P^-$\\
\smallskip
{\rm 12.} $p(X)\in \mathcal{N}_2({\bf O})$ &
 $3^{-1}\mathrm{tr}(X)E+Q^+(1)$\\
\hline
\end{tabular}\\
\end{center}
\par
{\rm (5)}  
Under the action of  $\mathrm{F}_{4(-20)}$,
these canonical 
forms {\rm 1} -- {\rm 12}
in {\rm (1)} -- {\rm (4)}
cannot be transformed from each other.
\end{main-theorem}
\bigskip

The Heisenberg group ${\rm H}_{{\rm Im}{\bf O}, {\bf O}}$
in the sense of J.A.~Wolf 
\cite{Wja1975,Wja2007}
 and
its subgroups ${\rm Im}{\bf O}$
and ${\rm H}_{{\rm Im}{\bf O},{\rm Im}{\bf O}}$
are given in \S \ref{sdp}.
The Lie group structure of
stabilizer for each $\mathrm{F}_{4(-20)}$-orbit 
on $\mathcal{J}^1$ is given as follows.

\begin{main-theorem}\label{stabilizer-main}
{\rm (Orbit types of $\mathrm{F}_{4(-20)}$-orbits
on $\mathcal{J}^1$).}

\noindent
The Lie group types of stabilizers of 
the canonical forms $1$--$12$
in Main-Theorem~\ref{orb-decomposition}
are given in the following table.

\begin{center}
 \begin{tabular}{ll}
  \hline
   Canonical forms  & 
Types of stabilizers\\
  \hline\hline
  {\rm ~1.} $ {\rm diag}(\lambda_1,\lambda_2,\lambda_3)$
     & ${\rm Spin}(8)$\\
  \smallskip
  {\rm ~2.} ${\rm diag}(\lambda_2,\lambda_3,\lambda_1)$
     & ${\rm Spin}(8)$\\
  \smallskip
  {\rm ~3.} ${\rm diag}(\lambda_3,\lambda_1,\lambda_2)$
     & ${\rm Spin}(8)$\\
  \smallskip
  {\rm ~4.} 
  ${\rm diag}(p,p,\lambda_1)+F_3^1(q)$
   & ${\rm Spin}^0(7,1)$\\
 {\rm ~5.}  
    ${\rm diag}(\lambda_1,\lambda_2,\lambda_2)$
     & ${\rm Spin}(9)$\\
 \smallskip
 {\rm ~6.} 
    ${\rm diag}(\lambda_2,\lambda_2,\lambda_1)$
     & ${\rm Spin}^0(8,1)$\\
 \smallskip
 {\rm ~7.} 
    ${\rm diag}(\lambda_2,\lambda_2,\lambda_1)+P^+$
     &  ${\rm Spin}(7)\ltimes 
{\rm Im}{\bf O}$\\
 \smallskip
 {\rm ~8.} 
    ${\rm diag}(\lambda_2,\lambda_2,\lambda_1)+P^-$
     & ${\rm Spin}(7)\ltimes 
{\rm Im}{\bf O}$\\
 \smallskip
  {\rm ~9.} $3^{-1}\mathrm{tr}(X)E$
   & $\mathrm{F}_{4(-20)}$\\
 \smallskip
  {\rm 10.} 
  $3^{-1}\mathrm{tr}(X)E+P^+$
    & ${\rm Spin}(7)
\ltimes {\rm H}_{{\rm Im}{\bf O}, {\bf O}}$\\
  \smallskip
  {\rm 11.}
  $3^{-1}\mathrm{tr}(X)E+P^-$
    & ${\rm Spin}(7)
\ltimes {\rm H}_{{\rm Im}{\bf O}, {\bf O}}$\\
 \smallskip
 {\rm 12.} 
 $3^{-1}\mathrm{tr}(X)E+Q^+(1)$
   & $\mathrm{G}_2\ltimes 
{\rm H}_{{\rm Im}{\bf O}, 
{\rm Im}{\bf O}}$\\
 \hline
 \end{tabular}\\
\end{center}
\end{main-theorem}

\section{Preliminaries.}\label{prl}
Denote the Cartesian $n$-power of a set $X$
by $X^n:=X\times \cdots \times X$ ($n$ times)
such as $\mathrm{SO}(8)^3
=\mathrm{SO}(8) \times \mathrm{SO}(8) \times
\mathrm{SO}(8)$.
Let $\mathbb{R}$ be the field of real numbers 
and $\mathbb{C}:=\mathbb{R}\oplus \sqrt{-1}\mathbb{R}$ 
the field of complex numbers.
Denote $\mathbb{F}=\mathbb{R}$ or $\mathbb{C}$. 
Let $V$ be a $\mathbb{F}$-linear space,
$\mathrm{GL}_{\mathbb{F}}(V)$ 
the group of $\mathbb{F}$-linear automorphism of $V$, and
${\rm End}_{\mathbb{F}}(V)$ the linear space of 
$\mathbb{F}$-linear endomorphisms on $V$.
For a mapping $f:V\rightarrow V$ and $c\in\mathbb{F}$, 
put $V_{f,c}:=\{v\in V|~ f(v)=cv\}$
and $V_f:=V_{f,1}$.
A subset $C$ in $V$ is said to be a {\it cone}
if $x\in C$ and $\lambda>0$ imply that $\lambda x\in C$.
The {\it exponential} of 
$f\in {\rm End}_{\mathbb{F}}(V)$
is defined by 
$\exp f=\sum_{k=0}^{\infty}\frac{f^k}{k!}
\in \mathrm{GL}_{\mathbb{F}}(V).$ 
Let $G$ be a subgroup of 
$\mathrm{GL}_{\mathbb{F}}(V)$ 
and $\phi$ an automorphism  on $G$.
Denote the subgroup $\{g\in G |~ \phi g=g\}$ of $G$
as $G^\phi$. 
For $v_1,\cdots,v_n\in V$,
the pointwize  stabilizer of $\{v_1,\cdots,v_n\}$
in $G$ is denoted by $G_{v_1,\cdots,v_n}$.
For $v\in V$, 
the $G$-{\it orbit} of $v$ is denoted by
$Orb_G(v):=\{g v|~g\in G\}$.

Let $V$ be an $\mathbb{R}$-linear space.
Its complexification 
$V \otimes_{\mathbb{R}} \mathbb{C}$
denoted by $V^\mathbb{C}$.
For $f\in{\rm End}_{\mathbb{R}}(V)$, 
its complexification 
is written 
by the same letter $f$. 
The complex conjugation on $V^\mathbb{C}$ with respect to 
$V$ is denoted by $\tau$: 
$\tau (u+\sqrt{-1}v):=u-\sqrt{-1}v$
for all $u+\sqrt{-1}v
\in V^\mathbb{C}$ with $u,v\in V$.

Let $V$ be a $\mathbb{F}$-linear space.
A {\it quadratic form} on $V$ 
is a mapping $\mathrm{q}:V \to \mathbb{F}$
such that
(i) $\mathrm{q}(\lambda v)=\lambda^2\mathrm{q}(v)$
for all $\lambda\in\mathbb{F}$ and $v\in V,$
(ii) the associated symmetric 
form $\mathrm{q}:V\times V\to \mathbb{F}$:
$\mathrm{q}(v,w):
=2^{-1}\left(\mathrm{q}(v)+\mathrm{q}(w)
-\mathrm{q}(v-w)\right)$
is bilinear. 
The pair $(V,\mathrm{q})$ is called a {\it quadratic space}.
For
a quadratic space $(V',\mathrm{q}')$
and  a $\mathbb{F}-$linear map $f:V\rightarrow V'$,
the quadratic form $f^*q'$ on $V$
is given
by
$f^*q'(x):=q'(fx)$.
An {\it isomorphism}
$f:(V,\mathrm{q})\rightarrow(V',\mathrm{q}')$
is defined as 
$f:V\rightarrow V'$ is 
a $\mathbb{F}-$linear isomorphism
and $f^*q'=q$.
Denote the {\it orthogonal group} of $(V,\mathrm{q})$
by
$\mathrm{O}(V,\mathrm{q}):=
\{g\in \mathrm{GL}_\mathbb{F}(V)|
~g^*\mathrm{q}=\mathrm{q}\}=
\{g\in \mathrm{GL}_\mathbb{F}(V)|
~\mathrm{q}(gv,gw)=\mathrm{q}(v,w)\}$
and the {\it special orthogonal group} of $(V,\mathrm{q})$
by $\mathrm{SO}(V,\mathrm{q}):=\{g\in\mathrm{O}(V,\mathrm{q})|
~\mathrm{det}(g)=1\}$
where $\mathrm{det}(g)$ is the determinant of 
$g\in\mathrm{End}_\mathbb{F}(V)$.
Let $k,l$ be non-negative integers
with $k+l>0$.  
A quadratic from $\mathrm{q}_{k,l}$
on  $\mathbb{R}^{k+l}$ is defined by
$\mathrm{q}_{k,l}(x):=-\sum_{i=1}^k x_i^2
+\sum_{i=1}^lx_{k+i}^2$
for $x=(x_1,\cdots,x_{k+l})$,
and denote the quadratic space 
by $(\mathbb{R}^{k,l},\mathrm{q}_{k,l})$.
Assume that $(V,\mathrm{q})$ is an $\mathbb{R}-$quadratic space 
and the quadratic form $\mathrm{q}$ 
is not the trivial quadratic from $\mathrm{q}'\equiv 0$.
Put the subspace
$\mathrm{rad}(V,\mathrm{q}):=\{v\in V|~\mathrm{q}(v,w)=0
~~\text{for~all}~
w\in V\}$ in $V$.
Then there exist a subspace $W$ of $V$ such that 
$V=W\oplus \mathrm{rad}(V,\mathrm{q})$
and $(W,\mathrm{q})$ is isomorphic to 
$(\mathbb{R}^{k,l},\mathrm{q}_{k,l})$ for some integers $k,l$.
The pair of integers $(k, l)$ 
depends only on the quadratic form
$\mathrm{q}$ and 
is called the {\it signature} of
the quadratic form $\mathrm{q}$.
\medskip

Let ${\bf O}$ be the $\mathbb{R}$-algebra of 
{\it octonions} \cite{Fh1951,Dd1978,Yi_arxiv} with a base 
$1$, $e_1$, $e_2$, $e_3$, 
$e_4$, $e_5$, $e_6$, $e_7$ 
and the multiplications among them are 
given as follows: 
$1$ is the unit of $\mathbb{R};$
$e_i^2=-1;$ $e_ie_j+e_je_i=0$ for $i\neq j;$ 
$e_le_m=e_n,$ $e_me_n=e_l$ and $e_ne_l=e_m$
for each $(l,m,n)\in\{(1,2,3),~(3,5,6),~(6,7,1),~(1,4,5),\\
(3,4,7),~(6,4,2),~(2,5,7)\}.$
We write $e_0$ for the unit $1$ of ${\bf O}$.
Let 
${\bf O}^\mathbb{C}$
be the complexification of ${\bf O}$ 
with the complex conjugation $\tau$.
Denote $\tilde{\bf O}:=$ ${\bf O}$ 
or ${\bf O}^\mathbb{C}$.
Let $x=\sum_{i=0}^7x_ie_i$ and   
$y=\sum_{i=0}^7y_ie_i
\in \tilde{\bf O}$ with $x_i,y_i\in\mathbb{F}$.
The {\it conjugation} is defined by
$\overline{x}:=x_0-\sum_{i=1}^7x_ie_i$,
the {\it inner product} $(x|y):=\sum_{i=0}^7x_iy_i$, 
the {\it quadratic form}
${\rm n}(x):=(x|x)$,
the {\it
vector part} ${\rm Im}(x):=2^{-1}(x-\overline{x})$
and the {\it scalar part}
${\rm Re}(x):=2^{-1}(x+\overline{x})=(1|x)$,
respectively.

\begin{lemma}\label{prl-01}
{\rm (cf. \cite{Fh1951},
\cite{Dd1978}, \cite{YiJ1971})}.
Let $x,y,z,a,b\in\tilde{\bf O}$.
\begin{align*}
\tag{\ref{prl-01}.a} &(xy|xy)=(x|x)(y|y).\\
\tag{\ref{prl-01}.b} &(ax|ay)=(a|a)(x|y) = (xa|ya).\\
\tag{\ref{prl-01}.d} &(ax|by) + (bx|ay) = 2(a|b)(x|y).\\
\tag{\ref{prl-01}.e} &(ax|y) = (x|\overline{a}y),~~
(xa|y) = (x|y\overline{a}).\\
\tag{\ref{prl-01}.f} &\overline{\overline{x}} = x,
~~ \overline{x + y} = \overline{x} + \overline{y},
~~ \overline{xy} = \overline{y}~\overline{x}.\\
\tag{\ref{prl-01}.g} &
\left\{\begin{array}{l}
(x|y) = (y|x) 
=2^{-1}(x\overline{y} + y\overline{x})
=2^{-1}(\overline{x}y+ \overline{y}x),\\
x\overline{x} = \overline{x}x = (x|x).
\end{array}\right.\\
\tag{\ref{prl-01}.h} &\left\{\begin{array}{l}
a(\overline{a}x) = (a\overline{a})x,
a(x\overline{a}) = (ax)\overline{a},
x(a\overline{a}) = (xa)\overline{a},\\
a(ax) = (aa)x, a(xa) = (ax)a, 
x(aa) = (xa)a.
\end{array}\right.\\
\tag{\ref{prl-01}.i} &\overline{b}(ax) + \overline{a}(bx) 
= 2(a|b)x = (xa)\overline{b} + (xb)\overline{a}.\\
\tag{\ref{prl-01}.j} 
&\left\{
\begin{array}{l}
(ax)y + x(ya) = a(xy) + (xy)a,\\
(xa)y + (xy)a = x(ay) + x(ya),\\
(ax)y + (xa)y = a(xy) + x(ay).
\end{array}\right.\\
\tag{\ref{prl-01}.k} 
&(ax)(ya) = a(xy)a \quad (\text{Moufang's formula}).\\
\tag{\ref{prl-01}.l} &{\rm Re}(xy) = {\rm Re}(yx),~~
{\rm Re}(x(yz)) = {\rm Re}(y(zx))
 = {\rm Re}(z(xy)).
\end{align*}
\end{lemma}
\medskip

Denote 
${\rm Im}\tilde{\bf O}:=\{x\in\tilde{\bf O}|~{\rm Re}(x)=0\}
=\{x\in\tilde{\bf O}|~\overline{x}=-x\}$.
The quadratic spaces
$({\bf O},{\rm n})$ and
$({\rm Im}{\bf O},{\rm n})$
are isomorphic to 
$(\mathbb{R}^{0,8},\mathrm{q}_{0,8})$ 
and $(\mathbb{R}^{0,7},\mathrm{q}_{0,7})$,
respectively.

Let ${\bf K}$ be a $\mathbb{F}$-subalgebra 
of $\tilde{\bf O}$ such that
${\bf K}$ have the unit $1$
and $\overline{x}\in {\bf K}$ for all $x\in {\bf K}$.
Let $M(n,{\bf K})$ be
the set of all $n\times n$ matrices with entries in ${\bf K}$.
For $A \in M(n, {\bf K})$ 
with the $(i, j)$-entry $a_{ij}\in{\bf K}$,
let $^t A \in M(n, {\bf K})$ be the transposed matrix  
having the $(i, j)$-entry $a_{j i}$, 
$\bar{A} \in M(n, {\bf K})$ the conjugate matrix  
having the $(i, j)$-entry $\overline{a}_{ij}$, 
and denote $A^*:=~^t \bar{A} \in M(n, {\bf K})$.
The matrix in $M(n, {\bf K})$ 
with $1$ at the $(i,j)$-th place and zeros elsewhere
is denoted by $E_{i,j}$,
and  
the diagonal matrix $\sum_{i=1}^n a_{ii}E_{i,i}$
by ${\rm diag}(a_{11},\cdots,a_{nn})$.
In particular,
denote 
$E:={\rm diag}(1,\cdots,1)$ 
and $I_p:=-\sum_{i=1}^pE_{i,i}
+\sum_{i=1}^qE_{p+i,p+i}
\in M(p+q,{\bf K})$.
The subalgebra
${\bf C}:=\{x_0+x_1e_1|~x_i\in\mathbb{R}\}$
of ${\bf O}$ is isomorphic with
the field of complex numbers.
We use the following notations 
about some of classical Lie groups:
${\rm O}(n)
:=\{A\in M(n,\mathbb{R})|~{}^tAA=E\}$,
${\rm SO}(n)
:=\{A\in M(n,\mathbb{R})
|~{}^tAA=E,~\mathrm{det}(A)=1\}$,
${\rm SU}(n):=\{A\in M(n,{\bf C})
|~
A^*A=E,~\mathrm{det}(A)=1\}$,
${\rm O}(p,q)
:=\{A\in M(n,\mathbb{R})
|~{}^tAI_pA=I_p\}$
where $\mathrm{det}(A)$ and $\mathrm{tr}(A)$ denote the
determinant of $A\in M(n,{\bf C})$
and
the trace of $A\in M(n,{\bf C})$, respectively.
For differential manifolds $X$ and $Y$,
$X\simeq Y$ denotes that $X$ and $Y$ 
are diffeomorphic.
Let $G$ be a Lie group. 
Its Lie algebra is denoted by $Lie(G)$
and the identity connected component of $G$
by $G^0$.
For Lie groups $G$ and $G'$,
$G\cong G'$ denotes that $G$ and $G'$ are isomorphic 
as Lie group.
Let $N$ be a normal subgroup of $G$ and
$H$ a subgroup of $G.$
If 
$G=HN$ and $N\cap H=\{1\}$
hold
where $1$ denotes the identity element of $G$,
then $G$ is called a {\it 
semidirect product of $N$ and $H$},
and is denoted by $H\ltimes N$.
\medskip

Let $i\in\{1,2,3\}$ and 
indexes $i,i+1,i+2$ be counted modulo $3$.
For $\xi=(\xi_1,\xi_2,\xi_3)\in \mathbb{C}^3$ and
$x=(x_1,x_2,x_3)\in ({\bf O^{\mathbb{C}}})^3$,
denote the hermitian matrix 
\[h(\xi;x):=\begin{pmatrix}
\xi_1 &x_3 & \overline{x_2}\\
\overline{x_3} & \xi_2 & x_1\\
x_2 &\overline{x_1} & \xi_3
\end{pmatrix}.
\]
The complex exceptional Jordan algebra $\mathcal{J}^\mathbb{C}$
is defined by
\[
\mathcal{J}^\mathbb{C}:=\{X\in M(3,{\bf O}^\mathbb{C})|
~X^*=X\}
=\{h(\xi;x)|~
\xi\in \mathbb{C}^3,
x\in ({\bf O}^{\mathbb{C}})^3\}
\]
with the Jordan product \[X\circ Y:=2^{-1}(XY+YX)
\quad
\text{for}~X,Y \in \mathcal{J}^\mathbb{C}.\]
Then $E$ is
the identity element
of the Jordan product.
Denote the elements 
$E_i,~F_i(x)\in \mathcal{J}^\mathbb{C}$ as
\[E_i:=E_{i,i},\quad F_i(x):=x E_{i+1,i+2}
+\overline{x}E_{i+2,i+1}\]
where $x\in{\bf O}^{\mathbb{C}}$.
Then \[h(\xi_1,\xi_2,\xi_3;x_1,x_2,x_3)
=\sum{}_{i=1}^3(\xi_iE_i+F_i(x_i)).\]
Let $X=\sum_{i=1}^3(\xi_iE_i+F_i(x_i)),
~Y\in \mathcal{J}^\mathbb{C}$.
The {\it trace} $\mathrm{tr}(X)$
is defined by
$\mathrm{tr}(X):=\xi_1+\xi_2+\xi_3$
and
the inner product $(X|Y)$
by $(X|Y):=\mathrm{tr}(X\circ Y)$.
Then the inner product $(X|Y)$
is a non-degenerate inner product.
The
{\it cross product} of H. Freudenthal is defined
by 
\[
X\times Y:=2^{-1}\bigl(2X\circ Y
-\mathrm{tr}(X)Y-\mathrm{tr}(Y)X
+(\mathrm{tr}(X)\mathrm{tr}(Y)-(X|Y))E\bigr)\]
\cite{Fh1953} (cf. 
\cite[p.232,(47)]{Jn1968},
\cite{Yi_arxiv}, \cite{SV2000})
with $X^{\times 2}:=X\times X$. The trilinear
from $(X|Y|Z)$ and
the {\it determinant}  
$\mathrm{det}(X)$ are defined by
\[(X|Y|Z):=(X|Y\times Z),\quad
\mathrm{det}(X):=3^{-1}(X|X|X)\]
respectively.
The {\it characteristic polynomial} 
$\Phi_X(\lambda)$ of $X\in \mathcal{J}^{\mathbb{C}}$ 
is defined by
$\Phi _X(\lambda):=\mathrm{det}(\lambda E-X)$
and a solution of 
$\Phi _{X}(\lambda)=0$ in $\mathbb{C}$
is called
a {\it characteristic root} of 
$X$.
From direct calculations, we have the following two lemmas.

\begin{lemma}\label{prl-02}
{\rm (cf. \cite{NY2010}).}
Let 
$i\in\{1,2,3\}$
and indexes $i,i+1,i+2$ be counted modulo $3$.
Let
$X=\sum_{i=1}^3(\xi_iE_i+F_i(x_i)),~
Y=\sum_{i=1}^3(\eta_iE_i+F_i(y_i))
\in \mathcal{J}^{\mathbb{C}}$.
Then
the following equations hold:
\begin{gather*}
\tag{\ref{prl-02}.a}
(X|Y)=\sum{}_{i=1}^3 
\bigl(\xi_i\eta_i+2(x_i|y_i)\bigr),\\
\tag{\ref{prl-02}.b}
X\times Y
=\sum{}_{i=1}^3\left(
2^{-1}(\xi_{i+1}\eta_{i+2}+\eta_{i+1}\xi_{i+2})
-(x_i|y_i)\right)E_i
\\
+\sum{}_{i=1}^3F_i\left(2^{-1}
(
\overline{x_{i+1}y_{i+2}}
+\overline{y_{i+1}x_{i+2}}
-\xi_iy_i-\eta_ix_i
)\right) ,\\
\tag{\ref{prl-02}.c}
\mathrm{det}(X)=\xi_1\xi_2\xi_3
+2{\rm Re}(x_1x_2x_3)
-\sum{}_{i=1}^3\xi _i(x_i|x_i).
\end{gather*}
\end{lemma}
\medskip

\begin{lemma}\label{prl-03}
{\rm (cf. \cite{NY2010}).}
Let $X,Y\in \mathcal{J}^{\mathbb{C}}$.
Then
\begin{gather*}
\tag{\ref{prl-03}.a}
\mathrm{tr}(X\times Y)
=2^{-1}(\mathrm{tr}(X)\mathrm{tr}(Y)-(X|Y)),\\
\tag{\ref{prl-03}.b}
\left\{
\begin{array}{rl}
{\rm (i)}&E\times E=E,\quad
{\rm (ii)}~~
X\times E=2^{-1}(\mathrm{tr}(X)E-X),
\\
\smallskip
{\rm (iii)}&
(X^{\times 2})\times E
=2^{-1}(\mathrm{tr}(X^{\times 2})E-X^{\times 2}),\\
\smallskip
{\rm (iv)}&
(X^{\times 2})^{\times 2}=\mathrm{det}(X)X,
\\
\smallskip
{\rm (v)}&
(X^{\times 2})\times X
=2^{-1}\left(-\mathrm{tr}(X)X^{\times 2}
-\mathrm{tr}(X^{\times 2})X \right.\\
&\quad\left. 
\quad+(\mathrm{tr}(X^{\times 2})\mathrm{tr}(X)
-\mathrm{det}(X))E\right),
\end{array}
\right.
\end{gather*}
\begin{align*}
\tag{\ref{prl-03}.c}
\Phi _X(\lambda)&=
\lambda^3-\mathrm{tr}(X)\lambda^2
+\mathrm{tr}(X^{\times 2})\lambda
-\mathrm{det}(X)\\
&=\lambda^3-\mathrm{tr}(X)\lambda^2
+2^{-1}(\mathrm{tr}(X)^2-(X|X))\lambda
-\mathrm{det}(X).
\end{align*}
\end{lemma}
\medskip

The linear Lie group $\mathrm{F}_4^{\mathbb{C}}$ is defined by
\[
\mathrm{F}_4^{\mathbb{C}}:=\{g\in 
\mathrm{GL}_{\mathbb{C}}(\mathcal{J}^{\mathbb{C}})
|~g(X\circ Y)=g X\circ g Y\}.\]

The following result is proved in \cite{Yi1990}, \cite{NY2010}
after O. Shukuzawa and I. Yokota \cite{SY1979a}.

\begin{proposition}\label{prl-04}
Let $X,Y,Z\in \mathcal{J}^{\mathbb{C}}$.

{\rm (1)} For all $g\in \mathrm{F}_4^{\mathbb{C}}$, 
\[
\tag{\ref{prl-04}.a}
\mathrm{tr}(g X)=\mathrm{tr}(X).
\]

{\rm (2)} The following equations hold.
\begin{align*}
\tag{\ref{prl-04}.b}
\mathrm{F}_4^{\mathbb{C}}
&=\{g\in \mathrm{GL}_{\mathbb{C}}(\mathcal{J}^{\mathbb{C}})|
~\mathrm{det} (g X)=X,~g E=E\}\\
&=\{g\in \mathrm{GL}_{\mathbb{C}}(\mathcal{J}^{\mathbb{C}})~|
~\Phi _{g X}(\lambda)=\Phi _X(\lambda)\}\\
&=\{g\in \mathrm{GL}_{\mathbb{C}}(\mathcal{J}^{\mathbb{C}})|~
\mathrm{det} (g X)=X,~(g X|g Y)=(X|Y)\}\\
&=\left\{g\in 
\mathrm{GL}_{\mathbb{C}}(\mathcal{J}^{\mathbb{C}})
\biggm|
~\begin{array}{rcl}(g X| g Y| g Z)&=&(X|Y|Z),\\
(g X|g Y)&=&(X|Y)
\end{array}
\right\}\\
&=\{g\in \mathrm{GL}_{\mathbb{C}}(\mathcal{J}^{\mathbb{C}})|
~g (X\times Y)=g X\times g Y\}.
\end{align*}
\end{proposition}
\medskip

The $\mathbb{R}$-exceptional 
Jordan algebra $\mathcal{J}$ is defined by
\[\mathcal{J}:=\{X\in M(3,{\bf O})|~X^*=X\}
=\{h(\xi;x)|~
\xi\in \mathbb{R}^3,~
x\in {\bf O}^3\}
\]
with the Jordan product $X\circ Y=2^{-1}(XY+YX)$.
Denote the complex conjugation $\tau$ with respect to 
$\mathcal{J}$ in $\mathcal{J}^{\mathbb{C}}$.
We define 
$\sigma\in\mathrm{GL}_{\mathbb{C}}
(\mathcal{J}^{\mathbb{C}})$ by
\[\sigma h(\xi_1,\xi_2,\xi_3;x_1,x_2,x_3)
:=h(\xi_1,\xi_2,\xi_3;x_1,-x_2,-x_3).\]
Because of $\mathrm{det} (\sigma X)=X$ and $\sigma E=E$,
applying (\ref{prl-04}.b),
we see
$\sigma \in \mathrm{F}_4^{\mathbb{C}}$ 
and clearly $\sigma^2=1$
where $1$ denotes the identity element 
of $\mathrm{F}_4^{\mathbb{C}}$.
We consider the complex conjugation $\tau\sigma$ 
in $\mathcal{J}^{\mathbb{C}}$ and define
the involutive automorphism 
$\widetilde{\tau\sigma}$ of $\mathrm{F}_4^{\mathbb{C}}$ by 
\[\widetilde{\tau\sigma}(g)=\tau\sigma g\sigma\tau\quad 
\text{for}~g\in \mathrm{F}_4^{\mathbb{C}}.\]
For $\xi=(\xi_1,\xi_2,\xi_3)\in \mathbb{R}^3$ and
$x=(x_1,x_2,x_3)\in {\bf O}^3,$
denote 
\[
h^1(\xi;x):=\begin{pmatrix}
\xi_1 &\sqrt{-1}x_3 & \sqrt{-1}\overline{x_2}\\
\sqrt{-1}\overline{x_3} & \xi_2 & x_1\\
\sqrt{-1}x_2 &\overline{x_1} & \xi_3
\end{pmatrix}
.\]
The {\it exceptional Jordan algebra} 
$\mathcal{J}^1$ is defined by
\[\mathcal{J}^1:=(\mathcal{J}^{\mathbb{C}})_{\tau\sigma}
=\{h^1(\xi;x)|~
\xi\in\mathbb{R}^3, x\in{\bf O}^3\}\]
with the Jordan product $X\circ Y=2^{-1}(XY+YX)$.
Then the Jordan algebra $\mathcal{J}^1$ has 
the trace $\mathrm{tr}(X)\in\mathbb{R}$,  
the identity element $E$ of Jordan product,
the inner product $(X|Y)\in\mathbb{R}$,
the cross product $X\times Y\in\mathcal{J}^1$,
the trilinear 
from $(X|Y|Z)\in\mathbb{R}$,
the determinant 
$\mathrm{det}(X)\in\mathbb{R}$
and
the characteristic polynomial $\Phi_X(\lambda)$
is a polynomial with $\mathbb{R}$-coefficients.
Let $i\in\{1,2,3\}$ and $x\in{\bf O}$. Denote the elements
\[
F^1_1(x):=F_1(x),\quad F^1_j(x):=F_j(\sqrt{-1}x)
\quad\text{with}
~j\in\{2,3\}.
\] 
Then we see
\[h^1(\xi_1,\xi_2,\xi_3;x_1,x_2,x_3)
=\sum{}_{i=1}^3 (\xi_iE_i+ F^1_i(x_i))\]
and for $X=\sum_{i=1}^3 (\xi_iE_i+ F^1_i(x_i))\in \mathcal{J}^1$,
denote
\[(X)_{E_i}:=\xi_i=(X|E_i),\quad(X)_{F^1_i}:=x_i\]
respectively.
Moreover, denote the elements
\begin{align*}
 P^+&:=h^1(1,-1,0;0,0,1),
 &P^-&:=h^1(-1,1,0;0,0,1),\\
 Q^+(x)&:=h^1(0,0,0;x,\overline{x},0),
 &Q^-(x)&:=h^1(0,0,0;x,-\overline{x},0)
\end{align*}
and the subspaces 
$F_i^1({\bf O})
:=\{F_i^1(x)|~x\in{\bf O}\},$
$F_i^1({\rm Im}{\bf O})
:=\{F_i^1(p)|~p\in{\rm Im}{\bf O}\},$
$Q^+({\bf O}):=\{Q^+(x)|~x\in{\bf O}\}$,
$Q^-({\bf O}):=\{Q^-(x)|~x\in{\bf O}\}$,
respectively.
Easily we have the following decompositions
of $\mathcal{J}^1$. 

\begin{lemma}\label{prl-05}
\begin{align*}
\tag{\ref{prl-05}.a}
\mathcal{J}^1
&=
\mathbb{R}E_1\oplus
\mathbb{R}E_2\oplus
\mathbb{R}E_3\oplus
F_1^1({\bf O})\oplus
F_2^1({\bf O})\oplus
F_3^1({\bf O}),\\
\tag{\ref{prl-05}.b}
\mathcal{J}^1&=\mathbb{R}(-E_1+E_2)
\oplus \mathbb{R}P^-\oplus \mathbb{R}E\oplus
\mathbb{R}E_3\oplus F_3^1({\rm Im}{\bf O})\\
&\quad
\oplus Q^+({\bf O})\oplus Q^-({\bf O}).
\end{align*}
\end{lemma}
\medskip

We use the notation $\epsilon_i(j)$
in this article.
If $i=j$ then $\epsilon_i(j):=1$ 
else $\epsilon_i(j):=-1$.
Then $\epsilon_i(j)=(-1)^{1+\delta_{i,j}}$
where $\delta_{i,j}$ is the Kronecker delta
and we write the notation $\epsilon(i)$ instead of $\epsilon_1(i)$
for short.
Using Lemma~\ref{prl-02}, the following lemma  follows from
direct calculations.

\begin{lemma}\label{prl-06}
Let $x,y\in{\bf O}$,
$i,j \in \{1,2,3\}$ and indexes
$i+1,i+2\in\{1,2,3\}$ be counted modulo $3$.
Let
$X=\sum_{i=1}^3(\xi_iE_i+F_i^1(x_i)),
~Y=\sum_{i=1}^3(\eta_iE_i+F_i^1(y_i))\in \mathcal{J}^1$.
Then the following equations hold:
\begin{gather*}
\tag{\ref{prl-06}.a}
\left\{
\begin{array}{rlrl}
{\rm (i)}& E_i \times E_i=0,&
{\rm (ii)}&E_i \times E_{i+1}=2^{-1}E_{i+2},\\
{\rm (iii)}&
E_i \times F^1_i(x)=F^1_i\left(-2^{-1}x\right),&
{\rm (iv)}&E_i \times F^1_j(x)=0,~i\ne j,\\
\smallskip
{\rm (v)}&\multicolumn{3}{l}{F^1_i(x)\times F^1_i(y)
=-\epsilon(i)(x|y)E_i,}\\
\smallskip
{\rm (vi)}&
\multicolumn{3}{l}{F^1_{i+1}(x)\times F^1_{i+2}(y)=
F^1_i\left(-\epsilon(i)2^{-1}\overline{xy}\right),}
\end{array}
\right.
\\
\tag{\ref{prl-06}.b}
(X|Y)=\sum{}_{i=1}^3 
(\xi_i\eta_i+\epsilon(i)2(x_i|y_i)),\\
\tag{\ref{prl-06}.c}
\mathrm{det}(X)=\xi_1\xi_2\xi_3
+2{\rm Re}(x_1x_2x_3)
-\sum{}_{i=1}^3\epsilon(i)\xi _i(x_i|x_i),\\
\tag{\ref{prl-06}.d}
X^{\times 2}=\sum{}_{i=1}^3 ((\xi_{i+1}\xi_{i+2}
-\epsilon(i)(x_i|x_i))E_i
+F^1_i(-\epsilon(i)
\overline{x_{i+1}x_{i+2}}-\xi_ix_i)).
\end{gather*}
\end{lemma}
\medskip

For all $X\in \mathcal{J}^1,$ the {\it minimal subspace} $V_X$ 
of $X$
is defined by 
\[
V_X:=\{aX^{\times 2}+bX+cE~|~a,b,c\in\mathbb{R}\}.\] 

\begin{lemma}\label{prl-07}
For all $X\in\mathcal{J}^1$,
the minimal space $V_X$ is closed under the cross product.
\end{lemma}
\begin{proof}
It follows from (\ref{prl-03}.b).
\end{proof}
\medskip

The linear Lie group $\mathrm{F}_{4(-20)}$ is defined by
\[\mathrm{F}_{4(-20)}:=
\{g\in \mathrm{GL}_{\mathbb{R}}(\mathcal{J}^1)|
~g (X\circ Y)=g X\circ g Y\}.
\]
Because of $\mathcal{J}^1
=(\mathcal{J}^{\mathbb{C}})_{\tau\sigma}$,
we can write
$\mathrm{F}_{4(-20)}=
(\mathrm{F}_4^{\mathbb{C}})^{\widetilde{\tau\sigma}}$.
In \cite[Theorem 2.2.2]{Yi1990}, I.~Yokota  shows that
$\mathrm{F}_{4(-20)}$ is an exceptional linear Lie group
of type {\bf F}$_{4(-20)}$.
From
Proposition~\ref{prl-04},
we have the following proposition. 

\begin{proposition}\label{prl-08}
Let $X,Y,Z\in \mathcal{J}^1$.

{\rm (1)} For all $g\in \mathrm{F}_{4(-20)}$, 
\[\tag{\ref{prl-08}.a}
\mathrm{tr}(g X)=\mathrm{tr}(X).\]

{\rm (2)} The following equations hold.
\begin{align*}
\tag{\ref{prl-08}.b}
\mathrm{F}_{4(-20)}
&=\{g\in \mathrm{GL}_{\mathbb{R}}(\mathcal{J}^1)|
~\mathrm{det} (g X)=X,~g E=E\}\\
&=\{g\in \mathrm{GL}_{\mathbb{R}}(\mathcal{J}^1)~|
~\Phi _{g X}(\lambda)=\Phi _X(\lambda)\}\\
&=\{g\in \mathrm{GL}_{\mathbb{R}}(\mathcal{J}^1)|~
\mathrm{det} (g X)=X,~(g X|g Y)=(X|Y)\}\\
&=\left\{g\in \mathrm{GL}_{\mathbb{R}}(\mathcal{J}^1)
\biggm|
~\begin{array}{rcl}(g X| g Y| g Z)&=&(X|Y|Z),\\
(g X|g Y)&=&(X|Y)
\end{array}
\right\}\\
&=\{g\in \mathrm{GL}_{\mathbb{R}}(\mathcal{J}^1)|
~g (X\times Y)=g X\times g Y\}.
\end{align*}
\end{proposition}
\medskip

By Proposition~\ref{prl-08}, 
the identity element $E$,
the trace,
the inner product,
the determinant,
the trilinear form,    
the cross product,
the characteristic polynomial
and the set of all characteristic roots 
with multiplicities
are invariant under the action of $\mathrm{F}_{4(-20)}$
and we use this fact without notice.

Let $X\in\mathcal{J}^1$ 
and $\lambda_0\in\mathbb{R}$. 
The elements $p(X),~E_{X,\lambda_0},
W_{X,\lambda_0}\in V_X$ 
(see Lemma~\ref{prl-07})
are defined as
\begin{align*}
p(X)&:=X-3^{-1}\mathrm{tr}(X)E,\\
E_{X,\lambda_0}&:=
{\mathrm{tr}((\lambda_0 E-X)^{\times 2})}^{-1}
(\lambda_0 E-X)^{\times 2},\\
W_{X,\lambda_0}&:=
X-\left(\lambda_0E_{X,\lambda_0}
+2^{-1}(\mathrm{tr}(X)-\lambda_0)
(E-E_{X,\lambda_0})\right)
\end{align*}
respectively.
Immediately, we have the following lemma.

\begin{lemma}\label{prl-09}
Let $X \in \mathcal{J}^1$.
If 
$\mathrm{tr}((\lambda_0 E-X)^{\times 2})\ne 0$
then
$E_{X,\lambda_0}$ and  $W_{X,\lambda_0}$
are well-defined 
and the following equation holds.
\[
\tag{\ref{prl-09}}
X=\lambda_0E_{X,\lambda_0}+2^{-1}(\mathrm{tr}(X)-\lambda_0)
(E-E_{X,\lambda_0})+W_{X,\lambda_0}.\]
\end{lemma}
\medskip

\begin{lemma}\label{prl-10} 
Let $X\in \mathcal{J}^1$.
For all $g\in \mathrm{F}_{4(-20)},$ 
\[
\tag{\ref{prl-10}}
\left\{
\begin{array}{rlrl}
{\rm (i)}&g(V_X)=V_{g X},&
{\rm (ii)}&gp(X)=p(gX),\\
\smallskip
{\rm (iii)}&g E_{X,\lambda_1}=E_{g X,\lambda_1},&
{\rm (iv)}&g W_{X,\lambda_1}=W_{g X,\lambda_1}.
\end{array}\right.
\]
\end{lemma}
\begin{proof}
It follows from
Proposition~\ref{prl-08}.
\end{proof}

\section{The principle of triality.}\label{trl}
In this section, we explain the groups $\mathrm{Spin}(8)$
and $\mathrm{Spin}(7)$
by means of the triality.
In the next section,
we explain that
these groups are  
isomorphic to some stabilizers in $\mathrm{F}_{4(-20)}$,
respectively.

We write
$S^0$ for the subset $\{{\rm diag}(1,1,\cdots,1),
{\rm diag}(-1,1,\cdots,1)\}
\cong \mathbb{Z}_2$  of $M(n,\mathbb{R})$
and ${\rm SO}(0)$ for the group $\{1\}$,
respectively.

\begin{lemma}\label{trl-01}
{\rm (cf. \cite{YiJ1971,YiJ1990})}.
Let $n$ be a natural number, 
and $p,q$ non-negative integers with $p+q>0$.

{\rm (1)} {\rm (cf. \cite[Theorem 20(2)]{YiJ1971}).}
${\rm O}(n)=S^0\ltimes {\rm SO}(n).$
Especially, ${\rm O}^0(n)={\rm SO}(n).$
\smallskip

{\rm (2)} {\rm (cf. \cite[Theorem 6.12(2)]{YiJ1990}).}
${\rm O}^0(p,q)\simeq 
({\rm SO}(p)\times {\rm SO}(q))\times \mathbb{R}^{pq}.$
\smallskip

{\rm (3)} For all $n\geq 3,$ the fundamental group  
$\pi_1({\rm O}^0(n,1))
=\mathbb{Z}_2$ 
\end{lemma}
\begin{proof} (3) We note $\pi_1({\rm SO}(n))
=\mathbb{Z}_2$ $(n\geq 3)$ 
(cf. \cite[Theorem 59(2)]{YiJ1971}).
By (2), $\pi_1({\rm O}^0(n,1))
=\pi_1({\rm SO}(n))\times \pi_1({\rm SO}(0))
\times \pi_1(\mathbb{R}^{pq})=\mathbb{Z}_2$.
\end{proof}
\medskip

From now on, the groups ${\rm O}(8),$ ${\rm SO}(8),$ 
${\rm O}(7)$ and ${\rm SO}(7)$ 
are identified with the groups:
${\rm O}(8)=
\{g\in \mathrm{GL}_{\mathbb{R}}({\bf O})|
~(g x|g y)=(x|y)\}$,
${\rm SO}(8)=\{g\in \mathrm{GL}_{\mathbb{R}}({\bf O})|~
(g x|g y)=(x|y),~\mathrm{det}(g)=1\}$,
${\rm O}(7)=\{g\in 
{\rm O}(8)|
~g 1=1\}$,
${\rm SO}(7)=\{g
\in {\rm SO}(8)|
~g 1=1\}$,
respectively.
The element $\epsilon \in{\rm O}(8)$ is defined by
\[
\epsilon x:=\overline{x}\quad{\rm for}~x\in{\bf O}.\]
Then $\epsilon^2=1$ and its determinant is $-1:$ 
$\mathrm{det}(\epsilon)=-1.$ 
The involutive automorphism $t$ of the group 
${\rm SO}(8)^3$ is defined by
\[t(g_1,g_2,g_3):=
(g_1,g_2,\epsilon g_3\epsilon)\quad{\rm for}~
(g_1,g_2,g_3)\in {\rm SO}(8)^3.\]
The subgroup $T({\bf O})$ of
${\rm SO}(8)^3$ is defined by
\[
T({\bf O}):=
\{(g_1,g_2,g_3)\in {\rm SO}(8)^3
|~(g_1 x)(g_2 y)
=g_3(xy)~{\rm for~all}~x,y\in{\bf O}
\}\]
(cf. \cite[(2.4.6)]{Fh1951}, 
\cite{My1954}, \cite{SV2000}, \cite{Yi_arxiv})
and the subgroup $\tilde{D}_4$ of
${\rm SO}(8)^3$ by
\begin{align*}
\tilde{D}_4&:=t^{-1}(T({\bf O}))
=\{(g_1,g_2,g_3)\in {\rm SO}(8)^3~|~
t(g_1,g_2,g_3)\in T({\bf O})\}\\
&=\{(g_1,g_2,g_3)\in {\rm SO}(8)^3
|~(g_1 x)(g_2 y)
=\epsilon g_3\epsilon(xy)~{\rm for~all}~x,y\in{\bf O}
\}.
\end{align*}
The equation
$(g_1 x)(g_2 y)
=g_3(xy)$ or
$(g_1 x)(g_2 y)
=\epsilon g_3\epsilon (xy)$
is called the {\it triality}.
The following result is proved in \cite{My1954}
(cf. \cite[Lemma~1.14.3]{Yi_arxiv}).

\begin{lemma}\label{trl-02}
Let 
$i \in \{1,2,3\}$ and
indexes $i,i+1,i+2$ be counted modulo $3.$
Assume that
$(g_1,g_2,g_3)\in {\rm O}(8)^3$ 
satisfies
$(g_i x)(g_{i+1} y)
=\epsilon g_{i+2}\epsilon(xy)$
for all $x,y\in{\bf O}$.
Then
$(g_{i+1} x)(g_{i+2} y)
=\epsilon g_i\epsilon(xy)$
for all $x,y\in{\bf O}$.
Especially,
\[
\tag{\ref{trl-02}}
(g_1,g_2,g_3)\in \tilde{D}_4\Leftrightarrow
(g_2,g_3,g_1)\in \tilde{D}_4
\Leftrightarrow
(g_3,g_1,g_2)\in \tilde{D}_4.
\]
\end{lemma}
\medskip

For $i\in\{1,2,3\},$ the homomorphism 
$p_i: \tilde{D}_4\rightarrow {\rm SO}(8)$ 
is defined by
\[p_i(g_1,g_2,g_3):=g_i\quad\quad{\rm for}~
(g_1,g_2,g_3)\in \tilde{D}_4.\]

\begin{lemma}\label{trl-03} 
Let $x,y\in{\bf O}.$

{\rm (1)} Let $g\in{\rm SO}(7).$ 
Then
\[
\tag{\ref{trl-03}.a}
{\rm (i)}~~~\overline{g x}
=g\overline{x},\quad
{\rm (ii)}~~~\epsilon g\epsilon=g,\quad
{\rm (iii)}~~~g({\rm Im}(x))={\rm Im}(g x).
\]
Especially, $g(\mathrm{Im}{\bf O})\subset \mathrm{Im}{\bf O}$.
\smallskip

{\rm (2)} Let 
$(g_1,g_2,g_3)\in \tilde{D}_4$,
$i \in \{1,2,3\}$ and indexes $i,i+1,i+2$
be counted modulo $3$.
Then \[
\tag{\ref{trl-03}.b}
g_i 1=1~~
\Leftrightarrow~~g_{i+1}=\epsilon g_{i+2}\epsilon
~~
\Leftrightarrow~~g_{i+2}=\epsilon g_{i+1}\epsilon.\]

{\rm (3)} Assume that $(g_1,g_2,g_3)
\in \tilde{D}_4$
and $g_3 1=1.$
Then
\[
\tag{\ref{trl-03}.c}
\left\{
\begin{array}{rlrl}
{\rm (i)}&
g_3(x\overline{y})
=(g_1x)(\overline{g_1 y}),&
{\rm (ii)}&
g_3({\rm Im}(x\overline{y}))
={\rm Im}((g_1x)(\overline{g_1 y})),\\
\smallskip
{\rm (iii)}&
g_1(xy)=(g_3x)(g_1y).
\end{array}\right.
\]
\end{lemma}
\begin{proof}
(1) 
Because of $g1=1$,
$\overline{x}=2(1|x)-x$
and ${\rm Im}(x)=x-(1|x)$,
we see
$g\overline{x}
=2(1|g x)-g x=\overline{g x}$.
and 
$g({\rm Im}(x))
=g(x-(1|x))=g x-(1|g x)
={\rm Im}(g x)$.
Thus (i) and (iii) follow.
From (i), 
$\epsilon g\epsilon x
=g(\overline{\overline{x}})=g x$.

(2) Obviously $g_{i+1}=\epsilon g_{i+2}\epsilon$
iff $g_{i+2}=\epsilon g_{i+1}\epsilon.$
We show
$g_i 1=1$ iff
$g_{i+1}=\epsilon g_{i+2}\epsilon.$
By (\ref{trl-02}),
$(g_1,g_2,g_3)\in \tilde{D}_4$ iff
$(g_i,g_{i+1},g_{i+2})\in \tilde{D}_4$
so that
\[\tag{i}(g_i x)(g_{i+1}y)
=\epsilon g_{i+2}\epsilon(xy)\quad
\text{for~all}~x,y\in{\bf O}.\]
Suppose $g_i1=1$. 
Substituting $x=1$ in (i),
we obtain $g_{i+1}=\epsilon g_{i+2}\epsilon$.
Conversely, 
suppose $g_{i+1}=\epsilon g_{i+2}\epsilon$. 
Substituting $x=y=1$ in (i), 
$(g_i 1)(\epsilon g_{i+2}\epsilon 1)
=\epsilon g_{i+2}\epsilon 1.$
Multiplying $(\epsilon g_{i+2} \epsilon 1)^{-1}$ 
from right,
$g_i 1=1$.

(3)
By (\ref{trl-03}.b),
$g_2=\epsilon g_1\epsilon$
and $g_1=\epsilon g_2\epsilon$. 
First, because of
$g_3
=\epsilon g_3\epsilon$ (by (\ref{trl-03}.a)(ii))
and
$(g_1,g_2,g_3)=(g_1,\epsilon g_1\epsilon,g_3)
\in \tilde{D}_4$, we see that
$g_3(x\overline{y})
=\epsilon g_3\epsilon(x\overline{y})
=(g_1x)(\epsilon g_1\epsilon\overline{y})
=(g_1x)(\overline{g_1y})$.
Second, by (\ref{trl-03}.a)(iii) 
and (\ref{trl-03}.c)(i),
we see that 
$g_3({\rm Im}(x\overline{y}))={\rm Im}(g_3(x\overline{y}))
={\rm Im}((g_1x)(\overline{g_1 y}))$.
Last, because of
$g_1=\epsilon g_2\epsilon$ and
$(g_3,g_1,g_2)
\in \tilde{D}_4$ (by (\ref{trl-02})),
we obtain 
$g_1(xy)=\epsilon g_2 \epsilon (xy)
=(g_3x)(g_1y)$.
\end{proof}
\medskip

The subgroup $\tilde{B}_3$ 
of $\tilde{D}_4$ is defined by 
\begin{align*}
\tilde{B}_3&:=\{(g_1,g_2,g_3)\in\tilde{D}_4
|~g_3 1=1\}\\
&=\{(g_1,g_2,g_3)\in\tilde{D}_4
|~g_2=\epsilon g_1\epsilon\}
=\{(g_1,g_2,g_3)\in\tilde{D}_4
~|~g_1=\epsilon g_2\epsilon\}
\end{align*}
and the homomorphism 
$q: \tilde{B}_3\rightarrow {\rm SO}(7)$ 
by $q:=p_3|\tilde{B}_3$: 
$q(g_1,g_2,g_3)=g_3.$
The linear Lie group $\mathrm{G}_2$
is defined by
\[\mathrm{G}_2:={\rm Aut}({\bf O})
=\{g\in\mathrm{GL}_{\mathbb{R}}({\bf O})~|
~(g x)(g y)=g(xy)\}.\]
Then $\mathrm{G}_2$ is a subgroup of ${\rm SO}(7)$
(cf. \cite[Lemma~1.2.1, Theorem~1.9.3]{Yi_arxiv}).
In particular, $(gx|gy)=(x|y)$ and $g 1=1$
for all $g\in \mathrm{G}_2$.
Also, for any $g\in \mathrm{G}_2,$ considering 
$(g,g,g)
\in{\rm SO}(8)^3,$ 
$\mathrm{G}_2$ is a subgroup of $\tilde{B}_3.$
Now $S^7=\{a\in{\bf O}|~{\rm n}(a)=1\}$ 
and $S^6=\{a\in{\rm Im}{\bf O}|~{\rm n}(a)=1\}.$
For all $a\in S^7,$ 
the elements
$L_a,R_a,T_a\in \mathrm{GL}_{\mathbb{R}}({\bf O})$
are defined
by
\[
L_a x:=ax,~R_a x:=xa,~T_a x:=axa\quad
{\rm for}~x\in{\bf O}\]
respectively.
By (\ref{prl-01}.b),
$L_a,R_a,T_a\in {\rm O}(8)$.
Because $S^7$ is connected
and for the unite $1\in{\bf O}$,
$L_1=R_1=T_1=1_{\bf O}$
where $1_{\bf O}$ denotes
the identity element of $\mathrm{O}(8)$,
we obtain
$L_a,R_a,T_a\in {\rm O}^0(8)={\rm SO}(8)$
(see Lemma~\ref{trl-01}(1)).
For any $a_i\in S^7,$ 
denote the elements $L_{a_n,\cdots,a_1}$,
$R_{a_n,\cdots,a_1}$,
$T_{a_n,\cdots,a_1} \in {\rm SO}(8)$
as
\[
\left\{
\begin{array}{c}
L_{a_n,\cdots,a_1}
:=L_{a_n}\cdots L_{a_1},\quad
R_{a_n,\cdots,a_1}:=R_{a_n}\cdots R_{a_1},\\
T_{a_n,\cdots,a_1}:=T_{a_n}\cdots T_{a_1}.
\end{array}
\right.\]
\begin{lemma}\label{trl-04}
{\rm (1)} {\rm (cf. \cite[Theorem 1.9.1,Theorem 
1.9.2]{Yi_arxiv})}.
\begin{gather*}
\tag{\ref{trl-04}.a}S^6=Orb_{\mathrm{G}_2}(e_1),\\
(\mathrm{G}_2)_{e_1}\cong {\rm SU}(3),\quad
\mathrm{G}_2/{\rm SU}(3)\simeq S^6.
\end{gather*}
Furthermore, $\mathrm{G}_2$ is connected.
\smallskip

{\rm (2)} If $a_i\in S^7,$
then 
$(L_{a_n,\cdots ,a_1},R_{a_n,\cdots,a_1},
\epsilon{T_{a_n,\cdots,a_1}}\epsilon)\in 
\tilde{D}_4$.
\smallskip

{\rm (3)} 
If $a,b\in S^6,$  
then
$(L_{b,a},R_{b,a},T_{b,a})
\in\tilde{B}_3$.

{\rm (4)} Let $j\in\{1,2\}$.
\[
\tag{\ref{trl-04}.b}
\mathrm{G}_2=\{(g_1,g_2,g_3)
\in \tilde{B}_3
|~g_j 1=1\}.
\]
\end{lemma}

\begin{proof}
(2) Let $a\in S^7.$ 
By (\ref{prl-01}.k),
\[(L_a x)(R_a y)=T_a(xy)
=\epsilon(\epsilon T_a\epsilon)\epsilon(xy).\]
Because of $L_a,~R_a,~
\epsilon T_a\epsilon \in{\rm SO}(8),$
we see
$(L_a,R_a,\epsilon T_a\epsilon)\in \tilde{D}_4$.
Next, from the induction, it follows that
\[(L_{a_n,\cdots,a_1} x)(R_{a_n,\cdots,a_1} y)
=\epsilon(\epsilon T_{a_n,\cdots,a_1}
\epsilon)\epsilon(xy).\]
Hence
(2) follows.

(3) From (2), 
$(L_{b,a},R_{b,a},\epsilon 
T_{b,a}\epsilon)\in\tilde{D}_4.$
Because of $a^2=b^2=-1$, 
we see
$\epsilon T_{b,a}\epsilon 1
=\overline{b(a(\overline{1})a)b}=1$.
Then $\epsilon T_{b,a}\epsilon 
\in \mathrm{SO}(7)$
and by (\ref{trl-03}.a)(ii),
$\epsilon T_{b,a}\epsilon=T_{b,a}$.
Hence (3) follows.

(4) 
Suppose that $(g_1,g_2,g_3)
\in \tilde{B}_3$. Because of
$g_2=\epsilon g_1 \epsilon$,
we see that $g_1 1=1$ iff $g_2 1=1$.
Thus it is enough to show
$\mathrm{G}_2=\{(g_1,g_2,g_3)
\in \tilde{B}_3
|~g_1 1=1\}$.
Put $G=\{(g_1,g_2,g_3)
\in \tilde{B}_3
|~g_1 1=1\}.$
Immediately, $\mathrm{G}_2\subset G$. 
Conversely, take $g=(g_1,g_2,g_3)\in G$.
Because of $g \in G$, 
we see
$g_1 1=g_2 1=g_3 1=1$.
Because of $g_i 1=g_{i+1} 1=1$  
and (\ref{trl-03}.b), we see
$g_{i+1}=\epsilon g_{i+2}\epsilon$
and $g_{i+2}=\epsilon g_{i}\epsilon$.
Then
$g_{i+1}=\epsilon g_{i+2}\epsilon
=\epsilon(\epsilon g_{i}\epsilon)\epsilon=g_{i}$.
Moving $i\in \{1,2,3\}$,
$g_1=g_2=g_3$.
Thus  $(g_1,g_2,g_3)\in \mathrm{G}_2$
and so $G\subset \mathrm{G}_2.$
Hence $G=\mathrm{G}_2$.
\end{proof}
\medskip

\begin{lemma}\label{trl-05}
{\rm (1)} $\tilde{B}_3/\mathrm{G}_2\simeq S^7.$
Furthermore, $\tilde{B}_3$ is connected.
\smallskip

{\rm (2)} $\tilde{D}_4/\tilde{B}_3\simeq S^7.$
Furthermore, $\tilde{D}_4$ is connected.
\end{lemma}
\begin{proof}
(1)
We consider the action of $\tilde{B}_3$ on $S^7$ 
as
$x\mapsto p_1(g_1,g_2,g_3)x
=g_1 x$
for
$x\in S^7$ and
$(g_1,g_2,g_3)\in\tilde{B}_3.$
Fix $x\in S^7$.
Then $x$ can be expressed by
\[x=\cos \theta+a \sin \theta
\quad\text{for~some}~ 
a\in S^6~
\text{and}~\theta\in\mathbb{R}.\]
First, 
by (\ref{trl-04}.a),
there exists $g_1\in \mathrm{G}_2$
such that $g_1a=e_1$. Obviously $g_1 1=1$
and we set $h_1=(g_1,g_1,g_1)\in \mathrm{G}_2\subset
\mathrm{Spin}(7)$.
Then
\[p_1(h_1)x=\cos \theta +e_1\sin \theta.\]
Second, put
$h_{e_1,e_2}=(L_{e_1,e_2},R_{e_1,e_2},T_{e_1,e_2})$.
Because of $e_i\in S^6$ and Lemma~\ref{trl-04}(3),
we see
$h_{e_1,e_2}\in\tilde{B}_3$
and 
\[p_1(h_{e_1,e_2})p_1(h_1) x=
e_1(e_2(\cos \theta +e_1\sin \theta))
=e_3\cos\theta+e_2\sin\theta.\]
Third,
because of $e_3\cos\theta+e_2\sin\theta\in S^6$,
there exists $g_2\in \mathrm{G}_2$
such that $g_2(e_3\cos\theta+e_2\sin\theta)=e_1$.
Then letting $h_2=(g_2,g_2,g_2)\in \mathrm{G}_2
\subset \mathrm{Spin}(7)$,
\[p_1(h_2)p_1
(h_{e_1,e_2})p_1(h_1) x=e_1.\]
Last, letting $h_{e_3,e_2}
=(L_{e_3,e_2},R_{e_3,e_2},T_{e_3,e_2})
\in\tilde{B}_3$,
\[p_1(h_{e_3,e_2})p_1(h_2)
p_1(h_{e_1,e_2})p_1(h_1) x
=e_3(e_2e_1)
=1.\]
Hence $\tilde{B}_3$ acts transitively on $S^7$.
By (\ref{trl-04}.b), $(\tilde{B}_3)_1=\mathrm{G}_2$.
Thus $\tilde{B}_3/\mathrm{G}_2\simeq S^7$ follows.
Since $\mathrm{G}_2$ and $S^7$ are connected,
$\tilde{B}_3$ is also connected.

(2) 
We consider the action of $\tilde{D}_4$ on $S^7$ 
as
$x\mapsto p_3(g_1,g_2,g_3)x=g_3 x$
for $x\in S^7$
and $(g_1,g_2,g_3)\in\tilde{D}_4.$
Let  $x\in S^7.$
Because of $(\overline{x}|\overline{x})=1$,
$\overline{x}\in S^7$.
By Lemma~\ref{trl-04}(2),
$(L_{\overline{x}},R_{\overline{x}},
\epsilon T_{\overline{x}}\epsilon)
\in \tilde{D}_4$.
Then  it follows
from (\ref{trl-02})  that
$(R_{\overline{x}},
\epsilon T_{\overline{x}}\epsilon,L_{\overline{x}})
\in \tilde{D}_4$ and
\[p_3(R_{\overline{x}},
\epsilon T_{\overline{x}}\epsilon,L_{\overline{x}})x
=\overline{x}x=1.\]
Thus $\tilde{D}_4$  acts transitively on $S^7$.
Because of $(\tilde{D})_1=\tilde{B}_3$,
we see $\tilde{D}_4/\tilde{B}_3\simeq S^7$.
Since $\tilde{B}_3$  
and $S^7$ are connected,  
$\tilde{D}_4$ is also connected.
\end{proof}
\medskip

The following result is shown in \cite{Fh1951}
(cf. \cite[Theorems 1.16.2, 1.15.2]{Yi_arxiv}).

\begin{proposition}\label{trl-07}
{\rm (1)}
The following sequence is exact:
\[
\tag{\ref{trl-07}.a}
1\to 
\{(1,1,1),(\epsilon_i(1),\epsilon_i(2),\epsilon_i(3))\} 
\to \tilde{D}_4\xrightarrow{p_i} {\rm SO}(8)\to 1.\]

{\rm (2)} 
The following sequence is exact:
\[
\tag{\ref{trl-07}.b}
1\to \{(1,1,1),(-1,-1,1)\} 
\to \tilde{B}_3\xrightarrow{q} {\rm SO}(7)\to 1.\]
\end{proposition}
\medskip

By Lemma~\ref{trl-05}(2) 
and (\ref{trl-07}.a), 
$\tilde{D}_4$ is connected 
and a two-hold covering group of ${\rm SO}(8)$,
and by Lemma~\ref{trl-05}(1) 
and (\ref{trl-07}.b),
$\tilde{B}_3$ is connected 
and a two-hold covering group of ${\rm SO}(7).$
So denote 
\[{\rm Spin}(8):=\tilde{D}_4,\quad
{\rm Spin}(7):=\tilde{B}_3.\]

\section{The construction of concrete
elements of $\mathrm{F}_{4(-20)}$.}\label{cce}
In order to give the orbit decomposition
of $\mathcal{J}^1$
under the action of $\mathrm{F}_{4(-20)}$,
we must know concrete
elements of $\mathrm{F}_{4(-20)}$ and these operation on
$\mathcal{J}^1$.
In this section, we present
concrete
elements $\varphi_0(g_1,g_2,g_3)$ and 
$\exp (t\tilde{A}_i^1(a))$.

\begin{lemma}\label{cce-01} 
The following equations hold.
\begin{gather*}
\tag{\ref{cce-01}.a}
(\mathrm{F}_{4(-20)})_{F_3^1(1)}
=(\mathrm{F}_{4(-20)})_{E_3,F_3^1(1)}.\\
\tag{\ref{cce-01}.b}
(\mathrm{F}_{4(-20)})_{E_1,F_3^1(1)}
=(\mathrm{F}_{4(-20)})_{E_1,E_2,E_3,F_3^1(1)}
=(\mathrm{F}_{4(-20)})_{E_2,F_3^1(1)}.
\end{gather*}
\end{lemma}
\begin{proof}
For all $g\in (\mathrm{F}_{4(-20)})_{F_3^1(1)}$, 
$g E_3=g(F_3^1(1)^{\times2})
=(g F_3^1(1))^{\times 2}
=F_3^1(1)^{\times 2}=E_3$. 
Thus $g\in (\mathrm{F}_{4(-20)})_{F_3^1(1),E_3}$
and so $(\mathrm{F}_{4(-20)})_{F_3^1(1)}\subset
(\mathrm{F}_{4(-20)})_{F_3^1(1),E_3}$.
From $(\mathrm{F}_{4(-20)})_{F_3^1(1),E_3}
\subset(\mathrm{F}_{4(-20)})_{F_3^1(1)}$,
(\ref{cce-01}.a) follows. 
Next, 
obviously,
$(\mathrm{F}_{4(-20)})_{E_1,E_2,E_3,F_3^1(1)}\subset 
(\mathrm{F}_{4(-20)})_{E_1,F_3^1(1)}$. 
Conversely, fix 
$g\in (\mathrm{F}_{4(-20)})_{E_1,F_3^1(1)}$.
By (\ref{cce-01}.a), 
$g\in (\mathrm{F}_{4(-20)})_{E_1,E_3,F_3^1(1)}$
and then
$g E_2=g(E-(E_1+E_3))=E_2$.
Thus
$g\in (\mathrm{F}_{4(-20)})_{E_1,E_2,E_3,F_3^1(1)}$
and so $(\mathrm{F}_{4(-20)})_{E_1,F_3^1(1)}\subset
(\mathrm{F}_{4(-20)})_{E_1,E_2,E_3,F_3^1(1)}$.
Hence $(\mathrm{F}_{4(-20)})_{E_1,E_2,E_3,F_3^1(1)}
=(\mathrm{F}_{4(-20)})_{E_1,F_3^1(1)}$.
Similarly,
$(\mathrm{F}_{4(-20)})_{E_1,E_2,E_3,F_3^1(1)}
=(\mathrm{F}_{4(-20)})_{E_2,F_3^1(1)}$.
\end{proof}
\medskip

\begin{lemma}\label{cce-02} 
{\rm (1)} 
The group ${\rm Spin}(8)$
is isomorphic to the stabilizer
$(\mathrm{F}_{4(-20)})_{E_1,E_2,E_3}$
of $\mathrm{F}_{4(-20)}$,
that is, the isomorphism
$\varphi_0:{\rm Spin}(8)\to 
(\mathrm{F}_{4(-20)})_{E_1,E_2,E_3}$
is given by
\[
\tag{\ref{cce-02}}
\varphi_0(g_1,g_2,g_3)
(\sum{}_{i=1}^3(\xi_iE_i+F_i^1(x_i)))
=\sum{}_{i=1}^3(\xi_iE_i+F_i^1(g_ix_i)).
\]

{\rm (2)}
The restriction  of  $\varphi_0$
on the subgroup ${\rm Spin}(7)$ 
is an isomorphism from ${\rm Spin}(7)$
onto $(\mathrm{F}_{4(-20)})_{E_1,E_2,E_3,F_3^1(1)}$.
\end{lemma}

\begin{proof} 
(1) We can prove it similar to
\cite[Theorem 2.7.1]{Yi_arxiv}.

(2)
By (1),  $\varphi_0$ is a mono-morphism.
Thus it is enough to show that $\varphi_0$ is onto.
Fix $g\in \mathrm{B}_3.$
By (1), $g=\varphi_0(g_1,g_2,g_3)$
for some $(g_1,g_2,g_3)
\in {\rm Spin}(8).$
Then 
$F_3^1(1)=\varphi_0(g_1,g_2,g_3)F_3^1(1)
=F_3^1(g_3 1).$
Thus $g_3 1=1$
and so
$(g_1,g_2,g_3)
\in {\rm Spin}(7)$.
Hence  $\varphi_0:{\rm Spin}(7)\to 
(\mathrm{F}_{4(-20)})_{E_1,E_2,E_3,F_3^1(1)}$ is onto.
\end{proof}

Denote the subgroups
$\mathrm{D}_4:=\varphi_0(\mathrm{Spin}(8))
=(\mathrm{F}_{4(-20)})_{E_1,E_2,E_3}$
and $\mathrm{B}_3:=\varphi_0(\mathrm{Spin}(7))
=(\mathrm{F}_{4(-20)})_{E_1,E_2,E_3,F_3^1(1)}$, 
respectively.

\begin{lemma}\label{cce-03} 
Let $g
=(g_1,g_2,g_3)
\in {\rm Spin}(7)$
and $p,x \in {\bf O}.$
\[
\tag{\ref{cce-03}}
\left\{
\begin{array}{l}
\begin{array}{rlrl}
{\rm (i)}&
\multicolumn{3}{l}{
\varphi_0(g)
(-E_1+E_2)=-E_1+E_2,}\\
\smallskip
{\rm (ii)}&
\varphi_0(g)P^-=P^-,&
{\rm (iii)}&\varphi_0(g)F_3^1(p)
=F_3^1(g_3 p),\\
\smallskip
{\rm (iv)}&\varphi_0(g)E_3=E_3,
&{\rm (v)}&\varphi_0(g)E=E,
\end{array}\\
\smallskip
\begin{array}{rl}
{\rm (vi)}&
\varphi_0(g)
Q^+(x)=Q^+(g_1 x),\\
\smallskip
{\rm (vii)}&
\varphi_0(g)
Q^-(x)=Q^-(g_1 x).
\end{array}
\end{array}
\right.
\]
\end{lemma}
\begin{proof} 
From $\varphi_0(g)\in
\mathrm{B}_3=
(\mathrm{F}_{4(-20)})_{E_1,E_2,E_3,F_3^1(1)}$,
the first five equations follow.
Because of $g_2=\epsilon g_1 \epsilon$, we see
$\varphi_0(g)F_2^1(\overline{x})
=F_2^1(\epsilon g_1\epsilon\overline{x})
=F_2^1(\overline{g_1 x})$.
Thus
(vi) and (vii) follow.
\end{proof}
\medskip

\begin{proposition}\label{cce-04}
Let $Y={\rm diag}(r_1,r_2,r_3)\in\mathcal{J}^1$
where $r_1,r_2,r_3$ are different from each other.
Then  
$(\mathrm{F}_{4(-20)})_Y=\mathrm{D}_4.$
\end{proposition}
\begin{proof} Obviously
 $\mathrm{D}_4=(\mathrm{F}_{4(-20)})_{E_1,E_2,E_3}
\subset(\mathrm{F}_{4(-20)})_Y$.
Conversely, fix $g\in (\mathrm{F}_{4(-20)})_Y$.
Let $i\in\{1,2,3\}$ and
indexes $i,i+1,i+2$ are counted modulo $3$. 
Because of
$(r_iE-Y)^{\times 2}
=(r_{i+1}-r_i)(r_{i+2}-r_i)E_i$
and $\mathrm{tr}((r_iE-Y)^{\times 2})
=(r_{i+1}-r_i)(r_{i+2}-r_i)\ne 0$,
we see that
$E_{Y,r_i}$ is well-defined and
$E_{Y,r_i}=E_i$. 
By (\ref{prl-10})(iii),
$g E_i=g E_{Y,r_i}=E_{g Y,r_i}
= E_{Y,r_i}=E_i$.
Thus  $g \in \mathrm{D}_4$ and so 
$(\mathrm{F}_{4(-20)})_Y\subset \mathrm{D}_4.$ 
Hence $(\mathrm{F}_{4(-20)})_Y=\mathrm{D}_4.$ 
\end{proof}
\medskip

Denote the Lie algebra 
$\mathfrak{f}_{4(-20)}:=Lie(\mathrm{F}_{4(-20)})$
({\it resp}. $\mathfrak{f}_4^{\mathbb{C}}
:=Lie(\mathrm{F}_4^{\mathbb{C}})$).
By Proposition~\ref{prl-08} ({\it resp}.
Proposition~\ref{prl-04}),
\begin{gather*}
\mathfrak{f}_{4(-20)}
=\{ \delta \in 
\mathrm{End}_{\mathbb{R}}(\mathcal{J}^1)|
~\delta(X \times Y)=\delta X \times Y+ X \times 
\delta Y\}\\
(resp.~\mathfrak{f}_4^{\mathbb{C}}
=\{ \delta \in 
\mathrm{End}_{\mathbb{C}}(\mathcal{J}^{\mathbb{C}})|
~\delta(X \times Y)=\delta X \times Y+ 
X \times \delta Y\}).
\end{gather*}
Given
$\delta \in \mathfrak{f}_{4(-20)}$
({\it resp}. $\mathfrak{f}_4^{\mathbb{C}}$),
\[\mathrm{tr}(\delta X)=0,~~
\delta E = 0,~~(\delta X|Y)+(X | \delta Y)=0,
~~(\delta X|X |X)=0\]
for all $X,Y \in \mathcal{J}^1$
({\it resp}. $\mathcal{J}^{\mathbb{C}}$).
The Lie subalgebra $\mathfrak{d}_4$ 
({\it resp}. $\mathfrak{d}_4^{\mathbb{C}}$)
of 
$\mathfrak{f}_{4(-20)}$ ({\it resp}. 
$\mathfrak{f}_4^{\mathbb{C}}$) is defined by
\begin{gather*}
\mathfrak{d}_4
:=\{D \in\mathfrak{f}_{4(-20)}
|~D E_i=0,~i=1,2,3\}\\
(resp.~\mathfrak{d}_4^{\mathbb{C}}
:=\{D\in\mathfrak{f}_4^{\mathbb{C}}
|~D E_i=0,~i=1,2,3\})
\end{gather*}
and 
the Lie algebra $\mathfrak{D}_4$ by
\[
\mathfrak{D}_4
:=\{D\in{\rm End}_{\mathbb{R}}({\bf O})
|~(Dx|y)+(x|Dy)=0\}.
\]

The following lemma is proved
in \cite{Fh1951}(cf. \cite[Lemmas 1.3.6, 1.3.7,
Proposition 2.3.7]{Yi_arxiv}).
\begin{lemma}\label{cce-05}
Let $x,y\in {\bf O}$.

{\rm (1)}
 For all $D_1\in
\mathfrak{D}_4$ , there exist $D_2,D_3\in
\mathfrak{D}_4$  such that
\[(D_1x)y + x(D_2y) = \epsilon D_3\epsilon(xy).\]
Also such $D_2$ and $D_3$ 
are uniquely determined for $D_1.$
\smallskip

{\rm (2)} 
For $D_1,~D_2,~D_3\in\mathfrak{D}_4$,
the relation
\[(D_1x)y+x(D_2y)=\epsilon D_3\epsilon(xy)\]
implies that
\[(D_2x)y+x(D_3y)=\epsilon D_1\epsilon(xy),\quad
(D_3x)y+x(D_1y)=\epsilon D_2\epsilon (xy).
\]

{\rm (3)}
The Lie algebra
$\mathfrak{d}_4$ 
is isomorphic to the Lie algebra $\mathfrak{D}_4$
under the correspondence
$D_1\mapsto d\varphi_0(D_1,D_2,D_3)$
given by
\[\tag{\ref{cce}.5}
d\varphi_0(D_1,D_2,D_3)(\sum_{i=1}^3(
\xi_iE_i+F_i^1(x_i)))
=\sum_{i=1}^3 F_i^1(D_ix_i).\]
where $D_2$ and $D_3$ are elements 
of $\mathfrak{D}_4$ 
which are determined by $D_1$ from the infinitesimal
triality:
\[(D_1x)y + x(D_2y) = \epsilon D_3\epsilon(xy).\]
\end{lemma}
\medskip

Let $i \in \{ 1, 2, 3\}$ 
and indexes $i, i+1, i+2$ be
counted modulo $3$.
For $a \in {\bf O}^{\mathbb{C}}$,
the skew-hermitian matrix 
$A_i(a) \in M(3,{\bf O}^{\mathbb{C}})$
is denoted by
\[A_i(a):=a E_{i+1,i+2}- \overline{a} E_{i+2,i+1},\]
the linear subspace $\mathfrak{u}_i^{\mathbb{C}}$
of $M(3,{\bf O}^{\mathbb{C}})$ by
$\mathfrak{u}_i^{\mathbb{C}}
:= \{A_i(a) |~a \in {\bf O}^{\mathbb{C}} \}$
and the linear subspace $\mathfrak{R}^{\mathbb{C}}$ of 
$M(3,{\bf O}^{\mathbb{C}})$ by
\[\mathfrak{R}^{\mathbb{C}}
:=\mathfrak{u}_1^{\mathbb{C}}
\oplus
\mathfrak{u}_2^{\mathbb{C}}
\oplus
\mathfrak{u}_3^{\mathbb{C}}
=\{A\in M(3,{\bf O}^{\mathbb{C}})
|~A^*=-A,~\mathrm{diag}(A)=0\}\]
where $\mathrm{diag}(A)=0$ means 
that all diagonal elements
$a_{ii}$ of $A$ are $0$.
For $A \in \mathfrak{R}^{\mathbb{C}}$,
the element $\tilde{A}
\in \mathrm{End}_{\mathbb{C}}(\mathcal{J}^{\mathbb{C}})$
is
defined 
by
\[
\tilde{A} X := A X- X A
\quad \text{for}~ X \in \mathcal{J}^{\mathbb{C}}
\]
the linear subspaces $\tilde{\mathfrak{u}}_i^{\mathbb{C}}$
of ${\rm End}_{\mathbb{C}}(\mathcal{J}^{\mathbb{C}})$
by
$\tilde{\mathfrak{u}}_i^{\mathbb{C}}
:= \{\tilde{A}_i(a) |~a \in {\bf O}^{\mathbb{C}} \}$
and the linear subspace $\tilde{\mathfrak{R}^{\mathbb{C}}}$
of ${\rm End}_{\mathbb{C}}(\mathcal{J}^{\mathbb{C}})$
by
\[\tilde{\mathfrak{R}}^{\mathbb{C}}:=\{\tilde{A}
|~A \in \mathfrak{R}^{\mathbb{C}}\}
=\tilde{\mathfrak{u}}_1^{\mathbb{C}}
\oplus
\tilde{\mathfrak{u}}_2^{\mathbb{C}}
\oplus
\tilde{\mathfrak{u}}_3^{\mathbb{C}}.\]
By direct calculations, we have the following lemma.

\begin{lemma}\label{cce-06}
Let $a \in {\bf O}^{\mathbb{C}}$,
$i \in \{1,2,3\}$ and indexes $i,i+1,i+2$
be counted modulo $3$.
Then the operation of 
$\tilde{A}_i(a)$ on $\mathcal{J}^{\mathbb{C}}$ 
is given by
\[
\tag{\ref{cce-06}}
\left\{
\begin{array}{l}
\begin{array}{rlrl}
{\rm (i)}&
\tilde{A}_i(a)E_i=0,\\
\smallskip
{\rm (ii)}&
\tilde{A}_i(a) E_{i+1}=F_i(-a),&
{\rm (iii)}&
\tilde{A}_i(a) E_{i+2}=F_i(a),
\end{array}
\\
\smallskip
\begin{array}{rl}
{\rm (iv)}&
\tilde{A}_i(a) F_i(x)
=2(a|x)(E_{i+1}-E_{i+2}),\\
\smallskip
{\rm (v)}&
\tilde{A}_i(a) F_{i+1}(x)
=F_{i+2}(\overline{ax}),\\
\smallskip
{\rm (vi)}&
\tilde{A}_i(a) F_{i+2}(x)
=F_{i+1}(-\overline{xa}).
\end{array}
\end{array}
\right.
\]
\end{lemma}
\medskip

The following lemma is proved in 
\cite{Fh1951}(cf. \cite[Proposition 2.3.6,
Theorem 2.3.8]{Yi_arxiv}).
\begin{lemma}\label{cce-07}
{\rm (1)}
$\tilde{\mathfrak{R}}^{\mathbb{C}}$
 is a $\mathbb{C}$-linear subspace of
$\mathfrak{f}_4^{\mathbb{C}}$.
\smallskip

{\rm (2)}
Any element $\delta\in \mathfrak{f}_4^{\mathbb{C}}$
is uniquely expressed by
\[\delta=D+\tilde{A}\quad
\text{for~some}~D \in \mathfrak{d}_4^{\mathbb{C}}
~\text{and}~\tilde{A} 
\in \tilde{\mathfrak{R}}^{\mathbb{C}}.\]
Especially,
\[
\tag{\ref{cce-07}}
\mathfrak{f}_4^{\mathbb{C}}
=\mathfrak{d}_4^{\mathbb{C}} 
\oplus \tilde{\mathfrak{R}}^{\mathbb{C}}
=\mathfrak{d}_4^{\mathbb{C}} 
\oplus \tilde{\mathfrak{u}}_1^{\mathbb{C}}
\oplus \tilde{\mathfrak{u}}_2^{\mathbb{C}}
\oplus \tilde{\mathfrak{u}}_3^{\mathbb{C}}.
\]
\end{lemma}
\medskip

We denote the differential of the involutive automorphism 
$\widetilde{\tau\sigma}$ of $\mathrm{F}_4^{\mathbb{C}}$ as
same notation
$\widetilde{\tau\sigma}$.
Looking again $\mathfrak{f}_4^{\mathbb{C}}$
as an $\mathbb{R}$-Lie algebra,
the involutive $\mathbb{R}$-automorphism 
$\widetilde{\tau\sigma}$ of $\mathfrak{f}_4^{\mathbb{C}}$
induces
the $\mathbb{R}$-Lie subalgebra 
$(\mathfrak{f}_4^{\mathbb{C}})_{\widetilde{\tau\sigma}}$
of $\mathfrak{f}_4^{\mathbb{C}}$.
Then we can write
$\mathfrak{f}_{4(-20)}
=(\mathfrak{f}_4^{\mathbb{C}})_{\widetilde{\tau\sigma}}=\{
\delta \in \mathfrak{f}_4^{\mathbb{C}}|~\tau \sigma \delta
\sigma \tau =\delta\}$.

\begin{lemma}\label{cce-08}
The following equations hold.
\[
\tag{\ref{cce-08}}
\left\{
\begin{array}{rrl}
{\rm (i)}&
(\mathfrak{d}_4^{\mathbb{C}})_{\widetilde{\tau\sigma}}
&=\mathfrak{d}_4,\\
\smallskip
{\rm (ii)}&
(\tilde{\mathfrak{u}}_1^{\mathbb{C}})_{\widetilde{\tau\sigma}}
&=\{ \tilde{A}_1(a)|~a \in {\bf O}\},\\
\smallskip
{\rm (iii)}&
(\tilde{\mathfrak{u}}_j^{\mathbb{C}})_{\widetilde{\tau\sigma}}
&=\{ \tilde{A}_j(\sqrt{-1}a)|~a \in {\bf O}\}\quad
\text{with}~j=2,3.
\end{array}
\right.\]
\end{lemma}
\begin{proof}
Easily, (i) follows.
Suppose that $\tilde{A}_i(z) \in 
(\tilde{\mathfrak{u}}_i^{\mathbb{C}})_{\widetilde{\tau\sigma}}$
with $z \in{\bf O}^{\mathbb{C}}$.
From (\ref{cce-06}), we see
$\tau \sigma \tilde{A}_i(z) \sigma \tau E_{i+1}
=\epsilon(i)F_i(-\tau z)$ and
$\tilde{A}_i(z) E_{i+1}
=\epsilon(i)F_i(-z)$,
so that $\epsilon(i)\tau z=z$.
Thus, if $i=1$ then $z \in {\bf O}$,
else $z=\sqrt{-1}a$ for some $a \in {\bf O}$.
Therefore, we have the necessary conditions
of (ii) and (iii).
Conversely, put
$S=\{E_i,F_i^1(x)|~x\in {\bf O},~i=1,2,3\}$.
$S$
spans the linear space of $\mathcal{J}^{\mathbb{C}}$
over $\mathbb{C}$.
Let $X \in S$ and $a\in{\bf O}$.
Because of (\ref{cce-06}) and $\tau \sigma X=X$,
we see that
$\tau \sigma \tilde{A}_1(a) \sigma \tau X
=\tilde{A}_1(a) X$
and $\tau \sigma \tilde{A}_j(\sqrt{-1}a) \sigma \tau X
=\tilde{A}_j(\sqrt{-1}a) X$ $(j=2,3)$.
Hence the sufficient condition follows.
\end{proof}
\medskip

For $a \in {\bf O}$,
skew-hermitian matrix $A_i^1(a) \in 
M(3,{\bf O}^{\mathbb{C}})$ is defined by
\[
A_1^1(a) := A_1(a),
\quad
A_j(a) := A_j(\sqrt{-1}a)
\quad\text{with}~
j\in\{2,3\}.
\]
The element
$\tilde{A}_i^1(a) \in 
\mathfrak{f}_{4(-20)}$ 
is defined
by
\[
\tilde{A}_i^1(a)X:=A_i^1(a)X-XA_i^1(a)
\quad\text{for}~X\in \mathcal{J}^1\]
and the linear subspace $\tilde{\mathfrak{u}}_i^1$ of
$\mathfrak{f}_{4(-20)}$ by
$\tilde{\mathfrak{u}}_i^1
:=\{\tilde{A}_i^1(a)|~a\in{\bf O}\}$.
Then
we have the following lemma.

\begin{lemma}\label{cce-09}
{\rm (1)} The following equation holds.
\[
\tag{\ref{cce-09}.a}
\mathfrak{f}_{4(-20)}
=\mathfrak{d}_4
\oplus \tilde{\mathfrak{u}}_1^1
\oplus \tilde{\mathfrak{u}}_2^1
\oplus \tilde{\mathfrak{u}}_3^1.\]

{\rm (2)}
Let $a \in {\bf O}$,
$i \in \{1,2,3\}$ and indexes $i,i+1,i+2$
be counted modulo $3$.
Then the operation of 
$\tilde{A}_i^1(a)$ on $\mathcal{J}^1$ is given by
\[
\tag{\ref{cce-09}.b}
\left\{
\begin{array}{l}
\begin{array}{rlrl}
{\rm (i)}&\tilde{A}_i^1(a)E_i=0,\\
\smallskip
{\rm (ii)}&
\tilde{A}_i^1(a) E_{i+1}=F_i^1(-a),&
{\rm (iii)}&
\tilde{A}_i^1(a) E_{i+2}=F_i^1(a),
\end{array}
\\
\smallskip
\begin{array}{rl}
{\rm (iv)}&\tilde{A}_i^1(a) F_i^1(x)
=\epsilon(i)2(a|x)(E_{i+1}-E_{i+2}),\\
\smallskip
{\rm (v)}&\tilde{A}_i^1(a) F_{i+1}^1(x)
=F_{i+2}^1(-\epsilon(i+2)\overline{ax}),\\
\smallskip
{\rm (vi)}&\tilde{A}_i^1(a) F_{i+2}^1(x)
=F_{i+1}^1(\epsilon(i+1)\overline{xa}).
\end{array}
\end{array}\right.
\]
\end{lemma}
\medskip
 
\begin{lemma}\label{cce-10} 
Let $t\in \mathbb{R}$, $a\in S^7,$ 
$\xi,\eta\in\mathbb{R}^3$ and
$x,y\in{\bf O}^3$.
Let
$i \in \{1,2,3\}$ and indexes $i,i+1,i+2$ be
counted modulo $3$.

{\rm (1)}
When $i=1$,
let 
$h^1(\eta;y)\in\mathcal{J}^1$ be
\[\tag{\ref{cce-10}.a}
\left\{
\begin{array}{ccl}
\eta_1&=&\xi_1,\\
\eta_2&=&2^{-1}((\xi_2+\xi_3)
+(\xi_2-\xi_3)\cos 2t)+(a|x_1) \sin 2t,\\
\eta_3&=&2^{-1}((\xi_2+\xi_3)
-(\xi_2-\xi_3)\cos 2t)-(a|x_1) \sin 2t,\\
y_1&=&x_1-2^{-1}(\xi_2-\xi_3) a\sin 2t
-2(a|x_1)a\sin^2 t,\\
y_2&=&x_2\cos t-\overline{x_3a} \sin t,\\
y_3&=&x_3\cos t+\overline{ax_2} \sin t
\end{array}\right.
\]
and when $i\in\{2,3\}$, 
let
$h^1(\eta;y)\in\mathcal{J}^1$ be
\[\tag{\ref{cce-10}.b}
\left\{
\begin{array}{ccl}
\eta_i&=&\xi_i,\\
\eta_{i+1}
&=&2^{-1}((\xi_{i+1}+\xi_{i+2})
+(\xi_{i+1}-\xi_{i+2})\cosh 2t)-(a|x_i) \sinh 2t,\\
\eta_{i+2}&=&2^{-1}((\xi_{i+1}+\xi_{i+2})
-(\xi_{i+1}-\xi_{i+2})\cosh 2t)+(a|x_i) \sinh 2t,\\
y_i&=&x_i
-2^{-1}(\xi_{i+1}-\xi_{i+2})a\sinh 2t+2(a|x_i)a\sinh^2 t.\\
y_{i+1}&=&x_{i+1}\cosh t+\overline{x_{i+2}a}\sinh t,\\
y_{i+2}&=&x_{i+2}\cosh t+\overline{ax_i} \sinh t.
\end{array}\right.
\]
Then $h^1(\eta;y)=\exp(t\tilde{A}_i^1(a))h^1(\xi;x)$
and $\exp(t\tilde{A}_i^1(a))\in (\mathrm{F}_{4(-20)})_{E_i}^0.$
\smallskip

{\rm (2)} 
If $a\in S^6,$ then $\exp(t\tilde{A}_i^1(a))
\in(\mathrm{F}_{4(-20)})_{F_i^1(1)}^0$
for all $i\in\{1,2,3\}$.
\end{lemma}
\begin{proof} (1)
Fix $i \in \{1,2,3\}$.
Let 
$F:\mathbb{R}\times \mathcal{J}^1\to \mathcal{J}^1$
be the mapping defined by $F(t,h^1(\xi;x))=h^1(\eta;y).$
From direct calculations, we have
\[\frac{d}{dt} F(t,h^1(\xi;x))
=\tilde{A}_i^1(a) F(t,h^1(\xi;x))~\text{and}~
F(0,h^1(\xi;x))=h^1(\xi;x)\]
and it follows from the uniqueness of solutions
that
$F(t,h^1(\xi;x))=\exp(t\tilde{A}_i^1(a))
h^1(\xi;x).$
Because of $\eta_i=\xi_i$, we see
$\exp(t\tilde{A}_j^1(a))E_i=E_i$
and therefore $\exp(t\tilde{A}_i^1(a))
\in (\mathrm{F}_{4(-20)})_{E_i}^0$.
Hence (1) follow.

(2) 
Because of (1) and $(a|1)=0$, we see
$\exp(t\tilde{A}_i^1(a))F_i^1(1)=F_i^1(1)$. 
Hence (3) follows.
\end{proof}
\medskip

We give elementarily two lemmas which
implies the difference of orbits of the elements
in $\mathcal{J}^1$
under the action of $\mathrm{F}_{4(-20)}$.

Fix $Y\in\mathcal{J}^1$. 
The inner product $B_Y$ on $\mathcal{J}^1$ 
is defined by 
\[B_Y(X_1,X_2)=(Y|X_1|X_2)
\quad\text{for}~X_i\in\mathcal{J}^1.\]

\begin{lemma}\label{cce-11}
Let $Y_1,Y_2\in\mathcal{J}^1$. 
Assume that $B_{Y_1}$ and $B_{Y_2}$ 
have different signatures
from each other. Then 
$Orb_{\mathrm{F}_{4(-20)}}(Y_1)
\neq Orb_{\mathrm{F}_{4(-20)}}(Y_2)$. 
Furthermore,
\begin{gather*}
\tag{\ref{cce-11}.a}
Orb_{\mathrm{F}_{4(-20)}}(E_1)
\neq Orb_{\mathrm{F}_{4(-20)}}(E_2)
=Orb_{\mathrm{F}_{4(-20)}}(E_3),\\
\tag{\ref{cce-11}.b}
Orb_{\mathrm{F}_{4(-20)}}(E_1-E_2)
\neq Orb_{\mathrm{F}_{4(-20)}}(-E_1+E_2).
\end{gather*}
\end{lemma}
\begin{proof}
Suppose that there exists $g\in \mathrm{F}_{4(-20)}$ 
such that $g Y_1=Y_2$. 
Then
\[
B_{Y_1}(X_1,X_2)=(Y_1|X_1|X_2)
=(gY_1|gX_1|gX_2)
=B_{Y_2}(g X_1,g X_2)\]
for all $X_1,X_2\in\mathcal{J}^1$.
Using Sylvester's theorem, 
inner products $B_{Y_1}$ and $B_{Y_2}$ 
have the same signature.
It contradicts with the assumption.
Thus $Orb_{\mathrm{F}_{4(-20)}}(Y_1)
\neq Orb_{\mathrm{F}_{4(-20)}}(Y_2)$.
Let 
$Y\in\{E_1,~E_2,~E_1-E_2,~-E_1+E_2\}$.
For any
$X=\sum_{i=1}^3(\xi_i E_i+F^1_i(x_i))\in\mathcal{J}^1$,
we have the following table:
 \smallskip

 \begin{tabular}{lll}
  \hline
  $Y$ & $B_Y(X,X)$ & 
The signature of $B_Y$\\
  \hline\hline
  $E_1$ & $\xi_2\xi_3-(x_1|x_1)$ & $(9,1)$\\
  $E_3$ & $\xi_1\xi_2+(x_3|x_3)$ & $(1,9)$\\
  $E_1-E_2$ & $\xi_2\xi_3-\xi_3\xi_1
-(x_1|x_1)-(x_2|x_2)$ & $(18,2)$\\
  $-E_1+E_2$ & $-\xi_2\xi_3+\xi_3\xi_1
+(x_1|x_1)+(x_2|x_2)$ & $(2,18)$\\
  \hline
 \end{tabular}\\
 \smallskip

\noindent 
Then we see $Orb_{\mathrm{F}_{4(-20)}}(E_1)
\neq Orb_{\mathrm{F}_{4(-20)}}(E_3)$
and $Orb_{\mathrm{F}_{4(-20)}}(E_1-E_2)
\neq Orb_{\mathrm{F}_{4(-20)}}(-E_1+E_2)$.
Now from (\ref{cce-10}.a), 
$\exp(2^{-1}\pi\tilde{A}_1^1(1))E_3=E_2$ and  
$Orb_{\mathrm{F}_{4(-20)}}(E_2)
=Orb_{\mathrm{F}_{4(-20)}}(E_3)$.
Hence the result follows.
\end{proof}
\medskip

Denote the linear subspace $(\mathcal{J}^1)_0$
of $\mathcal{J}^1$ as
\[(\mathcal{J}^1)_0
:=\{X\in\mathcal{J}^1|~\mathrm{tr}(X)=0\},\]
and the subsets $\mathcal{R}^{\pm}$  
of $\mathcal{J}^1$ as
\[\mathcal{R}^+
:=\{X\in(\mathcal{J}^1)_0|~X^{\times 2}=P^+\},\quad
\mathcal{R}^-
:=\{X\in(\mathcal{J}^1)_0|~X^{\times 2}=P^-\}\]
respectively.

\begin{lemma}\label{cce-12}
{\rm (1)} For all $X\in\mathcal{R}^{\pm}$,
$\mathrm{tr}(X)=\mathrm{tr}(X^{\times 2})
=\mathrm{det}(X)=0$.
Furthermore, 
$(X^{\times 2})\times X=0$.
\smallskip

{\rm (2)} The following equations hold:
\begin{align*}
\mathcal{R}^+&=\emptyset, \tag{\ref{cce-12}.a}\\
\mathcal{R}^-
&=\{rP^-+Q^+(x)|
~r\in\mathbb{R},~x\in S^7\}.
\tag{\ref{cce-12}.b}
\end{align*}
Furthermore,
\[
\tag{\ref{cce-12}.c}
Orb_{\mathrm{F}_{4(-20)}}(P^+) 
\neq Orb_{\mathrm{F}_{4(-20)}}(P^-).\]
\end{lemma}
\begin{proof}
(1) 
Because of 
$X\in\mathcal{R}^{\pm}$, we see
$\mathrm{tr}(X)=0=\mathrm{tr}(P^{\pm})
=\mathrm{tr}(X^{\times 2})$, 
and from (\ref{prl-03}.b)(iv), we see
$\mathrm{det}(X)X=(X^{\times 2})^{\times 2}
=(P^{\pm})^{\times 2}=0$.
Thus $\mathrm{tr}(X)
=\mathrm{tr}(X^{\times 2})=\mathrm{det}(X)=0$
and by
(\ref{prl-03}.b)(v),
$(X^{\times 2})\times X=0$. 

(2)
Suppose that there exists 
$X=\sum_{i=1}^3(\xi_i E_i+F^1_i(x_i))\in \mathcal{R}^+$.
Using (\ref{prl-06}.a),
from (1), we see
\begin{align*}
0&=(X^{\times 2})\times X= P^+\times X
=2^{-1}\left(-\xi_3E_1+\xi_3E_2\right.\\
&\left.
 +(\xi_2-\xi_1+2(1|x_3)
)E_3
+F^1_1(x_1-\overline{x_2})
+F^1_2(-x_2-\overline{x_1})
+F^1_3(-\xi_3)\right).
\end{align*}
Then $\xi_3=0$.
However, by (\ref{prl-06}.d),
\[1=(P^+)_{E_1}=(X^{\times 2})_{E_1}
=\xi_2\cdot 0-(x_1|x_1)=-(x_1|x_1)\leq 0.\]
It is a contradiction and (\ref{cce-12}.a) follows.

Next,
put $\mathcal{R}=\{
rP^-+Q^+(x)
|~r\in\mathbb{R},~x\in S^7\}$. 
Take $X=\sum_{i=1}^3(\xi_iE_i+F^1_i(x_i))
\in \mathcal{R}^-$. 
From (1), 
\begin{align*}
0&=(X^{\times 2})\times X=P^- \times X
=2^{-1}\left(\xi_3 E_1-\xi_3 E_2\right.\\
&\left.+(-\xi_2
+\xi_1+2(1|x_3))E_3
+F^1_1(x_1-\overline{x_2})
+F^1_2(x_2-\overline{x_1})
+F^1_3(-\xi_3)\right).
\end{align*}
Then $\xi_3=0$ and $x_2=\overline{x_1}$.
Because of $\xi_1+\xi_2=\mathrm{tr}(X)=0$, 
we see $\xi_2=-\xi_1$.
Next, by (\ref{prl-06}.d),
\begin{align*}
P^-=X^{\times 2}
=&-(x_1|x_1)E_1+(x_1|x_1)E_2+(-\xi_1^2+(x_3|x_3))E_3\\
&+F^1_1(-\overline{\overline{x_1}x_3}-\xi_1x_1)
+F^1_2(\overline{x_3x_1}+\xi_1\overline{x_1})
+F^1_3(\mathrm{n}(x_1)).
\end{align*}
Then $\mathrm{n}(x_1)=1$ and 
$0=x_1(\overline{x_3}+\xi_1)$. 
From $x_1\ne 0$, we see
$x_3=-\xi_1$
and $X=
-\xi_1P^-+F^1_1(x_1)+F^1_2(\overline{x_1})$
where $\mathrm{n}(x_1)=1$.
Thus $X\in \mathcal{R}$ 
and so $\mathcal{R}^-\subset\mathcal{R}$.
Conversely, let $X=
rP^-+Q^+(x)\in \mathcal{R}$ 
where $r\in\mathbb{R}$ and $x\in S^7$.
From direct calculations, 
we see $X\in \mathcal{R}^-$
and so $\mathcal{R}\subset\mathcal{R}^-$.
Hence $\mathcal{R}^-=\mathcal{R}$.

Last, 
we show (\ref{cce-12}.c).
Suppose that there exists $g \in \mathrm{F}_{4(-20)}$ 
such that $g P^+=P^-$.
From (\ref{cce-12}.a) and (\ref{cce-12}.b),
we see $\emptyset=g(\mathcal{R}^+)
=\mathcal{R}^-\neq \emptyset$. 
It is a contradiction as required.
\end{proof}

\section{The stabilizers of Spin group type.}
\label{spn}

In this section, we will explain
the construction of the spin groups
$\mathrm{Spin}(9)$, $\mathrm{Spin}^0(8,1)$
and $\mathrm{Spin}^0(7,1)$
as the stabilizers,
respectively.

For $X\in\mathcal{J}^1$, denote the element
$L^{\times}(X) \in \mathrm{End}_{\mathbb{R}}(\mathcal{J}^1)$
by
\[L^{\times}(X)Y:=X\times Y
\quad\text{for}~Y\in\mathcal{J}^1\] 
and for $r\in\mathbb{R}$,
consider the $r$-eigenspace 
of $L^{\times}(X)$ on $\mathcal{J}^1$:
\[
\mathcal{J}^1_{L^{\times}(X),r}=
\{Y\in\mathcal{J}^1|~L^{\times}(X) Y=rY\}.
\]
The quadratic form $Q$ on $\mathcal{J}^1$
is defined by
\[Q(X):=-\mathrm{tr}(X^{\times 2})
\quad\text{for}~X\in\mathcal{J}^1.\]

\begin{lemma}\label{spn-01}
Let $i\in\{1,2,3\}$ and
indexes $i,i+1,i+2$ be counted modulo $3$.
Let $X \in\mathcal{J}^1$ and $r\in\mathbb{R}$.
Then
\begin{gather*}
\tag{\ref{spn-01}.a}
Q(gX)=Q(X) \quad 
\text{for~all}~g\in \mathrm{F}_{4(-20)},\\
\tag{\ref{spn-01}.b}
g\mathcal{J}^1_{L^{\times}(X),r}
=\mathcal{J}^1_{L^{\times}(gX),r}
\quad 
\text{for~all}~g\in \mathrm{F}_{4(-20)},\\
\tag{\ref{spn-01}.c}
\mathcal{J}^1_{L^{\times}(2E_i),-1}
=\{\xi(E_{i+1}-E_{i+2})+F_i^1(x)|~
\xi\in\mathbb{R},~x\in{\bf O}\},\\
\tag{\ref{spn-01}.d}
Q(\xi(E_{i+1}-E_{i+2})+F_i^1(x))
=\xi^2+\epsilon(i)(x|x).
\end{gather*}
Especially, when $i=1$,
then
the quadratic space
$(\mathcal{J}^1_{L^{\times}(2E_1),-1},Q)$ 
is isomorphic to 
$(\mathbb{R}^{0,9},\mathrm{q}_{0,9})$,
and 
when $i\in\{2,3\}$, then
$(\mathcal{J}^1_{L^{\times}(2E_i),-1},Q)$ 
is isomorphic to 
$(\mathbb{R}^{8,1},\mathrm{q}_{8,1})$.
\end{lemma}
\begin{proof}
Using Proposition~\ref{prl-08}
and Lemma~\ref{prl-06},
it follows from
direct calculations.
\end{proof}

The quadratic space
$(\mathcal{J}^1_{L^{\times}(2E_1),-1},Q)$
has a sphere $S^8$  as 
\[S^8:=\{X\in \mathcal{J}^1_{L^{\times}(2E_1),-1}~|
~Q(X)=1\}\] 
and 
the quadratic space
$(\mathcal{J}^1_{L^{\times}(2E_3),-1},Q)$
has 
a positive sphere $\mathcal{S}^{8,1}$ ,
a negative sphere $\mathcal{S}^{8,1}(-1)$
and 
a null cone $\mathcal{N}^{8,1}$
as
\begin{align*}
\mathcal{S}^{8,1}&
:=\{X\in \mathcal{J}^1_{L^{\times}(2E_3),-1}~|
~Q(X)=1\},\\
\mathcal{S}^{8,1}(-1)&
:=\{X\in \mathcal{J}^1_{L^{\times}(2E_3),-1}~|
~Q(X)=-1\},\\
\mathcal{N}^{8,1}&
:=\{X\in \mathcal{J}^1_{L^{\times}(2E_3),-1}~|
~Q(X)=0,~X\ne 0\}
\end{align*}
respectively.
Denote
the subsets $\mathcal{S}^{8,1}_+,
~\mathcal{S}^{8,1}_-\subset\mathcal{S}^{8,1}$ by
\[
\mathcal{S}^{8,1}_+:=\{X\in \mathcal{S}^{8,1}~|
~(X|E_1)>0\},\quad
~~\mathcal{S}^{8,1}_-:=\{X\in \mathcal{S}^{8,1}~|
~(X|E_1)<0\}\]
respectively,
and the subsets $\mathcal{N}^{8,1}_+,
~\mathcal{N}^{8,1}_-\subset\mathcal{N}^{8,1}$ by
\[
\mathcal{N}^{8,1}_+:=\{X\in \mathcal{N}^{8,1}~|
~(X|E_1)>0\},\quad
~~\mathcal{N}^{8,1}_-:=\{X\in \mathcal{N}^{8,1}~|
~(X|E_1)<0\}
\]
respectively.

\begin{lemma}\label{spn-02}
The following equations holds.
\begin{align*}
\tag{\ref{spn-02}.a}
\mathcal{S}^{8,1}&=\mathcal{S}^{8,1}_+\coprod
\mathcal{S}^{8,1}_-,\\
\smallskip
\tag{\ref{spn-02}.b}
\mathcal{N}^{8,1}&=\mathcal{N}^{8,1}_+\coprod
\mathcal{N}^{8,1}_-.
\end{align*}
\end{lemma}
\begin{proof}
Suppose that $X\in \mathcal{J}^1_{L^{\times}(2E_3),-1}$ 
satisfies $X\ne 0$ and
$(E_1|X)=0$.
From (\ref{spn-01}.c),
we see
$X=F_3^1(x)$
for some $(0\ne)x\in{\bf O}$ and $Q(X)<0$.
Therefore, if 
$X\in \mathcal{S}^{8,1}$ (resp. $X\in \mathcal{N}^{8,1}$),
then $(E_1|X)\ne 0$.
\end{proof}
\medskip

The quadratic subspace $(\mathcal{J}^1_{7,1},Q)$  
is defined by
\begin{align*}
\mathcal{J}^1_{7,1}
&:=\{X\in \mathcal{J}^1_{L^{\times}(2E_3),-1}|
~(F_3^1(1)|X)=0\}\\
&=\{\xi(E_1-E_2)+F_3^1(x)~|
~\xi\in\mathbb{R},~x\in{\rm Im}{\bf O}\}
\end{align*}
and
$Q(\xi(E_1-E_2)+F_3^1(x))=\xi^2-(x|x)$.
So the quadratic space
$(\mathcal{J}^1_{7,1},Q)$ 
is isomorphic to $(\mathbb{R}^{7,1},\mathrm{q}_{7,1})$
and we denote a positive sphere $\mathcal{S}^{7,1}$
in the quadratic space $(\mathcal{J}^1_{7,1},Q)$
and its subset $\mathcal{S}^{7,1}_+$ as
\[
\mathcal{S}^{7,1}:=\{X\in \mathcal{J}^1_{7,1}~|
~Q(X)=1\},~~\mathcal{S}^{7,1}_+:=\{X\in \mathcal{S}^{7,1}~|
~(X|E_1)>0\}
\]
respectively.
Denote the homomorphisms $\tilde{p}_i,\tilde{q},p_i,q$
as
\begin{align*}
\tilde{p}_i&:(\mathrm{F}_{4(-20)})_{E_i} 
\to {\rm O}(\mathcal{J}^1_{L^{\times}(2E_i),-1},Q),
&\tilde{p}_i(g)
&:=g|\mathcal{J}^1_{L^{\times}(2E_i),-1},\\
\tilde{q}&:(\mathrm{F}_{4(-20)})_{F_3^1(1)} 
\to {\rm O}(\mathcal{J}^1_{7,1},Q),
&\tilde{q}(g)
&:=g|\mathcal{J}^1_{7,1},\\
p_i&:\mathrm{D}_4 \to {\rm O}(F_i^1({\bf O}),Q),
&p_i(g)
&:=g|F_i^1({\bf O}),\\
q&:\mathrm{B}_3 \to {\rm O}(F_3^1({\rm Im}{\bf O}),Q),
&q(g)
&:=g|F_3^1({\rm Im}{\bf O})
\end{align*}
respectively.
\medskip

\begin{lemma}\label{spn-03} 
The homomorphisms
$\tilde{p}_i,~\tilde{q},~p,~q$
are well-defined.
\end{lemma}
\begin{proof}
First, 
by
(\ref{spn-01}.b), $(\mathrm{F}_{4(-20)})_{E_i}$ invariants
$\mathcal{J}^1_{L^{\times}(2E_i),-1}$.
Second,
because of (\ref{cce-01}.a), 
the definition of $\mathcal{J}^1_{7,1}$
and (\ref{spn-01}.b),
$(\mathrm{F}_{4(-20)})_{F_3^1(1)}$
invariants $\mathcal{J}^1_{7,1}$.
Third,
because of $F_i^1({\bf O})=\{X \in \mathcal{J}^1|
~E_j \times X=0,~i\ne j\}$,
$\mathrm{D}_4=(\mathrm{F}_{4(-20)})_{E_1,E_2,E_3}$
invariants $F_i^1({\bf O})$.
Last,
because of 
$F_i^1({\rm Im}{\bf O})
=\{X\in F_i^1({\bf O})~|~(F_i^1(1)|X)=0\}$,
$\mathrm{B}_3=(\mathrm{F}_{4(-20)})_{E_1,E_2,E_3,F_3^1(1)}$ 
invariants $F_3^1(\mathrm{Im}{\bf O})$.
Therefore, from (\ref{spn-01}.a),
it follows that
the restrictions of suitable subspaces of $\mathcal{J}^1$
induce the homomorphisms into suitable orthogonal groups. 
\end{proof}
\medskip

We use trivial lemma to
determine the $G$-orbits of $X$.

\begin{lemma}\label{spn-04}
Let $X$ be a set and a group $G$ act on $X$.
Let $I$ be an index set and $i,j \in I$.
Assume that there exists 
a sequence $(X_i)_{i \in I}$ 
of subsets of $X$ and a sequence $(v_i)_{i \in I}$
of elements in $X$
such that the
following conditions {\rm (i)-(iv)} hold:

{\rm (i)} $X=\cup_{i\in I} X_i$, {\rm (ii)} $v_i\in X_i$, 
{\rm (iii)} $Orb_G(v_i)\neq 
Orb_G(v_j)\Leftrightarrow i\neq j$,\par
{\rm (iv)} $X_i\subset Orb_G(v_i)$.

\noindent
Then $X_i=Orb_G(v_i)$ for all $i\in I$.
\end{lemma}
\begin{proof}
Take $x\in Orb_G(v_i)$. 
Because of $x\in Orb_G(v_i)\subset X=\cup_{i\in I} X_i$, 
there exists $j \in I$ such that $x\in X_j$.
By (iv), $x\in X_j\subset Orb_G(v_j)$
so that $x \in Orb_G(v_i) \cap Orb_G(v_j)$.
By (iii), $i=j$ and 
$x\in X_i$. Then $Orb_G(v_i)\subset X_i$.
Thus, from (iv), we have
$X_i=Orb_G(v_i)$.
\end{proof}
\medskip

\begin{lemma}\label{spn-05}
Let $j\in\{1,2,3\}$.
For all $X=\sum_{i=1}^3(\xi_iE_i+F^1_i(x_i))
\in \mathcal{J}^1$,  
there exists $\varphi_0(g_1,g_2,g_3)\in \mathrm{D}_4$ 
such that
\[
\tag{\ref{spn-05}}
\varphi_0(g_1,g_2,g_3)X=\left(\sum{}_{i=1}^3\xi_iE_i\right)
+F^1_j\left(\sqrt{\mathrm{n}(x_j)}\right)
+\sum{}_{k=1}^2F^1_{j+k}(g_{j+k}x_{j+k})\] 
where the index $j+k$ is counted modulo $3$.
\end{lemma}
\begin{proof}
Given $x_j=(X)_{F_j^1}\in {\bf O}$,
we can take $g_j\in \mathrm{SO}(8)$ 
such that $g_j x_j=\sqrt{\mathrm{n}(x_j)}$. 
By (\ref{trl-07}.a), there exists  
$(g_1,g_2,g_3)\in \tilde{D}_4$ and
from (\ref{cce-02}), we have the result.
\end{proof}
\medskip

\begin{lemma}\label{spn-06}
The following equations hold.
\begin{align*}
\tag{\ref{spn-06}.a}
&S^8=Orb_{(\mathrm{F}_{4(-20)})_{E_1}}
(E_2-E_3).\\
\tag{\ref{spn-06}.b}
&\mathcal{S}^{8,1}(-1)=Orb_{(\mathrm{F}_{4(-20)})_{E_3}}
(F_3^1(1)).\\
\tag{\ref{spn-06}.c}
&\left\{
\begin{array}{crl}
{\rm (i)}&
\mathcal{S}^{8,1}_+
&=Orb_{(\mathrm{F}_{4(-20)})_{E_3}}(E_1-E_2),\\
{\rm (ii)}&
\mathcal{S}^{8,1}_-
&=Orb_{(\mathrm{F}_{4(-20)})_{E_3}}(-E_1+E_2).
\end{array}
\right.\\
\tag{\ref{spn-06}.d}
&\left\{\begin{array}{crl}
{\rm (i)}&
\mathcal{N}^{8,1}_+
&=Orb_{(\mathrm{F}_{4(-20)})_{E_3}}(P^+),\\
{\rm (ii)}&
\mathcal{N}^{8,1}_-
&=Orb_{(\mathrm{F}_{4(-20)})_{E_3}}(P^-).
\end{array}
\right.
\end{align*}
Furthermore, $\mathcal{S}^{8,1}_+$ is connected.
\end{lemma}
\begin{proof}
(a) 
By Lemma~\ref{spn-03}, $(\mathrm{F}_{4(-20)})_{E_1}$
acts on $S^8$.
Fix $X\in S^8$.
By (\ref{spn-01}.c) and (\ref{spn-01}.d),
$X$ can be expressed by $X=\xi(E_2-E_3)+F_1^1(x)$
where $\xi\in \mathbb{R},$ $x\in{\bf O}$ 
and $\xi^2+\mathrm{n}(x)=1.$
By (\ref{spn-05}),
there exists $g \in \mathrm{D}_4
\subset (\mathrm{F}_{4(-20)})_{E_1}$
such that
$g X=\xi(E_2-E_3)+F_1^1(\sqrt{{\rm n}(x)})$.
We can write 
\[g X
=\cos 2t(E_2-E_3)+F_1^1(\sin 2t)\]
for some $t\in\mathbb{R}$.
For this $t$,
using (\ref{cce-10}.a),
$\exp(t\tilde{A}_1^1(1))\in(\mathrm{F}_{4(-20)})_{E_1}$
and from direct calculations, we see that
\[\exp(t
\tilde{A}_1^1(1))
g X=E_2-E_3.\]
Hence (\ref{spn-06}.a) follows.

(b)
By Lemma~\ref{spn-03}, $(\mathrm{F}_{4(-20)})_{E_3}$
acts on $\mathcal{S}^{8,1}(-1)$.
Fix $X\in \mathcal{S}^{8,1}(-1)$.
By (\ref{spn-01}.c) and (\ref{spn-01}.d),
$X$ can be expressed by $X=\xi(E_1-E_2)+F_3^1(x)$
where $\xi\in \mathbb{R},$ $x\in{\bf O}$ 
and $\xi^2-\mathrm{n}(x)=-1.$
By (\ref{spn-05}),
there exists $g\in \mathrm{D}_4
\subset (\mathrm{F}_{4(-20)})_{E_3}$ such that
\[gX=\xi (E_1-E_2)+F^1_3(r_0)\quad\text{where}~r_0\geq 0,
~~\xi^2-r_0^2=-1.\]
By  (\ref{cce-10}.b),
$\exp\left(
4^{-1}\log\left((r_0+\xi)/(r_0-\xi)\right)
\tilde{A}_3^1(1)
\right)\in (\mathrm{F}_{4(-20)})_{E_3}$
and because of $\xi^2-r_0^2=-1$, 
we calculate
that
\[\exp\left(
4^{-1}\log\left((r_0+\xi)/(r_0-\xi)\right)
\tilde{A}_3^1(1)
\right)gX=
F_3^1(\pm1).\]
When $F_3^1(-1)$, multiplying 
$\varphi_0(1,-1,-1)\in \mathrm{D}_4$ 
from the left side,
\[\varphi_0(1,-1,-1)\exp\left(
4^{-1}\log\left((r_0+\xi)/(r_0-\xi)\right)
\tilde{A}_3^1(1)
\right)gX=
F_3^1(1).\]
Hence (\ref{spn-06}.b) follows.

(c) We show (\ref{spn-06}.c)
by using Lemma~\ref{spn-04}.
Denote $G=(\mathrm{F}_{4(-20)})_{E_3}$ or 
$(\mathrm{F}_{4(-20)})_{E_3}^0$. 
By Lemma~\ref{spn-03},
$G$ acts on $\mathcal{S}^{8,1}$. 
We consider that $X=\mathcal{S}^{8,1}$, 
$(X_1,X_2)=(\mathcal{S}^{8,1}_+,\mathcal{S}^{8,1}_-)$, 
$(v_1,v_2)=(E_1-E_,-E_1+E_2)$ in Lemma~\ref{spn-04}.
First, the condition (i)
follows from
(\ref{spn-02}.a).
Second, the condition (ii) 
follows from direct calculations.
Third, by (\ref{cce-11}.b),
the condition (iii) follows from
\begin{align*}
&Orb_{(\mathrm{F}_{4(-20)})_{E_3}}(E_1-E_2)\subset
Orb_{\mathrm{F}_{4(-20)}}(E_1-E_2)\\
\ne& Orb_{\mathrm{F}_{4(-20)}}(E_1-E_2)
\supset
Orb_{(\mathrm{F}_{4(-20)})_{E_3}}(E_1-E_2).
\end{align*}
Last, we show the condition (iv).
Take $X\in \mathcal{S}^{8,1}_+$. 
By (\ref{spn-05}),
there exists $g\in \mathrm{D}_4\subset G$ such that
\[gX=\xi (E_1-E_2)+F^1_3(r_0)\quad 
\text{where}~\xi>0,~r_0\geq 0,
~\xi^2-r_0^2=1.\]
By (\ref{cce-10}.b), 
$\exp\left(
4^{-1}\log\left((\xi+r_0)/(\xi-r_0)\right)
\tilde{A}_3^1(1)
\right)\in G$
and because of $\xi>0$ and $\xi^2-r_0^2=1$,
we calculate that
\[\exp\left(
4^{-1}\log\left((\xi+r_0)/(\xi-r_0)\right)
\tilde{A}_3^1(1)
\right)gX=
E_1-E_2.\]
Thus $X\in Orb_G(E_1-E_2)$ and so 
$\mathcal{S}^{8,1}_+\subset Orb_G(E_1-E_2)$.
Next, take $X\in \mathcal{S}^{8,1}_-$.
By (\ref{spn-05}),
there exists $g\in \mathrm{D}_4\subset G$ such that
\[gX=\xi (E_1-E_2)+F^1_3(r_0)\quad
\text{where}~\xi<0,~r_0\geq 0,
~\xi^2-r_0^2=1.\]  
By (\ref{cce-10}.b),
$\exp\left(
4^{-1}\log\left((\xi+r_0)/(\xi-r_0)\right)
\tilde{A}_3^1(1)
\right)\in G$
and because of $\xi<0$ and $\xi^2-r_0^2=1$,
we calculate that
\[\exp\left(
4^{-1}\log\left((\xi+r_0)/(\xi-r_0)\right)
\tilde{A}_3^1(1)
\right)gX=
-E_1+E_2.\]
Thus $X\in Orb_G(-E_1+E_2)$ and so 
$\mathcal{S}^{8,1}_-\subset Orb_G(-E_1+E_2)$.
Therefore, the condition (iv) follows.
Hence (\ref{spn-06}.c) follows from Lemma~\ref{spn-04}.
Furthermore, since $\mathcal{S}^{8,1}_+$ 
is a orbit of one element $E_1-E_2$ under the action
of a connected group $(\mathrm{F}_{4(-20)})^0_{E_3}$,
$\mathcal{S}^{8,1}_+$ is connected.

(d) We show (\ref{spn-06}.d)
by using Lemma~\ref{spn-04}.
Denote $G=(\mathrm{F}_{4(-20)})_{E_3}$. 
Since $G$ acts on $\mathcal{N}^{8,1}$, 
we consider $X=\mathcal{N}^{8,1}$, 
$(X_1,X_2)=(\mathcal{N}^{8,1}_+,\mathcal{N}^{8,1}_-)$, 
$(v_1,v_2)=(P^+,P^-)$ in Lemma~\ref{spn-04}.
First, the condition (i) follows from
(\ref{spn-02}.b).
Second, the condition (ii) follows 
from direct calculations.
Third, by (\ref{cce-12}.c),
the condition (iii) follows from 
\begin{align*}
&Orb_{(\mathrm{F}_{4(-20)})_{E_3}}(P^+)\subset
Orb_{\mathrm{F}_{4(-20)}}(P^+)\\
\ne& Orb_{\mathrm{F}_{4(-20)}}(P^-)
\supset
Orb_{(\mathrm{F}_{4(-20)})_{E_3}}(P^-).
\end{align*}
Last, we will show the condition (iv).
Take $X\in \mathcal{N}^{8,1}_+$. 
Because of (\ref{spn-05}),
there exists $g\in \mathrm{D}_4\subset G$ such that
\[gX=\xi (E_1-E_2)+F^1_3(\xi)\quad \text{where}~\xi>0.\]
By
(\ref{cce-10}.b),
$\exp\left(
2^{-1}(\log \xi)
\tilde{A}_3^1(1)
\right)\in G$
and because of $\xi>0$, we calculate that
\[\exp\left(
2^{-1}(\log \xi)
\tilde{A}_3^1(1)
\right)gX=
P^+.\]
Thus $X\in Orb_G(P^+)$, so that 
$\mathcal{N}^{8,1}_+\subset Orb_G(P^+)$.
Next, take $X\in \mathcal{N}^{8,1}_-$.
Because of (\ref{spn-05}),
there exists $g\in \mathrm{D}_4\subset G$ such that
\[gX=\xi (-E_1+E_2)+F^1_3(\xi)\quad\text{where}~\xi>0.\] 
By (\ref{cce-10}.b),
$\exp\left(
2^{-1}(\log \xi)
\tilde{A}_3^1(1)
\right)\in G$
and because of $\xi>0$, we calculate that
\[\exp\left(
2^{-1}(\log \xi)
\tilde{A}_3^1(1)
\right)gX=
P^-.\]
Thus $X\in Orb_G(P^-)$ and so 
$\mathcal{N}^{8,1}_-\subset Orb_G(P^-)$.
Therefore, the condition (iv) follows.
Hence (\ref{spn-06}.d) follows from Lemma~\ref{spn-04}.
\end{proof}
\medskip

\begin{lemma}\label{spn-07} 
{\rm (1)}
$(\mathrm{F}_{4(-20)})_{E_1}/\mathrm{D}_4\simeq S^8$.
Furthermore, $(\mathrm{F}_{4(-20)})_{E_1}$ 
is connected.
\smallskip

{\rm (2)}
$(\mathrm{F}_{4(-20)})_{E_3}/\mathrm{D}_4
\simeq \mathcal{S}^{8,1}_+$.
Furthermore, $(\mathrm{F}_{4(-20)})_{E_3}$ is connected.
\end{lemma}

\begin{proof}
(1) We notice that $(\mathrm{F}_{4(-20)})_{E_1,E_2-E_3}
=(\mathrm{F}_{4(-20)})_{E,E_1,E_2-E_3}
=(\mathrm{F}_{4(-20)})_{E_1,E_2,E_3}=\mathrm{D}_4$.
From (\ref{spn-06}.a), we see
$(\mathrm{F}_{4(-20)})_{E_1}/\mathrm{D}_4\simeq S^8$.
By Lemma~\ref{cce-02}(1), 
$\mathrm{D}_4$ is connected.
Obviously $S^8$ is connected.
Hence 
$(\mathrm{F}_{4(-20)})_{E_1}$ is also connected.

(2) 
We note 
$\mathrm{D}_4=(\mathrm{F}_{4(-20)})_{E_3,E_1-E_2}$.
From (\ref{spn-06}.c),
$(\mathrm{F}_{4(-20)})_{E_3}/\mathrm{D}_4
\simeq \mathcal{S}^{8,1}_+$.
By Lemma~\ref{spn-06},
$\mathcal{S}^{8,1}_+$ is connected.
Because $\mathrm{D}_4$ is connected, we see that
$(\mathrm{F}_{4(-20)})_{E_3}$ is also connected.
\end{proof}
\medskip

For $i\in\{1,2,3\},$ the element
$\sigma_i\in \mathrm{F}_{4(-20)}$ is defined
by
\[
\sigma_i \left(\sum{}_{j=1}^3(\xi_jE_j
+F_j^1(x_j))\right)
:=\sum{}_{j=1}^3(\xi_jE_j
+\epsilon_i(j)F_j^1(x_j)).
\]
Indeed, because of
$\mathrm{det}(\sigma_i X)=\mathrm{det}(X)$ 
and $\sigma_i E=E$, 
applying (\ref{prl-08}.b),
we see
$\sigma_i\in \mathrm{F}_{4(-20)}$ and clearly $\sigma_i^2=1$.
We write the notation $\sigma$ instead of $\sigma_1$
for short.
\medskip

The following result is proved in \cite{Yi1990}.
\begin{proposition}\label{spn-08} 
{\rm (1)}
The following sequence is exact:
\[
\tag{\ref{spn-08}.a}
1 \to  \{1,\sigma_i\}  \to \mathrm{D}_4 
\stackrel{p_i}{\to} {\rm SO}(F_i^1({\bf O}),Q) \to 1.
\]

{\rm (2)}
The following sequence is exact:
\[
\tag{\ref{spn-08}.b}
1\to \{1,\sigma\} \to (\mathrm{F}_{4(-20)})_{E_1} 
\stackrel{\tilde{p}_1}{\to} 
{\rm SO}(\mathcal{J}^1_{L^{\times}(2E_1),-1},Q)\to 1.\]

{\rm (3)}
 The following sequence is exact:
\[
\tag{\ref{spn-08}.c}
1\to \{1,\sigma_3\} \to (\mathrm{F}_{4(-20)})_{E_3} 
\stackrel{\tilde{p}_3}{\to} 
{\rm O}^0(\mathcal{J}^1_{L^{\times}(2E_3),-1},Q)\to 1.\]
\end{proposition}
\begin{proof} 
(1) It follows from 
Lemma~\ref{cce-02}(1) 
and (\ref{trl-07}.a).

(2) 
(cf. \cite[Theorem 2.7.4]{Yi_arxiv}).
By Lemma~\ref{spn-07}(1), 
$(\mathrm{F}_{4(-20)})_{E_1}$ is connected.
Then we see that
$\tilde{p}_1((\mathrm{F}_{4(-20)})_{E_1})\subset
{\rm SO}(\mathcal{J}^1_{L^{\times}(2E_1),-1},Q)$
and the following commutative diagram:
\[\begin{array}{ccccccccc}
1&\to& \mathrm{D}_4 &\to& (\mathrm{F}_{4(-20)})_{E_1} 
&\to& S^8&\to& *\\
 &   &  \downarrow p_1           
&   &  \downarrow \tilde{p}_1 &    & \parallel  &    &\\
1&\to& {\rm SO}(F_1^1({\bf O}),Q) 
&\to&{\rm SO}(\mathcal{J}^1_{L^{\times}(2E_1),-1},Q)  
&\to& S^8&\to& *.
\end{array}
\]
It follows from (1) and the five lemma 
that $\tilde{p}_1$
is onto
and 
$
{\rm Ker}(\tilde{p}_1)
={\rm Ker}(p_1)
=\{1,\sigma\}.$
Hence  (2) follows.

(3)
By Lemma~\ref{spn-07}(2), 
$(\mathrm{F}_{4(-20)})_{E_3}$ is connected.
Then we see that
$\tilde{p}_3((\mathrm{F}_{4(-20)})_{E_3})\subset
{\rm O}^0(\mathcal{J}^1_{L^{\times}(2E_3),-1},Q)$
and
the following commutative diagram:
\[\begin{array}{ccccccccc}
1&\to& \mathrm{D}_4 &\to& (\mathrm{F}_{4(-20)})_{E_3} 
&\to& \mathcal{S}^{8,1}_+&\to& *\\
 &   &  \downarrow p_3           &   
&  \downarrow \tilde{p}_3 &    & \parallel  &    &\\
1&\to& {\rm SO}(F_3^1({\bf O}),Q) 
&\to&{\rm O}^0(\mathcal{J}^1_{L^{\times}(2E_3),-1},Q)  
&\to& \mathcal{S}^{8,1}_+&\to& *.
\end{array}
\]
Similarly, using (1) and the five lemma, (3) follows.
\end{proof}
\medskip

By Lemma~\ref{spn-07}(1) 
and (\ref{spn-08}.b),
we have
$(\mathrm{F}_{4(-20)})_{E_1}$ is connected 
and a two-hold covering group of 
${\rm SO}(\mathcal{J}^1_{L^{\times}(2E_1),-1},Q)$,
and by Lemma~\ref{spn-07}(2) 
and (\ref{spn-08}.c),
$(\mathrm{F}_{4(-20)})_{E_3}$ is connected 
and  
a two-hold covering group of 
${\rm O}^0(\mathcal{J}^1_{L^{\times}(2E_3),-1},Q).$
So denote
\[{\rm Spin}(9):=(\mathrm{F}_{4(-20)})_{E_1},\quad
{\rm Spin}^0(8,1):=(\mathrm{F}_{4(-20)})_{E_3}.\]

\begin{proposition}\label{spn-09} 
{\rm (1)} Let $Y=(r_1-r_2)E_1+r_2E\in\mathcal{J}^1$ 
where $r_1\ne r_2.$
Then
$(\mathrm{F}_{4(-20)})_Y={\rm Spin}(9)$.
\smallskip

{\rm (2)} Let $Y'=(r_1-r_2)E_3+r_2E\in\mathcal{J}^1$
where $r_1\ne r_2.$
Then
$(\mathrm{F}_{4(-20)})_{Y'}={\rm Spin}^0(8,1)$.
\end{proposition}

\begin{proof}  
Since the element $E$ 
is invariant under the $\mathrm{F}_{4(-20)}$-action,
we have $(\mathrm{F}_{4(-20)})_Y=(\mathrm{F}_{4(-20)})_{E_1}$
and $(\mathrm{F}_{4(-20)})_{Y'}=(\mathrm{F}_{4(-20)})_{E_3}$.
\end{proof}
\medskip

\begin{lemma}\label{spn-10} 
{\rm (1)} The following equation holds.
\[
\tag{\ref{spn-10}}
\mathcal{S}^{7,1}_+
=Orb_{(\mathrm{F}_{4(-20)})_{F_3^1(1)}^0}(E_1-E_2)
=Orb_{(\mathrm{F}_{4(-20)})_{F_3^1(1)}}(E_1-E_2).
\]
Furthermore, $\mathcal{S}^{7,1}_+$ is connected.
\smallskip

{\rm (2)}
$(\mathrm{F}_{4(-20)})_{F_3^1(1)}/\mathrm{B}_3
\simeq \mathcal{S}^{7,1}_+.$
Furthermore, $(\mathrm{F}_{4(-20)})_{F_3^1(1)}$ 
is connected.
\end{lemma}

\begin{proof} 
(1) Note $\mathcal{S}^{7,1}_+
=\mathcal{S}^{8,1}_+\cap \mathcal{S}^{7,1}.$
Let $X \in \mathcal{S}^{7,1}_+$ 
and $g\in (\mathrm{F}_{4(-20)})_{F_3^1(1)}$.
Because of $X \in \mathcal{S}^{8,1}_+$
and (\ref{spn-06}.c),
$g X\in\mathcal{S}^{8,1}_+$.
Next, 
$0=(F_3^1(1)|X)
=(g F_3^1(1)|g X)=(F_3^1(1)|g X).$
Thus $g X\in\mathcal{S}^{7,1}_+$
and so  $(\mathrm{F}_{4(-20)})_{F_3^1(1)}$ 
acts on $\mathcal{S}^{7,1}_+$.
Especially, $(\mathrm{F}_{4(-20)})_{F_3^1(1)}^0$ 
acts on $\mathcal{S}^{7,1}_+.$

Next, we will show transitivity.
Fix $X\in\mathcal{S}^{7,1}_+.$
$X$ is expressed by $X=\xi(E_1-E_2)+F_3^1(x)$
where $\xi>0,$ $x\in{\rm Im}{\bf O}$ 
and $\xi^2-{\rm n}(x)=1.$
Using (\ref{trl-04}.a),
$g x=\sqrt{{\rm n}(x)}e_1$
for some $g \in \mathrm{G}_2.$
Then $\varphi_0(g,g,g)
\in \mathrm{G}_2\subset \mathrm{B}_3
\subset (\mathrm{F}_{4(-20)})_{F_3^1(1)}^0$ 
and $\varphi_0(g,g,g)X
=\xi(E_1-E_2)+F_3^1(\sqrt{{\rm n}(x)}e_1)$.
Put $t_0=4^{-1}\log\left(
(\xi+\sqrt{{\rm n}(x)})/(\xi-\sqrt{{\rm n}(x)})\right)
\in\mathbb{R}$.
Because of $e_1 \in S^6$ and Lemma~\ref{cce-10}(2),
we see 
$\exp(t_0\tilde{A}_3^1(e_1))
\in(\mathrm{F}_{4(-20)})_{F_3^1(1)}^0$
and
because of $\xi^2-{\rm n}(x)=1$
and (\ref{cce-10}.b), we calculate
$\exp\left(t_0
\tilde{A}_3^1(e_1)\right)
\varphi_0(g,g,g) X=E_1-E_2$.
Thus  $\mathcal{S}^{7,1}_+
=Orb_{(\mathrm{F}_{4(-20)})_{E_3}^0}(E_1-E_2)
=Orb_{(\mathrm{F}_{4(-20)})_{E_3}}(E_1-E_2)$.
Because $\mathcal{S}^{7,1}_+$ is an orbit of one element 
$E_1-E_2$
under the action of a connected group 
$(\mathrm{F}_{4(-20)})_{E_3}^0,$
$\mathcal{S}^{7,1}_+$ is connected.
Hence (1) follows.

(2) Note $(\mathrm{F}_{4(-20)})_{E_1-E_2,E_3}
=(\mathrm{F}_{4(-20)})_{E,E_1-E_2,E_3}
=(\mathrm{F}_{4(-20)})_{E_1,E_2,E_3}$.
By
(\ref{cce-01}.a) and (\ref{cce-01}.b),
$(\mathrm{F}_{4(-20)})_{F_3^1(1),E_1-E_2}
=(\mathrm{F}_{4(-20)})_{F_3^1(1),E_1-E_2,E_3}
=(\mathrm{F}_{4(-20)})_{F_3^1(1),E_1,E_2,E_3}
=\mathrm{B}_3$.
From (1), we see 
$(\mathrm{F}_{4(-20)})_{F_3^1(1)}/\mathrm{B}_3
\simeq \mathcal{S}^{7,1}_+$.
By Lemma~\ref{cce-02}(2),
$\mathrm{B}_3$ is connected
and by (1),
$\mathcal{S}^{7,1}_+$ are connected.
Hence $(\mathrm{F}_{4(-20)})_{F_3^1(1)}$ is also connected.
\end{proof}
\medskip

\begin{proposition}\label{spn-11} 
{\rm (1)} The following sequence is exact:
\[
\tag{\ref{spn-11}.a}
1 \to \{1,\sigma_3\} \to \mathrm{B}_3 
\stackrel{q}{\to} {\rm SO}
(F_3^1({\rm Im}{\bf O}),Q) \to 1.
\]

{\rm (2)} The following sequence is exact:
\[
\tag{\ref{spn-11}.b}
1\to \{1,\sigma_3\} \to (\mathrm{F}_{4(-20)})_{F_3^1(1)} 
\stackrel{\tilde{q}}{\to} 
{\rm O}^0(\mathcal{J}^1_{7,1},Q)\to 1.
\]
\end{proposition}

\begin{proof} 
(1) It follows 
from Lemma~\ref{cce-02}(2) 
and (\ref{trl-07}.b).

(2)
By Lemma~\ref{spn-10}(2),
$(\mathrm{F}_{4(-20)})_{F_3^1(1)}$ is connected.
Then we see that
$\tilde{q}((\mathrm{F}_{4(-20)})_{F_3^1(1)})\subset
{\rm O}^0(\mathcal{J}^1_{7,1},Q)$
and the following commutative diagram:
\[\begin{array}{ccccccccc}
1&\to& \mathrm{B}_3 &\to& (\mathrm{F}_{4(-20)})_{F_3^1(1)} 
&\to& \mathcal{S}^{7,1}_+&\to& *\\
 &   &  \downarrow q           
&   &  \downarrow \tilde{q} &    & \parallel  &    &\\
1&\to& {\rm SO}(F_3^1({\rm Im}{\bf O}),Q) 
&\to&{\rm O}^0(\mathcal{J}^1_{7,1},Q)  
&\to& \mathcal{S}^{7,1}_+&\to& *.
\end{array}
\]
It follows from  (1) and the five lemma  that
$\tilde{q}$
is onto  and 
${\rm Ker}(\tilde{q})={\rm Ker}(q)
=\{1,\sigma_3\}.$
Hence the assertion follows.
\end{proof}
\medskip

By Lemma~\ref{spn-10}(2) 
and (\ref{spn-11}.b), we have
$(\mathrm{F}_{4(-20)})_{F_3^1(1)}$ is connected and 
a two-hold covering group of ${\rm O}^0(\mathcal{J}^1_{7,1},Q).$
So denote
\[{\rm Spin}^0(7,1):=(\mathrm{F}_{4(-20)})_{F_3^1(1)}.
\]

\begin{proposition}\label{spn-12}
Let $Y=r E_3+p(E-E_3)+qF_3^1(1)$
$\in\mathcal{J}^1$ with $q\ne 0.$
Then  $(\mathrm{F}_{4(-20)})_Y={\rm Spin}^0(7,1)$.
\end{proposition}
\begin{proof} From (\ref{cce-01}.a), we see  
${\rm Spin}^0(7,1)
=(\mathrm{F}_{4(-20)})_{E_3,F_3^1(1)}
\subset(\mathrm{F}_{4(-20)})_Y$.
Conversely, take $g\in (\mathrm{F}_{4(-20)})_Y.$
Because of $(rE-Y)^{\times 2}=((p-r)^2+q^2)E_3$ and
$\mathrm{tr}((rE-Y)^{\times 2})=(p-r)^2+q^2\ne 0$,
we see
$E_{Y,r}\in \mathcal{J}^1$ is well-defined
and $E_{Y,r}=E_3$.
By (\ref{prl-10})(iii),
$g E_3=g E_{Y,r}
=E_{g Y,r}= E_{Y,r}=E_3$.  
Then
$g F_3^1(1)=g(q^{-1}
(Y-(r-p)E_3-p(E-E_3)))
=F_3^1(1)$.
Thus  $g \in {\rm Spin}^0(7,1)$
and so $(\mathrm{F}_{4(-20)})_Y
\subset {\rm Spin}^0(7,1)$. 
Hence $(\mathrm{F}_{4(-20)})_Y={\rm Spin}^0(7,1)$.
\end{proof}
\medskip

For $i\in\{1,2,3\},$ the involutive 
automorphism $\tilde{\sigma}_i$
of $\mathrm{F}_{4(-20)}$ is defined by
$\tilde{\sigma}_i(g):=\sigma_ig\sigma_i$
for $g\in \mathrm{F}_{4(-20)}$
and the subgroup $K$ of $\mathrm{F}_{4(-20)}$ by
\[
K:=(\mathrm{F}_{4(-20)})^{\tilde{\sigma}}
=\{g\in \mathrm{F}_{4(-20)}|~\sigma g=g\sigma\}.
\]

\begin{lemma}\label{spn-13}
Let 
$i,j \in \{1,2,3\}$ and
indexes $i,i+1,i+2,i+j$ be counted modulo $3$.

{\rm (1)} 
The following expressions hold.
\[\tag{\ref{spn-13}.a}
\left\{
\begin{array}{rrcl}
{\rm (i)}&
\mathcal{J}^1_{\sigma_i}
&=&\{(\sum{}_{j=1}^3 \xi_j E_j)+F_i^1(x)|
\xi_j\in\mathbb{R},x\in{\bf O}\}\\
\smallskip
&&=&
\{X\in\mathcal{J}^1~|~4E_i\times (E_i\times X)=X\}
\oplus \mathbb{R}E_i,\\
\smallskip
{\rm (ii)}&
\mathcal{J}^1_{\sigma_i,-1}
&=&\{\sum{}_{j=1}^2F_{i+j}^1(x_{i+j})|
x_{i+j}\in{\bf O}\}\\
\smallskip
&&=&
\{X\in\mathcal{J}^1~|~E_i\times X=0,~(E_i|X)=0\}.
\end{array}\right.
\]

{\rm (2)} $(\mathrm{F}_{4(-20)})^{\tilde{\sigma}_i}$
invariants the linear subspaces
$\mathcal{J}^1_{\sigma_i}$
and 
$\mathcal{J}^1_{\sigma_i,-1}$
of $\mathcal{J}^1$.

{\rm (3)}
Let
$g\in(\mathrm{F}_{4(-20)})^{\tilde{\sigma}_i}.$
Then \[
\tag{\ref{spn-13}.b}
gE_i=E_i+\xi_{i+1}E_{i+1}+\xi_{i+2}E_{i+2}+F_i^1(x)
\]
for~some $\xi_{i+1},\xi_{i+2}\in\mathbb{R}$
and $x\in{\bf O}.$
\end{lemma}
\begin{proof}
(1) 
Using the definition of $\sigma_i$ and Lemma~\ref{prl-06},
it follows from direct calculations.

(2) It follows from $\sigma_ig=g\sigma_i$ 
for all $g\in(\mathrm{F}_{4(-20)})^{\tilde{\sigma}_i}$.

(3) (cf. \cite[Theorem 2.9.1]{Yi_arxiv}).
Now $F_{i+1}^1(1),~F_{i+2}^1(1)
\in \mathcal{J}^1_{\sigma_i,-1}$. By
(2),  $gF_{i+1}^1(1),~gF_{i+2}^1(1)
\in \mathcal{J}^1_{\sigma_i,-1}$
and from (\ref{spn-13}.a), we can write
$gF_{i+1}^1(1)=\sum_{j=1}^2F_{i+j}^1(x_{i+j})$ and
$gF_{i+2}^1(1)=\sum_{j=1}^2F_{i+j}^1(y_{i+j})$
for~some $x_{i+j},y_{i+j}\in{\bf O}$.
By (\ref{prl-06}.a),
$
E_{i+1}
=-\epsilon(i+1)(F_{i+1}^1(1))^{\times 2}$ and
$
E_{i+2}
=-\epsilon(i+2)(F_{i+2}^1(1))^{\times 2}$,
so that
$gE_{i+1}
=-\epsilon(i+1)(gF_{i+1}^1(1))^{\times 2}$
and 
$gE_{i+2}
=-\epsilon(i+2)(gF_{i+2}^1(1))^{\times 2}$.
From (\ref{prl-06}.d),
we see that
$gE_{i+1}
=-\epsilon(i+1)(\sum_{j=1}^2F_{i+j}^1(x_{i+j}))^{\times 2}
=(\sum_{j=1}^2\xi_{i+j}E_{i+j})+F_i^1(u)$ and
$gE_{i+2}
=-\epsilon(i+2)(\sum_{j=1}^2F_{i+j}^1(y_{i+j}))^{\times 2}
=(\sum_{j=1}^2\eta_{i+j}E_{i+j})+F_i^1(v)$
for some $\xi_{i+j},\eta_{i+j}\in\mathbb{R}$ and
$u,v\in{\bf O}$.
Thus  $gE_i=g(E-E_{i+1}-E_{i+2})
=E-\sum_{j=1}^2(\xi_{i+j}+\eta_{i+j})E_{i+j}
-F_i^1(u+v)$.
Hence (3) follows.
\end{proof}
\medskip

The following result is shown in \cite{Yi1990}.
\begin{proposition}\label{spn-14}
Let $i\in\{1,2,3\}$.
Then the following equation hold.
\[
\tag{\ref{spn-14}.a}
(\mathrm{F}_{4(-20)})^{\tilde{\sigma}_i}
=(\mathrm{F}_{4(-20)})_{E_i}.
\]
Especially,
\begin{gather*}
\tag{\ref{spn-14}.b}
K=(\mathrm{F}_{4(-20)})_{E_1}={\rm Spin}(9),\\
\tag{\ref{spn-14}.c}
(\mathrm{F}_{4(-20)})^{\tilde{\sigma}_2}
=(\mathrm{F}_{4(-20)})_{E_2}
\cong{\rm Spin}^0(8,1).
\end{gather*}
\end{proposition}
\begin{proof}
(cf. \cite[Theorem 2.9.1]{Yi_arxiv}).
Fix $g \in (\mathrm{F}_{4(-20)})_{E_i}$.
For all $X\in\mathcal{J}^1,$ $X$ can be expressed by
$X=X_{\sigma_i}+X_{-\sigma_i}$ for some
$X_{\sigma_i} \in \mathcal{J}^1_{\sigma_i}$ and
$X_{-\sigma_i} \in \mathcal{J}^1_{\sigma_i,-1}$.
From (\ref{spn-13}.a), we see
$gX_{\sigma_i}\in\mathcal{J}^1_{\sigma_i}$
and $gX_{-\sigma_i}\in\mathcal{J}^1_{\sigma_i,-1}$.
Then
$g\sigma_iX=gX_{\sigma_i}-gX_{-\sigma_i}
=\sigma_igX$.
Hence $g\in(\mathrm{F}_{4(-20)})^{\tilde{\sigma}_i}$ 
and so $(\mathrm{F}_{4(-20)})_{E_i} 
\subset (\mathrm{F}_{4(-20)})^{\tilde{\sigma}_i}$.

Conversely, take $g\in(\mathrm{F}_{4(-20)})^{\tilde{\sigma}_i}$.
Let index $i+j$ be counted modulo $3$.
By (\ref{spn-13}.b),
$gE_i=E_i+\xi_{i+1}E_{i+1}+\xi_{i+2}E_{i+2}+F_i^1(x)$
for~some $\xi_{i+1},\xi_{i+2}\in\mathbb{R}$
and $x\in{\bf O}.$
Because of
$(gE_i)^{\times 2}=g(E_i^{\times 2})=0$
and (\ref{prl-06}.d),
we see
$0=((gE_i)^{\times 2})_{E_{i+j}}=\xi_{i+j}$ and
$0=((gE_i)^{\times 2})_{F_i^1}=-x$.
Then  $gE_i=E_i.$
Thus  $g\in(\mathrm{F}_{4(-20)})_{E_i}$ 
and so
 $(\mathrm{F}_{4(-20)})^{\tilde{\sigma}_i}
\subset (\mathrm{F}_{4(-20)})_{E_i}.$
Hence (\ref{spn-14}.a) follows.
\end{proof}

\section{The exceptional hyperbolic planes and 
the exceptional null cones.}\label{hp}
In this section, we define 
the exceptional hyperbolic planes and 
the exceptional null cones,
and we will show Proposition~\ref{hp-np}.
Denote
\begin{align*}
\mathcal{H}&:=\{X\in\mathcal{J}^1|~
X^{\times 2}=0,~\mathrm{tr}(X)=1\},\\
\mathcal{H}({\bf O})&:=\{X\in\mathcal{H}|
~(X|E_1)\geq 1\},\\
\mathcal{H}'({\bf O})&:=\{X\in\mathcal{H}|
~(X|E_1)\leq 0\}.
\end{align*}
Then $\mathcal{H}({\bf O})$ and $\mathcal{H}'({\bf O})$
are called the {\it hyperbolic planes} of ${\bf O}$
or the {\it exceptional hyperbolic planes}.
The cone $\mathcal{N}$ in $\mathcal{J}^1$ is defined by
\[\mathcal{N}:=\{X\in\mathcal{J}^1~|~
\mathrm{tr}(X)=\mathrm{tr}(X^{\times 2})=\mathrm{det}(X)=0\}.\]
We recall $(X^{\times 2})^{\times 2}=\mathrm{det}(X)X$
and observe that the cone
$\mathcal{N}$ contains the following cones:
\begin{align*}
\mathcal{N}_1({\bf O})&:=\{X\in\mathcal{J}^1|~X^{\times 2}=0,
~\mathrm{tr}(X)=0,~X\ne 0\},\\
\mathcal{N}_1^+({\bf O})&:=\{X\in\mathcal{J}^1|~
X^{\times 2}=0,~\mathrm{tr}(X)=0,~(X|E_1)>0\},\\
\mathcal{N}_1^-({\bf O})&:=\{X\in\mathcal{J}^1|~
X^{\times 2}=0,~\mathrm{tr}(X)=0,~(X|E_1)<0\},
\end{align*} 
and
\[
\mathcal{N}_2({\bf O}):=\{X\in\mathcal{J}^1|~
\mathrm{tr}(X)=\mathrm{tr}(X^{\times 2})=\mathrm{det}(X)=0,
X^{\times 2}\ne 0\}.\]
Then $\mathcal{N}_1^{\pm}({\bf O})$ 
and $\mathcal{N}_2({\bf O})$
are called the {\it exceptional null cones}.
We write $\mathcal{N}_0({\bf O})$  for the trivial space $\{0\}$
in $\mathcal{J}^1$.

\begin{lemma}\label{hp-01}
The group $\mathrm{F}_{4(-20)}$ acts on $\mathcal{H}$, 
$\mathcal{N}_1({\bf O})$ and $\mathcal{N}_2({\bf O})$
and
\begin{align*} 
\tag{\ref{hp-01}.a}
\mathcal{H}&=\mathcal{H}({\bf O})
\coprod\mathcal{H}'({\bf O}),\\
\tag{\ref{hp-01}.b}
\mathcal{N}_1({\bf O})
&=\mathcal{N}_1^+({\bf O})\coprod\mathcal{N}_1^-({\bf O}).
\end{align*}
\end{lemma}
\begin{proof}
From the definitions of $\mathcal{H}$, 
$\mathcal{N}_1({\bf O})$ and $\mathcal{N}_2({\bf O})$,
we see that
$\mathrm{F}_{4(-20)}$ acts on these spaces.
We show (\ref{hp-01}.a).
$\mathcal{H}({\bf O})
\cap\mathcal{H}'({\bf O})=\emptyset$ is obvious. 
Fix $X=\sum_{i=1}^3(\xi_iE_i+F^1_i(x_i)) \in \mathcal{H}$
where $\xi_i \in \mathbb{R}$ and $x_i \in {\bf O}$.
By (\ref{prl-06}.d), $0=(X^{\times 2})_{E_2}
=\xi_3\xi_1+(x_2|x_2)$ and
$0=(X^{\times 2})_{E_3}=\xi_1\xi_2+(x_3|x_3)$.
Then $\xi_1(\xi_2+\xi_3)=-(x_2|x_2)-(x_3|x_3)\leq 0$,  
so that $(X)_{E_1}=\xi_1\leq 0$ or $\xi_2+\xi_3\leq 0$. 
If $\xi_2+\xi_3\leq 0$, 
then $(X|E_1)=\xi_1
=\mathrm{tr}(X)-(\xi_2+\xi_3)=1-(\xi_2+\xi_3)\geq 1$. 
Hence (\ref{hp-01}.a) follows.

Next, we show (\ref{hp-01}.b).
$\mathcal{N}_1^+({\bf O})\cap\mathcal{N}_1^-({\bf O})
=\emptyset$ is obvious.
Suppose that  
$X\in \mathcal{N}_1({\bf O})$ and $\xi_1=(X|E_1)=0$. 
We can write 
$X=\xi E_2-\xi E_3+\sum_{i=1}^3F^1_i(x_i)$
where $\xi \in \mathbb{R}$ and $x_i \in {\bf O}$.
Then
$0=(X^{\times 2})_{E_1}=-\xi^2-(x_1|x_1)$ and
$0=(X^{\times 2})_{E_i}=(x_i|x_i)~~(i=2,3)$,
and therefore $\xi=0$ and  $x_i=0$ for all $i \in \{1,2,3\}$ 
iff $X=0$.
It contradicts with $X\neq 0$. 
Hence (\ref{hp-01}.b) follows.
\end{proof}
\medskip

Denote 
$\mathcal{J}^1(2;{\bf O})
:=\{\xi_1E_1+\xi_2E_2+F_3^1(x)|~
\xi_i\in\mathbb{R},~x\in{\bf O}\}$.

\begin{lemma}\label{hp-02}
{\rm (1)} For any $X\in \mathcal{J}^1$,
 there exists $g\in K$ such that 
$(g X)_{F_1}=0$ and $(g X|E_1)=(X|E_1)$.
\smallskip

{\rm (2)} Assume $X\in\mathcal{J}^1$ 
satisfies $X^{\times 2}=0$.
Then there exists $g\in K$ such that 
$g X\in\mathcal{J}^1(2;{\bf O})$ and $(g X|E_1)=(X|E_1)$.
\smallskip

{\rm (3)} For any $X\in\mathcal{H}$, 
there exists $g\in K$ such that 
$g X=2^{-1}(E-E_3)+2^{-1}W$ where 
$W\in\mathcal{S}^{8,1}$ 
and $(g X|E_1)=(X|E_1)$.
\smallskip

{\rm (4)} For any $X\in\mathcal{N}_1({\bf O})$, 
there exists $g\in K$ such that 
$g X\in\mathcal{N}^{8,1}$ and $(g X|E_1)=(X|E_1)$.
\end{lemma}
\begin{proof}
From (\ref{spn-14}.b),
we note 
$(k X|E_1)=(kX|kE_1)=(X|E_1)$
for all $k\in K$.
Thus it is enough to show
the conditions
in addition to the condition
$(g X|E_1)=(X|E_1)$.

(1) Set $X=\sum_{i=1}^3(\xi_iE_i+F_i^1(x_i))$.
We consider the decomposition
\begin{align*}
X=&\xi_1E_1+2^{-1}(\xi_2+\xi_3)(E-E_1)
+\left(2^{-1}(\xi_2-\xi_3)(E_2-E_3)+F_1^1(x_1)\right)\\
&+(F_2^1(x_2)+F_3^1(x_3)).
\end{align*}
Put 
$X_0=\xi_1E_1+2^{-1}(\xi_2+\xi_3)(E-E_1)$,
$Y=2^{-1}(\xi_2-\xi_3)(E_2-E_3)+F_1^1(x_1)$
and $X_{-\sigma}=F_2^1(x_2)+F_3^1(x_3)$.
By (\ref{spn-01}.c),
$Y\in \mathcal{J}^1_{L^{\times}(2E_1),-1}$
and
by (\ref{spn-13}.a),
$X_{-\sigma}\in\mathcal{J}^1_{\sigma,-1}$.
Fix $k_0 \in K$.
From (\ref{spn-14}.b) and
Lemma~\ref{spn-13}(2),
we see
$kX_0=X_0$ and
$k_0 X_{-\sigma} \in \mathcal{J}^1_{\sigma,-1}$
so that $(k_0 X_0)_{F_1^1}=0$ and
$(k_0 X_{-\sigma})_{F_1^1}=0$
(see (\ref{spn-13}.a)).
Now, we can write $Y=rW$ for some $r\in \mathbb{R}$ and
$W \in S^8$.
By (\ref{spn-06}.a), 
there exists $g \in K$ such that
$g W=r(E_2-E_3)$.
Thus $(g W)_{F_1^1}=0$
and so $(g X)_{F_1^1}=(g X_0)_{F_1^1}
+(g X_{-\sigma})_{F_1^1}+r(g W)_{F_1^1}=0$.

(2) By (1), there exists $g\in K$ such that 
$g X=(\sum_{i=1}^3r_i E_i)+
(\sum_{i=2}^3F^1_i(y_i))$
where $r_i\in\mathbb{R},$ $y_i\in{\bf O}$ 
and $(X|E_1)=(g X|E_1)=r_1$.
From the assumption, we see
$0=g(X^{\times 2})=(g X)^{\times 2}$.
By (\ref{prl-06}.d), 
$0=((g X)^{\times 2})_{E_1}=r_2r_3$. 
Then we have the following two cases: 
(i) $r_3=0$ or (ii) $r_2=0$.

Case (i) $r_3=0$. Then 
$0=((g X)^{\times 2})_{E_2}=(y_2|y_2)$
and therefore $y_2=0$.
Thus $g X=r_1E_1+r_2E_2+F^1_3(y_3)
\in \mathcal{J}^1(2;{\bf O})$.

Case (ii) $r_2=0$. Then $0=((g X)^{\times 2})_{E_3}=(y_3|y_3)$
and therefore $y_3=0$.
From  Lemma~\ref{cce-10}, 
we see
$\exp(2^{-1}\pi\tilde{A}_1^1(1))
\in (\mathrm{F}_{4(-20)})_{E_1}=K$
and
$\exp((2^{-1}\pi)\tilde{A}_1^1(1))
g  X=
r_1E_1+r_3E_2+F^1_3(\overline{y_2})
\in \mathcal{J}^1(2;{\bf O})$.
Hence (2) follows.
 
(3) By (2), there exists 
$g \in (\mathrm{F}_{4(-20)})_{E_1}$ such that 
$g X=\xi_1E_1+\xi_2E_2+F^1_3(x)
\in\mathcal{J}^1(2;{\bf O})$
with $1=\mathrm{tr}(X)=\mathrm{tr}(g X)=\xi_1+\xi_2$. 
Put $Y=2^{-1}(\xi_1-\xi_2)(E_1-E_2)+F^1_3(x)$.
Then
$g X=2^{-1}(E-E_3)+Y$
and $Y\in \mathcal{J}^1_{L^{\times}(2E_3),-1}$.
Because of $(gX)^{\times 2}=g(X^{\times 2})=0$,
using (\ref{prl-06}.d), we see
$0=(g(X^{\times 2}))_{E_3}
=((g X)^{\times 2})_{E_3}=\xi_1\xi_2+(x|x)$.
Then by (\ref{spn-01}.d),
$Q(Y)=4^{-1}(\xi_1-\xi_2)^2-(x|x)
=4^{-1}(\xi_1+\xi_2)^2-(\xi_1\xi_2+(x|x))
=4^{-1}$.
Thus $Y\in 2^{-1}\mathcal{S}^{8,1}$
and so (3) follows.

(4) By (2), there exists $g\in K$ 
such that $(0\ne )g X=\xi_1E_1+\xi_2E_2+F^1_3(x)$.
Because of  
(\ref{spn-01}.a), (\ref{spn-01}.c) and
$\xi_1+\xi_2=\mathrm{tr}(g X)=\mathrm{tr}(X)=0$,
we see that
$Q(gX)=0$ and
$g X=\xi_1(E_1-E_2)+F^1_3(x)
\in \mathcal{J}^1_{L^{\times}(2E_3),-1}$.
Hence $g X\in\mathcal{N}^{8,1}$.
\end{proof}
\medskip

\noindent
{\bf Proof of (\ref{hp-np}.a) and (\ref{hp-np}.b).}

We show these
by using Lemma~\ref{spn-04}.
By Lemma~\ref{hp-01},
$\mathrm{F}_{4(-20)}$ acts on $\mathcal{H}$. 
We consider that $X=\mathcal{H}$, 
$G=\mathrm{F}_{4(-20)}$, $(X_1,X_2)=(\mathcal{H}({\bf O}), 
\mathcal{H}'({\bf O}))$, $(v_1,v_2)=(E_1, 
E_2)$ in Lemma~\ref{spn-04}.
At first, the condition (i) follows 
from (\ref{hp-01}.a).
At second, the condition (ii) follows 
from direct calculations.
At third, the condition (iii) and 
$Orb_{\mathrm{F}_{4(-20)}}(E_2)=Orb_{\mathrm{F}_{4(-20)}}(E_3)$ 
follow from (\ref{cce-11}.a).
At last, we show the condition (iv).
Take $X\in \mathcal{H}=\mathcal{H}({\bf O})
\coprod \mathcal{H}'({\bf O})$.
By Lemma~\ref{hp-02}(3), 
there exists $g_0\in K$ such that 
\[g_0 X=2^{-1}(E-E_3)+2^{-1}W\]
where $W \in \mathcal{S}^{8,1}$
and $(g_0X|E_1)=(X|E_1)$.

Case $X\in\mathcal{H}({\bf O})$. 
Because of $(g_0 X|E_1)=(X|E_1)\geq 1$,
we see
$(W|E_1)=2(g_0 X|E_1)-(E-E_3|E_1)
=2(g_0 X|E_1)-1>0$ 
so that $W\in \mathcal{S}^{8,1}_+$.
By (\ref{spn-06}.c)(i), 
there exists $g_1\in (\mathrm{F}_{4(-20)})_{E_3}$ 
such that $g_1 W=E_1-E_2$
and it is clear that
\[g_1g_0 X=2^{-1}(E_1+E_2)
+2^{-1}(E_1-E_2)=E_1.\]
Thus $X\in Orb_{\mathrm{F}_{4(-20)}}(E_1)$ and so 
$\mathcal{H}({\bf O})\subset Orb_{\mathrm{F}_{4(-20)}}(E_1)$.

Case $X\in\mathcal{H}'({\bf O})$.
Because of $(g_0 X|E_1)=(X|E_1)\leq 0$,
$(W|E_1)=2(g_0 X|E_1)-1<0$. 
Then $W\in \mathcal{S}^{8,1}_-$.
By (\ref{spn-06}.c)(ii), 
there exists $g_1\in (\mathrm{F}_{4(-20)})_{E_3}$ 
such that $g_1' W=-E_1+E_2$ and it is clear that
\[g_1'g_0 X=2^{-1}(E_1+E_2)
+2^{-1}(-E_1+E_2)=E_2.\]
Thus $X\in Orb_{\mathrm{F}_{4(-20)}}(E_2)$ and so 
$\mathcal{H}({\bf O})\subset Orb_{\mathrm{F}_{4(-20)}}(E_2)$.

Therefore the condition (iv) follows.
Hence (\ref{hp-np}.a) and (\ref{hp-np}.b)
follow from Lemma~\ref{spn-04}.
\qed
\medskip

\begin{lemma}\label{hp-03}
\[\tag{\ref{hp-03}}
\mathcal{N}=\{0\}\coprod\mathcal{N}_1^+({\bf O})
\coprod\mathcal{N}_1^-({\bf O})
\coprod\mathcal{N}_2({\bf O}).\]
\end{lemma}
\begin{proof}
Let $X\in \mathcal{N}.$
By (\ref{prl-03}.b)(iv),
$(X^{\times 2})^{\times 2}=\mathrm{det}(X)X=0$.
Then we have the following three cases:
(i) $X=0$, (ii) $X\ne 0,~X^{\times 2}=0$
or  (iii) $X^{\times 2}\ne 0$.
It implies 
$X\in \{0\}\coprod\mathcal{N}_1({\bf O})
\coprod\mathcal{N}_2({\bf O})$.
Hence (\ref{hp-03}) follows from (\ref{hp-01}.b).
\end{proof}
\medskip

\noindent
{\bf Proof of (\ref{hp-np}.c) and (\ref{hp-np}.d).}

We show these
by using Lemma~\ref{spn-04}.
By Lemma~\ref{hp-01}, 
$\mathrm{F}_{4(-20)}$ acts on $\mathcal{N}_1({\bf O})$.
We consider that 
$X=\mathcal{N}_1({\bf O})$, $G=\mathrm{F}_{4(-20)}$, 
$(X_1,X_2)=
(\mathcal{N}_1^+({\bf O}),\mathcal{N}_1^-({\bf O}))$, 
$(v_1,v_2)=(P^+,P^-)$ in Lemma~\ref{spn-04}.
At first, 
the condition (i) follows from (\ref{hp-01}.b).
At second, 
the condition (ii) follows from direct calculations.
At third, 
the condition (iii) follows from (\ref{cce-12}.c).
At last, 
we show the condition (iv).
Take $X\in \mathcal{N}_1({\bf O})
=\mathcal{N}_1^+({\bf O})\coprod \mathcal{N}_1^-({\bf O})$.
By Lemma~\ref{hp-02}(4), 
there exists $g_0\in K$ such that 
$g_0 X\in \mathcal{N}^{8,1}$ and
$(g_0X|E_1)=(X|E_1)$.

Case $X\in\mathcal{N}_1^+({\bf O})$. 
Because of $(g_0 X|E_1)=(X|E_1)>0$, 
we see $g_0X\in \mathcal{N}^{8,1}_+$.
By (\ref{spn-06}.d)(i), 
there exists $g_1\in (\mathrm{F}_{4(-20)})_{E_3}$ 
such that $g_1 g_0 X=P^+$.
Thus $X\in Orb_{\mathrm{F}_{4(-20)}}(P^+)$ and so 
$\mathcal{N}_1^+({\bf O})\subset Orb_{\mathrm{F}_{4(-20)}}(P^+)$.

Case $X\in\mathcal{N}_1^-({\bf O})$. 
Because of $(g_0 X|E_1)=(X|E_1)<0$, 
we see $g_0X\in \mathcal{N}^{8,1}_-$.
By (\ref{spn-06}.d)(ii), 
there exists $g_1\in (\mathrm{F}_{4(-20)})_{E_3}$ 
such that $g_1 g_0 X=P^-$.
Thus $X\in Orb_{\mathrm{F}_{4(-20)}}(P^+)$ and so 
$\mathcal{N}_1^-({\bf O})\subset Orb_{\mathrm{F}_{4(-20)}}(P^+)$.

Therefore the condition (iv) follows.
Hence (\ref{hp-np}.c) and (\ref{hp-np}.d)
follows from Lemma~\ref{spn-04}.
\qed
\medskip

For $t\in\mathbb{R}$, denote 
\[\alpha_{1,2}(t):=\exp(t(\tilde{A}_1^1(-1)
+\tilde{A}_2^1(1)))\in \mathrm{F}_{4(-20)}.\]
\begin{lemma}\label{hp-04}
Let $X_0=rP^-+Q^+(x)$ 
where $r\in\mathbb{R}$ and $x\in{\bf O}$.
Then
\[
\tag{\ref{hp-04}}
\alpha_{1,2}(t)X_0=(r-2t{\rm Re}(x))P^-+Q^+(x).
\]
Especially, $\alpha_{1,2}(t)\in (\mathrm{F}_{4(-20)})_{P^-}$.
\end{lemma}
\begin{proof}
Using (\ref{cce-06}), 
we see that
$(\tilde{A}_1^1(-1)+\tilde{A}_2^1(1))P^-=0$
and 
$(\tilde{A}_1^1(-1)+\tilde{A}_2^1(1))Q^+(x)
=-2{\rm Re}(x)P^-$.
Thus $\exp(t(\tilde{A}_1^1(-1)
+\tilde{A}_2^1(1))P^-=P^-$
and $\exp(t(\tilde{A}_1^1(-1)
+\tilde{A}_2^1(1))Q^+(x)=Q^+(x)-2t{\rm Re}(x)P^-$.
Hence
the result follows.
\end{proof}
\medskip

\begin{lemma}\label{hp-05}
Let $X\in\mathcal{R}^-$.
Then there exists $g\in (\mathrm{F}_{4(-20)})_{P^-}$ 
such that \[g X=Q^+(1).\]
\end{lemma}
\par
\begin{proof}
By (\ref{cce-12}.b), 
$X$ can be expressed by 
\[X=rP^-+Q^+(x)\quad
\text{for~some}~r\in\mathbb{R}~\text{and}~x\in S^7.\]

(Step 1) We show the following assertion: 
if $x\in{\rm Im}{\bf O}$, 
then there exists $g_0\in(\mathrm{F}_{4(-20)})_{P^-}$ 
such that
$g_0 X=rP^-+Q^+(x')$ where $x'\in S^7$ 
and $\mathrm{Re}(x')\neq 0$.
Suppose that $x \in S^6$.
By (\ref{trl-04}.a), 
there exists $g_1\in \mathrm{G}_2$ such that $g_1x=e_1$.
Put $\alpha_1=\varphi_0(g_1,g_1,g_1)
\in \mathrm{G}_2\subset(\mathrm{F}_{4(-20)})_{P^-}$.
By (\ref{cce-03}),
$\alpha_1 X=rP^-+Q^+(g_0x)=rP^-+Q^+(e_1)$.
Because of $e_2,e_3\in S^6$ 
and Lemma~\ref{trl-04}(3), we can set
$\alpha_1=(L_{e_3,e_2},R_{e_3,e_2},T_{e_3,e_2})
\in \mathrm{B}_3\subset(\mathrm{F}_{4(-20)})_{P^-}$.
By (\ref{cce-03}),
$\alpha_2\alpha_1X
=rP^-+Q^+(e_3(e_2e_1))=rP^-+Q^+(1)$.
Hence the assertion of (Step 1) follows.

(Step 2) We may assume $X=rP^-+Q^+(x)$ 
where $r\in\mathbb{R}$, $x\in S^7$ 
and $\mathrm{Re}(x)\neq 0$ by (Step 1).
From (\ref{hp-04}), we see that
$\alpha_{1,2}(t)\in (\mathrm{F}_{4(-20)})_{P^-}$ and 
$\alpha_{1,2}\left(r/(2{\rm Re}(x))\right)X=Q^+(x)$.
 
(Step 3) We may assume $X=Q^+(x)$
where $x\in S^7$ by (Step 2).
Then $x$ can be expressed by $x=\cos \theta +a\sin \theta$
for some $\theta\in\mathbb{R}$ and $a\in S^6$.
By (\ref{trl-04}.a), 
there exists $g_1\in \mathrm{G}_2$ 
such that $g_1a=e_1$.
Letting $\alpha_1=\varphi_0(g_1,g_1,g_1)
\in \mathrm{G}_2\subset(\mathrm{F}_{4(-20)})_{P^-}$,
\[\alpha_1 X=Q^+(g_1x)
=Q^+(\cos \theta+e_1\sin \theta).\]
Letting
$\alpha_2
=(L_{e_1,e_2},R_{e_1,e_2},T_{e_1,e_2}) 
\in \mathrm{B}_3\subset(\mathrm{F}_{4(-20)})_{P^-}$,
\[\alpha_2\alpha_1X
=Q^+\left(e_1(e_2(\cos \theta+e_1\sin \theta))\right)
=Q^+(e_3\cos \theta+e_2\sin \theta).\]
Again,
there exists $g_2\in \mathrm{G}_2$ such that
$g_2 (e_3\cos \theta+e_2\sin \theta)=e_1$.
Letting $\alpha_3=\varphi_0(g_2,g_2,g_2)
\in \mathrm{G}_2\subset(\mathrm{F}_{4(-20)})_{P^-}$,
\[\alpha_3\alpha_2\alpha_1X=Q^+(e_1).\]
Last, letting $\alpha_4
=\varphi_0(L_{e_3,e_2},R_{e_3,e_2},T_{e_3,e_2})
\in \mathrm{B}_3\subset(\mathrm{F}_{4(-20)})_{P^-}$,
\[\alpha_4\alpha_3\alpha_2\alpha_1X=Q^+(e_3(e_2e_1))
=Q^+(1).\]
Hence the result follows.
\end{proof}
\medskip

\begin{lemma}\label{hp-06}
Let $X\in\mathcal{N}_2({\bf O})$.
Then $X^{\times 2}\in\mathcal{N}_1^-({\bf O})$.
\end{lemma}
\begin{proof}
Obviously $\mathrm{tr}(X^{\times 2})=0$
and from (\ref{prl-03}.b)(iv),
we see
$(X^{\times 2})^{\times 2}=\mathrm{det}(X)X=0$
so that 
$X^{\times 2}\in\mathcal{N}_1({\bf O})$.
By (\ref{hp-01}.b),
$X^{\times 2}\in\mathcal{N}_1^+({\bf O})$ 
or $X^{\times 2}\in\mathcal{N}_1^-({\bf O})$.
Suppose $X^{\times 2}\in\mathcal{N}_1^+({\bf O})$.
By (\ref{hp-np}.c), 
there exists $g \in \mathrm{F}_{4(-20)}$ 
such that 
$P^+=g (X^{\times 2})=(g X)^{\times 2}$ 
and $\mathrm{tr}(g X)=\mathrm{tr}(X)=0$.
Then $gX\in\mathcal{R}^+$. 
Thus it contradicts with (\ref{cce-12}.a).
Hence the result follows.
\end{proof}
\medskip

\noindent
{\bf Proof of (\ref{hp-np}.e).}

By Lemma~\ref{hp-01}, 
$\mathrm{F}_{4(-20)}$ acts on ${N}_2({\bf O})$.
We show transitivity.
Fix $X\in \mathcal{N}_2({\bf O})$.
By Lemma~\ref{hp-05},
$X^{\times 2}\in\mathcal{N}_1^-({\bf O})$
and therefore
by (\ref{hp-np}.d), 
there exists $g_1 \in \mathrm{F}_{4(-20)}$ such that 
$(g_1 X)^{\times 2}
=g_1 (X^{\times 2})=P^-$ and
$\mathrm{tr}(g_1 X)=\mathrm{tr}(X)=0$.
Then $g_1 X\in\mathcal{R}^-$.
Thus, applying Lemma~\ref{hp-05},
there exists $g_2\in (\mathrm{F}_{4(-20)})_{P^-}$ 
such that 
$g_2g_1X=Q^+(1)$.
\qed
\medskip

\begin{remark}{\rm
From (\ref{hp-01}.a), (\ref{hp-np}.a) and (\ref{hp-np}.b),
we see that
$\mathcal{H}=\mathcal{H}({\bf O})\coprod
\mathcal{H}'({\bf O})=
Orb_{\mathrm{F}_{4(-20)}}(E_1)\coprod
Orb_{\mathrm{F}_{4(-20)}}(E_3)$
and that $\mathcal{H}$ consists of
two $\mathrm{F}_{4(-20)}$-orbits.
Then
we can write
$\mathcal{H}=\{X\in\mathcal{J}^1|
~X\circ X=X,~\mathrm{tr}(X)=1\}$.

For $\xi=(\xi_1,\xi_2,\xi_3)\in \mathbb{R}^3$ and
$x=(x_1,x_2,x_3)\in {\bf O}^3$,
put $h'(\xi;x):=\begin{pmatrix}
r_1 &-\sqrt{-1}\overline{x_1} & -\sqrt{-1}\overline{x_2}\\
\sqrt{-1}x_1 & r_2 & x_3\\
\sqrt{-1}x_2 &\overline{x_3} & r_3\\
\end{pmatrix}$.
Denote the 
exceptional $\mathbb{R}$-Jordan algebra
${\rm Herm}'(3,{\bf O}):=\{h'(\xi;x)|~\xi\in\mathbb{R}^3,
x\in{\bf O}^3\}$
with the Jordan product $X\circ Y=2^{-1}(XY+YX)$
$(X,Y\in {\rm Herm}'(3,{\bf O}))$.
Put the linear Lie group $F_4':=
\{g\in{\rm GL}_{\mathbb{R}}({\rm Herm}'(3,{\bf O}))
|~g(X\circ Y)=g X\circ g Y\}$
and the subset $\mathcal{H}'
:=\{A\in{\rm Herm}'(3,{\bf O})|
~A\circ A=A,\mathrm{tr}(A)=1\}$
in ${\rm Herm}'(3,{\bf O})$,
respectively.
F.R.~Harvey \cite[page 296--297]{Hfr1990} 
mentions that $F_4'$ is considered 
to be a simple Lie group of the type of
$\mathbf{F}_{4(-20)}$
and
$F_4'/{\rm Spin}(9)\simeq \mathcal{H}'
=Orb_{F_4'}(E_1)$.
Then $\mathcal{H}'$ consists of
one $F_4'$-orbit. 
We notice that the linear Lie group $\mathrm{F}_4
:=\{g\in{\rm GL}_{\mathbb{R}}(\mathcal{J})|
~g(X\circ Y)=g X\circ g Y\}$ is a compact type of 
$\mathbf{F}_{4(-52)}$
with 
$\mathcal{J}=\{h(\xi,x)|~\xi\in\mathbb{R}^3,
x\in \mathbb{O}^3\}$
and that there exists an isomorphism
$\Phi:\mathcal{J}\rightarrow {\rm Herm}'(3,{\bf O})$
as $\mathbb{R}$-Jordan algebra given as follows
$\Phi(A)
={\rm diag}(-\sqrt{-1},1,1)
~ A~ {\rm diag}(-\sqrt{-1},1,1)^{-1}$.
Therefore, $F_4'$ is a compact type of $\mathbf{F}_{4(-52)}$.
}\end{remark}

\section{The orbit decomposition of $\mathcal{J}^1$
under $\mathrm{F}_{4(-20)}$.}
\label{odp}
In this section, we determine the orbit decomposition
of $\mathcal{J}^1$ under the action of 
$\mathrm{F}_{4(-20)}$.

\begin{lemma}\label{odp-01}
Let $\lambda_1\in\mathbb{R}$.
Then the following equations hold.
\begin{align*}
\tag{\ref{odp-01}.a}
((\lambda_1E-X)^{\times 2})^{\times 2}
&=\Phi_X(\lambda_1)(\lambda_1E-X),\\
\tag{\ref{odp-01}.b}
\mathrm{tr}((\lambda_1E-X)^{\times 2})
&=\left(\frac{d}{d\lambda}\Phi_X\right)(\lambda_1).
\end{align*}
\end{lemma}
\begin{proof}
By (\ref{prl-03}.b)(iv), 
\[((\lambda_1E-X)^{\times 2})^{\times 2}
=\mathrm{det}(\lambda_1E-X)
(\lambda_1E-X)
=\Phi_X(\lambda_1)(\lambda_1E-X).\]
Using
(\ref{prl-03}.a) and (\ref{prl-03}.b),
the left side hand of (\ref{odp-01}.b) is
\[\mathrm{tr}(\lambda_1^2 E^{\times 2}
-2\lambda_1(E\times X)+X^{\times 2})\\
=3\lambda_1^2-2\mathrm{tr}(X)\lambda_1
+\mathrm{tr}(X^{\times 2})
\]
and by (\ref{prl-03}.c), the right side hand 
of (\ref{odp-01}.b) is
$3\lambda_1^2-2\mathrm{tr}(X)\lambda_1
+\mathrm{tr}(X^{\times 2})$.
Thus (\ref{odp-01}.b) follows.
\end{proof}
\medskip

\begin{lemma}\label{odp-02}
Assume that $X\in\mathcal{J}^1$ has 
a characteristic root $\lambda_1\in \mathbb{R}$ 
of multiplicity $1$.
Let $Z=\lambda_1E-X$.

{\rm (1)} The following assertions hold.
\[
\tag{\ref{odp-02}.a}
\left\{
\begin{array}{rl}
{\rm (i)}&
\mathrm{det}(Z)=0,\quad {\rm (ii)}~~
\mathrm{tr}(Z^{\times 2}) \ne 0,\\
\smallskip
{\rm (iii)}&
(Z^{\times 2})^{\times 2}=0,\\
\smallskip
{\rm (iv)}&
(Z^{\times 2})\times Z
=2^{-1}(-\mathrm{tr}(Z)Z^{\times 2}
-\mathrm{tr}(Z^{\times 2})Z
+\mathrm{tr}(Z^{\times 2})\mathrm{tr}(Z)E).
\end{array}
\right.
\]

{\rm (2)} The following equations hold.
\begin{gather*}
\tag{\ref{odp-02}.b}
\left\{
\begin{array}{rl}
{\rm (i)}&
E_{X,\lambda_1}=(\mathrm{tr}(Z^{\times 2}))^{-1}
Z^{\times 2},\\
\smallskip
{\rm (ii)}&
W_{X,\lambda_1}=(\mathrm{tr}(Z)/2)E-Z
-\left(\mathrm{tr}(Z)/(2\mathrm{tr}(Z^{\times 2}))\right)
Z^{\times 2},\\
\smallskip
{\rm (iii)}&
W_{X,\lambda_1}^{\times 2}
=\left((4\mathrm{tr}(Z^{\times 2})-\mathrm{tr}(Z)^2)
/(4\mathrm{tr}(Z^{\times 2}))\right)Z^{\times 2}.
\end{array}
\right.\\
\tag{\ref{odp-02}.c}
Q(W_{X,\lambda_1})
=-4^{-1}(4\mathrm{tr}(Z^{\times 2})-\mathrm{tr}(Z)^2).
\end{gather*}
\end{lemma}
\begin{proof}
(1) Since $\lambda_1$
is a characteristic root of $X$
with multiplicity $1$,
we have
$\mathrm{det}(Z)=\Phi_X(\lambda_1)=0$ and
$\left(\frac{d}{d\lambda}\Phi_X\right)(\lambda_1)
\ne 0$.
Thus (i) follows and 
(ii) follows from (\ref{odp-01}.b).
By (\ref{odp-01}.a), 
$(Z^{\times 2})^{\times 2}=\Phi_X(\lambda_1)Z=0$
so that
(iii) follows.
Moreover, (iv) follows 
from (i) and (\ref{prl-03}.b)(v).

(2) From (\ref{odp-02}.a)(ii),
$E_{X,\lambda_1}$ and 
$W_{X,\lambda_1}$ are well-defined
and (i) and (ii) follow
from direct calculations.
Thus, using Lemma~\ref{prl-03} and
(\ref{odp-02}.a),
we calculate that
\begin{align*}
W_{X,\lambda_1}^{\times 2}&=\left((\mathrm{tr}(Z)/2)E-Z
-\left(\mathrm{tr}(Z)/(2\mathrm{tr}(Z^{\times 2}))\right)
Z^{\times 2}
\right)^{\times 2}\\
&=\left((4\mathrm{tr}(Z^{\times 2})-\mathrm{tr}(Z)^2)
/(4\mathrm{tr}(Z^{\times 2}))\right)Z^{\times 2}
\end{align*}
and so
$Q(W_{X,\lambda_1})=-\mathrm{tr}(W_{X,\lambda_1}^{\times 2})
=-4^{-1}(4\mathrm{tr}(Z^{\times 2})-\mathrm{tr}(Z)^2)$.
\end{proof}
\medskip

\begin{lemma}\label{odp-03}
Assume that $X\in\mathcal{J}^1$ has 
a characteristic root $\lambda_1\in \mathbb{R}$ 
of multiplicity $1$.
Then
\[
\tag{\ref{odp-03}.a}
X=\lambda_1E_{X,\lambda_1}+2^{-1}(\mathrm{tr}(X)-\lambda_1)
(E-E_{X,\lambda_1})+W_{X,\lambda_1}
\]
where
\begin{gather*}
\tag{\ref{odp-03}.b}
E_{X,\lambda_1} \in 
\mathcal{H}({\bf O}) \coprod \mathcal{H}'({\bf O}),\\
\tag{\ref{odp-03}.c}
E-E_{X,\lambda_1} \in 
\mathcal{J}^1_{L^{\times}(2E_{X,\lambda_1}),1},\\
\tag{\ref{odp-03}.d}
W_{X,\lambda_1} \in
\mathcal{J}^1_{L^{\times}(2E_{X,\lambda_1}),-1}.
\end{gather*}
Furthermore
\[
\tag{\ref{odp-03}.e}
Q(W_{X,\lambda_1})
=-4^{-1}(3\lambda_1^2-2\lambda_1\mathrm{tr}(X)
+4\mathrm{tr}(X^{\times 2})-\mathrm{tr}(X)^2).
\]
\end{lemma}
\begin{proof}
Let $Z=\lambda_1E-X$.
First,  
(\ref{odp-03}.a) follows
from (\ref{odp-02}.a)(ii) and (\ref{prl-09}).
From (\ref{odp-02}.b)(i) and (\ref{odp-02}.a)(iii),
we see that
\[E_{X,\lambda_1}^{\times 2}= 
\left(\mathrm{tr}(Z^{\times 2})\right)^{-2}
(Z^{\times 2})^{\times 2}=0,~~
\mathrm{tr}(E_{X,\lambda_1})
=(\mathrm{tr}(Z^{\times 2}))^{-1}
\mathrm{tr}(Z^{\times 2})=1\]
so that $E_{X,\lambda_1} \in \mathcal{H}$,
and from (\ref{hp-01}.a),  (\ref{odp-03}.b) follows.
Second, because of  $E_{X,\lambda_1}\in\mathcal{H}$
and (\ref{prl-03}.b)(ii),
we see that
\[2E_{X,\lambda_1}\times (E-E_{X,\lambda_1})=
2E_{X,\lambda_1}\times E
=\mathrm{tr}(E_{X,\lambda_1})E-E_{X,\lambda_1}
=E-E_{X,\lambda_1}\]
and $E-E_{X,\lambda_1}
\in \mathcal{J}^1_{L^{\times}(2E_{X,\lambda_1}),1}$.
Third, 
by (\ref{odp-02}.b)(i)(ii),
$2E_{X,\lambda_1}\times W_{X,\lambda_1}
=(2/\mathrm{tr}(Z^{\times 2}))Z^{\times 2}
\times 
\left((\mathrm{tr}(Z)/2)E-Z
-(\mathrm{tr}(Z)/(2\mathrm{tr}(Z^{\times 2})))Z^{\times 2}
\right)$.
Then by (\ref{prl-03}.b),
(\ref{odp-02}.a)(iii) and (\ref{odp-02}.a)(iv),
the right hand side is
\[-(\mathrm{tr}(Z)/2)E+Z
+(\mathrm{tr}(Z)/(2\mathrm{tr}(Z^{\times 2})))Z^{\times 2}
=-W_{X,\lambda_1}.\]
Thus $W_{X,\lambda_1}
\in \mathcal{J}^1_{L^{\times}(2E_{X,\lambda_1}),-1}$.
Last,
using (\ref{odp-02}.c) and (\ref{odp-01}.b),
\begin{align*}
Q(W_{X,\lambda_1})
&=-4^{-1}(4(3\lambda_1^2-2\mathrm{tr}(X)\lambda_1
+\mathrm{tr}(X^{\times 2}))
-(3\lambda_1-\mathrm{tr}(X))^2)\\
&=-4^{-1}(3\lambda_1^2-2\lambda_1\mathrm{tr}(X)
+4\mathrm{tr}(X^{\times 2})-\mathrm{tr}(X)^2).
\end{align*}
\end{proof}

\begin{lemma}\label{odp-04}
Assume that $X \in \mathcal{J}^1$ has 
a characteristic root $\lambda_1 \in \mathbb{R}$ 
of multiplicity $1$.
Then $E_{X,\lambda_1}\in
\mathcal{H}({\bf O})$ or  
$E_{X,\lambda_1}\in\mathcal{H}'({\bf O})$
and the following {\rm (i)} or {\rm (ii)} hold.

{\rm (i)} When $E_{X,\lambda_1}\in\mathcal{H}({\bf O})$, then
there exists $g\in \mathrm{F}_{4(-20)}$ such that
\[gX=\lambda_1 E_1+2^{-1}(\mathrm{tr}(X)-\lambda_1)(E-E_1)
+g W_{X,\lambda_1}\]
where 
$g W_{ X,\lambda_1}\in\mathcal{J}^1_{L^{\times}(2E_1),-1}$
and the quadratic space
$(\mathcal{J}^1_{L^{\times}(2E_{X,\lambda_1}),-1},Q)$
is isomorphic to $(\mathbb{R}^{0,9},\mathrm{q}_{0,9})$.
Especially,
$Q(W_{X,\lambda_1})\geq 0$ and
\[Q(W_{X,\lambda_1})=0\quad\text{iff}\quad
W_{X,\lambda_1}=0.\]

{\rm (ii)} When $E_{X,\lambda_1}\in\mathcal{H}'({\bf O})$, then
there exists $g\in \mathrm{F}_{4(-20)}$ such that
\[gX=\lambda_1 E_3+2^{-1}(\mathrm{tr}(X)-\lambda_1)(E-E_3)
+g W_{X,\lambda_1}\]
where
$g W_{X,\lambda_1}\in\mathcal{J}^1_{L^{\times}(2E_3),-1}$
and the quadratic space
$(\mathcal{J}^1_{L^{\times}(2E_{X,\lambda_1}),-1},Q)$
isomorphic to $(\mathbb{R}^{8,1},\mathrm{q}_{8,1})$.
\end{lemma}
\begin{proof}
By Lemma~\ref{odp-03},
\[
X=\lambda_1E_{X,\lambda_1}+2^{-1}(\mathrm{tr}(X)-\lambda_1)
(E-E_{X,\lambda_1})+W_{X,\lambda_1}
\]
where $E_{X,\lambda_1}\in
\mathcal{H}({\bf O})\coprod \mathcal{H}'({\bf O})$
and $W_{X,\lambda_1}
\in\mathcal{J}^1_{L^{\times}(2E_{X,\lambda_1}),-1}$.

Case $E_{X,\lambda_1}\in
\mathcal{H}({\bf O})$.
By (\ref{hp-np}.a),
there exists $g\in \mathrm{F}_{4(-20)}$ 
such that $gE_{X,\lambda_1}=E_1$
and it is clear that
\[g X=
\lambda_1 E_1+2^{-1}(\mathrm{tr}(X)-\lambda_1)(E-E_1)
+g W_{X,\lambda_1}.\]
By (\ref{spn-01}.b), 
$g W_{X,\lambda_1}
\in g\mathcal{J}^1_{L^{\times}(2E_{X,\lambda_1}),-1}
=\mathcal{J}^1_{L^{\times}(2E_1),-1}$.
From Lemma~\ref{spn-01},  we see that
$g$ gives the quadratic isomorphism
from $(\mathcal{J}^1_{L^{\times}(2E_{X,\lambda_1}),-1},Q)$
onto $(\mathcal{J}^1_{L^{\times}(2E_1),-1},Q)$
and 
that the quadratic space
$(\mathcal{J}^1_{L^{\times}(2E_{X,\lambda_1}),-1},Q)$
is isomorphic to $(\mathbb{R}^{0,9},\mathrm{q}_{0,9})$.

Case $E_{X,\lambda_1}\in
\mathcal{H}'({\bf O})$.
By (\ref{hp-np}.b),
there exists $g \in \mathrm{F}_{4(-20)}$ 
such that $g E_{X,\lambda_1}=E_3$
and 
as similar to (i),
we see
\[g X
=\lambda_1 E_3+2^{-1}(\mathrm{tr}(X)-\lambda_1)(E-E_3)
+gW_{X,\lambda_1}\]
with
$g W_{X,\lambda_1}
\in \mathcal{J}^1_{L^{\times}(2E_3),-1}$.
From Lemma~\ref{spn-01}, we see that
$g$ gives the quadratic isomorphism
from $(\mathcal{J}^1_{L^{\times}(2E_{X,\lambda_1}),-1},Q)$
onto $(\mathcal{J}^1_{L^{\times}(2E_3),-1},Q)$
and that
$(\mathcal{J}^1_{L^{\times}(2E_{X,\lambda_1}),-1},Q)$
is isomorphic to $(\mathbb{R}^{8,1},\mathrm{q}_{8,1})$.
\end{proof}
\medskip

\begin{lemma}\label{odp-05}
Assume that $X\in\mathcal{J}^1$ has 
a characteristic root $\lambda_1
\in \mathbb{R}$ of multiplicity $1$
and $Q(W_{X,\lambda_1})>0$.
Then $X$ is diagonalizable under the action
of $\mathrm{F}_{4(-20)}$
on $\mathcal{J}^1$.
\end{lemma}
\begin{proof}
By Lemma~\ref{odp-04},
there exists $g_1\in \mathrm{F}_{4(-20)}$ such that
\[g_1X=\lambda_1 E_i+2^{-1}(\mathrm{tr}(X)-\lambda_1)(E-E_i)
+g_1 W_{X,\lambda_1}\]
where
$g_1 W_{X,\lambda_1}\in\mathcal{J}^1_{L^{\times}(2E_i),-1}$
and $i\in\{1,3\}$.
From (\ref{hp-01}.a), we see
$Q(g_1 W_{X,\lambda_1})=Q(W_{X,\lambda_1})>0$
and therefore $g_1 W_{X,\lambda_1}$ 
can be expressed by
$g_1 W_{X,\lambda_1}=\sqrt{Q(W_{X,\lambda_1})}Y$
where
$Y\in\mathcal{J}^1_{L^{\times}(2E_i),-1}$
with $Q(Y)=1$.
If $i=1$ then $Y\in S^8$,
and if $i=3$ then $Y\in \mathcal{S}^{8,1}
=\mathcal{S}^{8,1}_+\coprod \mathcal{S}^{8,1}_-$.
Then we have
the following three cases:
(i) $i=1$,
(ii) $i=3$ and $Y\in\mathcal{S}^{8,1}_+$,
(iii)
$i=3$ and $Y\in\mathcal{S}^{8,1}_-$.

Case (i): $i=1$.
By (\ref{spn-06}.a), 
there exists
$g_2\in (\mathrm{F}_{4(-20)})_{E_1}$
such that $g_2 Y=E_2-E_3$
and it is clear that
\[g_2g_1X=\lambda_1 E_1
+\sum{}_{j=2}^3\left(2^{-1}(\mathrm{tr}(X)-\lambda_1)
+(-1)^j\sqrt{Q(W_{X,\lambda_1})}\right)E_j.\]

Case (ii): $i=3$ and $Y\in\mathcal{S}^{8,1}_+$.
By (\ref{spn-06}.c)(i), there exists
$g_2'\in (\mathrm{F}_{4(-20)})_{E_3}$
such that $g_2' Y=E_1-E_2$
and it is clear that
\[g_2'g_1X=\lambda_1 E_3+\sum{}_{j=1}^2
\left(2^{-1}(\mathrm{tr}(X)-\lambda_1)
+(-1)^{j+1}\sqrt{Q(W_{X,\lambda_1})}\right)E_j.\]

Case (iii): $i=3$ and $Y\in\mathcal{S}^{8,1}_-$.
By (\ref{spn-06}.c)(ii), there exists
$g_2''\in (\mathrm{F}_{4(-20)})_{E_3}$
such that $g_2'' Y=-E_1+E_2$
and it is clear that
\[g_2''g_1X=\lambda_1 E_3+\sum{}_{j=1}^2
\left(2^{-1}(\mathrm{tr}(X)-\lambda_1)
+(-1)^j\sqrt{Q(W_{X,\lambda_1})}\right)E_j.\]

Thus these cases imply that $X$ can be transformed to
a diagonal matrix under the action of $\mathrm{F}_{4(-20)}$.
\end{proof}
\medskip

\begin{lemma}\label{odp-06}
Assume that $X\in\mathcal{J}^1$ has a characteristic
polynomial
$\Phi_X(\lambda)
=(\lambda-\lambda_1)(\lambda-\lambda_2)(\lambda-\lambda_3)$
where $\lambda_i\in \mathbb{C}$.
Then
\[
\tag{\ref{odp-06}}
\left\{\begin{array}{rlrl}
{\rm (i)}&\mathrm{tr}(X)=\lambda_1+\lambda_2+\lambda_3,&
{\rm (ii)}&\mathrm{tr}(X^{\times 2})
=\lambda_1\lambda_2+\lambda_2\lambda_3+\lambda_3\lambda_1,\\
\smallskip
{\rm (iii)}&
\mathrm{det}(X)=\lambda_1\lambda_2\lambda_3.
\end{array}
\right.
\]
\end{lemma}
\begin{proof}
By (\ref{prl-03}.c), 
$\lambda^3-\mathrm{tr}(X)\lambda^2
+\mathrm{tr}(X^{\times 2})\lambda
-\mathrm{det}(X)=\Phi_X(\lambda)=
\lambda^3-(\lambda_1+\lambda_2+\lambda_3)\lambda^2
+(\lambda_1\lambda_2+\lambda_2\lambda_3
+\lambda_3\lambda_1)\lambda
-\lambda_1\lambda_2\lambda_3$.
Thus the result follows.
\end{proof}
\medskip

\begin{lemma}\label{odp-07}
Assume that $X\in\mathcal{J}^1$ has a characteristic root 
$\lambda_1\in\mathbb{R}$ of multiplicity $1$
and
$\Phi_X(\lambda)
=(\lambda-\lambda_1)(\lambda-\lambda_2)(\lambda-\lambda_3)$
where $\lambda_2,\lambda_3\in \mathbb{C}$.
Then \[
\tag{\ref{odp-07}}
Q(W_{X,\lambda_1})
=4^{-1}(\lambda_2-\lambda_3)^2.
\]
\end{lemma}
\begin{proof}
By (\ref{odp-03}.c) and (\ref{odp-06}),
$
Q(W_{X,\lambda_1})
=-4^{-1}(3\lambda_1^2
-2(\lambda_1+\lambda_2+\lambda_3)\lambda_1\\
+4(\lambda_1\lambda_2+\lambda_2\lambda_3
+\lambda_3\lambda_1)
-(\lambda_1+\lambda_2+\lambda_3)^2)
=4^{-1}(\lambda_2-\lambda_3)^2$.
\end{proof}
\medskip

\begin{lemma}\label{odp-08}  
Let real numbers $r_1$, $r_2$, $r_3$ 
be different from each other and 
$Y={\rm diag}(r_1,r_2,r_3)\in\mathcal{J}^1$.

{\rm (1)} All of 
characteristic roots of $Y$ are $r_1,~r_2,~r_3$.
\smallskip

{\rm (2)} $E_i=E_{Y,r_i}$ for all $i\in\{1,2,3\}$.
\smallskip

{\rm (3)} $V_Y=\{aE_1+bE_2+cE_3|~a,b,c\in\mathbb{R}\}$.
\smallskip

{\rm (4)} $\mathcal{H}({\bf O})\cap V_Y=\{E_{Y,r_1}\}$ and 
$\mathcal{H}'({\bf O})\cap V_Y=\{E_{Y,r_2},~E_{Y,r_3}\}$
with $E_{Y,r_2}\ne E_{Y,r_3}$.
\smallskip

{\rm (5)} For any $g\in \mathrm{F}_{4(-20)}$,
$\mathcal{H}({\bf O})\cap V_{g Y}=\{E_{g Y,r_1}\}$ 
and $\mathcal{H}'({\bf O})\cap V_{g Y}
=\{E_{g Y,r_2},~E_{g Y,r_3}\}$
with $E_{gY,r_2}\ne E_{gY,r_3}$. 
\end{lemma}
\par
\begin{proof}
(1) By (\ref{prl-06}.c), 
$\Phi_Y(\lambda)=(\lambda-r_1)(\lambda-r_2)(\lambda-r_3)$.
Hence (1) follows.

(2) Because of $(r_iE-Y)^{\times 2}
=(r_i-r_{i+1})(r_i-r_{i+2})E_i$, we see 
$E_{Y,r_i}=E_i$.

(3) 
Put the linear space 
$V=\{aE_1+bE_2+cE_3|~a,b,c\in\mathbb{R}\}$.
Because of (2) and
$E_{Y,r_i} \in V_Y$, we see
$E_1,E_2,E_3\in V_Y$
so that $V\subset V_Y$.
Thus it follows from
$3=\mathrm{dim}~V\leq \mathrm{dim}~V_Y\leq 3$.

(4)
Let $Z \in V_Y\cap\mathcal{H}$.
By (3), $Z$ can be expressed by
$Z=aE_1+bE_2+cE_3$
for some $a,b,c \in \mathbb{R}$.
Because of $aE_1+bE_2+cE_3 \in \mathcal{H}$,
we see that
$0=(aE_1+bE_2+cE_3)^{\times 2}
=bcE_1+caE_2+abE_3$ and
$1=\mathrm{tr}(aE_1+bE_2+cE_3)=a+b+c$.
Then $bc=ca=ab=0$ and $a+b+c=1$.
Solving these equations,
$(a,b,c)=(1,0,0)$, $(0,1,0)$, $(0,0,1)$
and it is clear that
$\mathcal{H}\cap V_Y=\{E_1,E_2,E_3\}$.
Thus $\mathcal{H}({\bf O})\cap V_Y=\{E_1\}
=\{E_{Y,r_1}\}$ 
because $(E_i|E_1)\geq 1$ iff $i=1$,
and $\mathcal{H}'({\bf O})\cap V_Y=\{E_2,E_3\}
=\{E_{Y,r_2},E_{Y,r_3}\}$ 
because $(E_i|E_1)\leq 0$ iff $i=2,3$.
Moreover $E_{Y,r_2}=E_2\ne E_3=E_{Y,r_3}$.
Hence (4) follows.

(5) 
From (\ref{hp-np}.a), (4) and (\ref{prl-10})(iii),
it follows that
$\mathcal{H}({\bf O})\cap V_{g Y}
=(g\mathcal{H}({\bf O}))\cap g(V_X)
=g(\mathcal{H}({\bf O})\cap V_Y)
=\{E_{g Y,\lambda_i}\}$ 
and $\mathcal{H}'({\bf O})\cap V_{g Y}
=(g\mathcal{H}'({\bf O}))\cap g(V_X)
=g(\mathcal{H}'({\bf O})\cap V_Y)
=\{E_{g Y,\lambda_{i+1}},~E_{g Y,\lambda_{i+2}}\}$
with $E_{gY,r_2}\ne E_{gY,r_3}$.
\end{proof}
\medskip

\begin{lemma}\label{odp-09}
Let $i \in\{1,2,3\}$ and indexes $i,i+1,i+2$ be
counted modulo $3$.
Let real numbers $\lambda_1$, $\lambda_2$, $\lambda_3$ 
be different from each other
and
$Y_i={\rm diag}(\lambda_i,\lambda_{i+1},\lambda_{i+2})
\in\mathcal{J}^1$.
Then orbits $Orb_{\mathrm{F}_{4(-20)}}(Y_1)$, 
$Orb_{\mathrm{F}_{4(-20)}}(Y_2)$,
$Orb_{\mathrm{F}_{4(-20)}}(Y_3)$
are different orbits from each other.
\end{lemma}
\begin{proof}
It is enough to show 
$Orb_{\mathrm{F}_{4(-20)}}(Y_i)
\ne Orb_{\mathrm{F}_{4(-20)}}(Y_{i+1})$.
Suppose that there exists 
$g \in \mathrm{F}_{4(-20)}$ such that $g Y_i=Y_{i+1}$.
By Lemma~\ref{odp-08}(4),
\begin{align*} 
\mathcal{H}({\bf O})\cap V_{Y_i}&=\{E_{Y_i,\lambda_i}\},& 
\mathcal{H}'({\bf O})\cap V_{Y_i}
&=\{E_{Y_i,\lambda_{i+1}},~E_{Y_i,\lambda_{i+2}}\},\\
\mathcal{H}({\bf O})\cap V_{Y_{i+1}}
&=\{E_{Y_{i+1},\lambda_{i+1}}\},& 
\mathcal{H}'({\bf O})\cap V_{Y_{i+1}}
&=\{E_{Y_{i+1},\lambda_{i+2}},~E_{Y_{i+1},\lambda_{i}}\}.
\end{align*}
In particular,
$E_{Y_i,\lambda_{i+1}} \in \mathcal{H}'({\bf O})$ and
$E_{Y_{i+1},\lambda_{i+1}} \in \mathcal{H}({\bf O})$.
From (\ref{hp-np}.a)
and (\ref{prl-10})(iii), we see
$E_{Y_{i+1},\lambda_{i+1}}=E_{gY_i,\lambda_{i+1}}
=gE_{Y_i,\lambda_{i+1}}\in g \mathcal{H}'({\bf O})
=\mathcal{H}'({\bf O})$.
Then from(\ref{hp-01}.a), it follows that
$E_{Y_{i+1},\lambda_{i+1}}\in \mathcal{H}({\bf O})\cap
\mathcal{H}'({\bf O})=\emptyset$.
It is a contradiction as required.
\end{proof}
\medskip

\begin{proposition}\label{odp-10}
Let $i \in \{1,2,3\}$ and
indexes $i$, $i+1$, $i+2$ be counted modulo $3$.
Assume that $X\in\mathcal{J}^1$ admits 
characteristic roots $\lambda_1>\lambda_2>\lambda_3.$

{\rm (1)}
There exist the unique $i\in \{1,2,3\}$ such that 
\[X\in Orb_{\mathrm{F}_{4(-20)}}
({\rm diag}(\lambda_i,\lambda_{i+1},\lambda_{i+2})).\]

{\rm (2)} The following assertions are equivalent.
\begin{align*}
\quad{\rm (i)}&\quad
X\in Orb_{\mathrm{F}_{4(-20)}}
({\rm diag}(\lambda_i,\lambda_{i+1},\lambda_{i+2})).\\
{\rm (ii)}&\quad
\left\{
\begin{array}{l}
\mathcal{H}({\bf O})\cap V_X=\{E_{X,\lambda_i}\},\\
\mathcal{H}'({\bf O})
\cap V_X=\{E_{X,\lambda_{i+1}},~E_{X,\lambda_{i+2}}\}
~~\text{with}~E_{X,\lambda_{i+1}}\ne E_{X,\lambda_{i+2}}.
\end{array}
\right.
\end{align*}
\end{proposition}
\begin{proof}
(1)
Fix the characteristic root $\lambda_1\in\mathbb{R}$.
Because of $0\neq \lambda_2-\lambda_3\in\mathbb{R}$
and (\ref{odp-07}),
we see
$Q(W_{X,\lambda_1})=4^{-1}(\lambda_2-\lambda_3)^2>0$.
By Lemma~\ref{odp-05},
we see that
$X$ is diagonalizable under 
the action of $\mathrm{F}_{4(-20)}$.
Then from Lemma~\ref{odp-08}(1),
there exists 
$i\in\{1,2,3\}$ and
$g_0 \in \mathrm{F}_{4(-20)}$ such that  
$g_0 X=
{\rm diag}(\lambda_i,\lambda_{i+1},\lambda_{i+2})$
or
$g_0 X=
{\rm diag}(\lambda_i,\lambda_{i+2},\lambda_{i+1})$.
When
$g_0X={\rm diag}(\lambda_i,\lambda_{i+2},\lambda_{i+1})$,
from (\ref{cce-10}.a), we see
$\exp(2^{-1}\pi\tilde{A}_1^1(1))g_0X
={\rm diag}(\lambda_i,\lambda_{i+1},\lambda_{i+2})$.
Thus
$X\in Orb_{F_{4(-20}}({\rm diag}
(\lambda_i,\lambda_{i+1},\lambda_{i+2}))$
and by Lemma~\ref{odp-09}, such $i$ is unique.
Hence (1) follows.

(2) We show (i) $\Rightarrow$ (ii).
Then it implies that (i) and (ii) are equivalent
(a proof by transposition).
Suppose
$gX={\rm diag}(\lambda_i,\lambda_{i+1},\lambda_{i+2})$ 
for some $g\in \mathrm{F}_{4(-20)}$.
Since $gX$ is diagonal matrix, 
from Lemma~\ref{odp-08}(4),
we see
$\mathcal{H}({\bf O})\cap V_{gX}
=\{E_{gX,\lambda_i}\}$ and
$\mathcal{H}'({\bf O})\cap V_{gX}
=\{E_{gX,\lambda_{i+1}},~E_{gX,\lambda_{i+2}}\}$
with $E_{gX,\lambda_{i+1}}\ne E_{gX,\lambda_{i+2}}$.
Then from Lemma~\ref{odp-08}(5), it follows that
$\mathcal{H}({\bf O})\cap V_{X}
=\mathcal{H}({\bf O})\cap V_{g^{-1}gX}
=\{E_{g^{-1}gX,\lambda_i}\}=\{E_{X,\lambda_i}\}$
and 
$\mathcal{H}'({\bf O})\cap V_{X}
=\mathcal{H}'({\bf O})\cap V_{g^{-1}gX}
=\{E_{X,\lambda_{i+1}},~E_{X,\lambda_{i+2}}\}$
with $E_{X,\lambda_{i+1}}\ne E_{X,\lambda_{i+2}}$.
\end{proof}
\medskip

\begin{proposition}\label{odp-11}
Assume that $X\in\mathcal{J}^1$ admits 
characteristic roots $\lambda_1\in\mathbb{R},$ 
$p\pm \sqrt{-1}q$ with $p\in\mathbb{R}$ and $q>0$.
Then 
\[X\in Orb_{\mathrm{F}_{4(-20)}}
({\rm diag}(p,p,\lambda_1)+F^1_3(q)).\]
\end{proposition}
\begin{proof}
Since $\lambda_1$
is a characteristic root in $\mathbb{R}$ of multiplicity $1$,
from (\ref{odp-07}), we see
$Q(W_{X,\lambda_1})
=4^{-1}((p+\sqrt{-1}q)-(p-\sqrt{-1}q))^2
=-q^2<0$.
From Lemma~\ref{odp-04}, we see that
$E_{X,\lambda_1}$ must be in $\mathcal{H}'({\bf O})$
and
there exists $g_1\in \mathrm{F}_{4(-20)}$ such that
\[g_1X=\lambda_1 E_3+2^{-1}(\mathrm{tr}(X)-\lambda_1)(E-E_3)
+g_1 W_{X,\lambda_1}\]
where
$g_1 W_{X,\lambda_1}\in\mathcal{J}^1_{L^{\times}(2E_3),-1}$
and $Q(g_1 W_{X,\lambda_1})=Q(W_{X,\lambda_1})=-q^2<0$.
By (\ref{odp-06})(i),
$\mathrm{tr}(X)=\lambda_1+(p+\sqrt{-1}q)+(p-\sqrt{-1}q)
=\lambda_1+2p$ and therefore
$g_1X$ can be expressed by
\[g_1X=\lambda_1 E_3+p(E-E_3)
+qY\quad
\text{for~some}~Y\in\mathcal{S}^{8,1}(-1).\]
Now by (\ref{spn-06}.b),
there exists $g_2\in (\mathrm{F}_{4(-20)})_{E_3}$
such that $g_2Y=F_3^1(1)$
and it is clear that
\[g_2g_1X=\lambda_1 E_3+p(E-E_3)
+F_3^1(q)
=\mathrm{diag}(p,p,\lambda_1)+F_3^1(q).\]
\end{proof}

\begin{lemma}\label{odp-12}
\[\mathcal{N}_1^+({\bf O})\cap
\mathcal{J}^1_{L^{\times}(2E_3),-1}=\mathcal{N}^{8,1}_+,
\quad
\mathcal{N}_1^-({\bf O})\cap
\mathcal{J}^1_{L^{\times}(2E_3),-1}=\mathcal{N}^{8,1}_-.\]
\end{lemma}
\begin{proof}
Take $X\in\mathcal{N}_1^+({\bf O})\cap
\mathcal{J}^1_{L^{\times}(2E_3),-1}$.
Because of $X\in\mathcal{N}_1^+({\bf O})$,
$Q(X^{\times 2})=0$ and $(X|E_1)>0$.
Thus,
$X\in\mathcal{N}^{8,1}_+$
and so
$\mathcal{N}_1^+({\bf O})\cap
\mathcal{J}^1_{L^{\times}(2E_3),-1}\subset
\mathcal{N}^{8,1}_+$.
Conversely, take $X\in\mathcal{N}^{8,1}_+$.
Then
$X$ can be expressed by
$X=\xi(E_1-E_2)+F_3^1(x)$
where $\xi=(X|E_1)>0$, $x\in{\bf O}$
and $0=Q(X)=\xi^2-(x|x)$.
Now $X^{\times 2}=-(\xi^2-(x|x))E_3=0$,
$\mathrm{tr}(X)=0$ and $(X|E_1)>0$.
Thus $X\in\mathcal{N}_1^+({\bf O})\cap
\mathcal{J}^1_{L^{\times}(2E_3),-1}$
and so $\mathcal{N}^{8,1}_+\subset 
\mathcal{N}_1^+({\bf O})\cap
\mathcal{J}^1_{L^{\times}(2E_3),-1}$.
Hence $\mathcal{N}_1^+({\bf O})\cap
\mathcal{J}^1_{L^{\times}(2E_3),-1}=\mathcal{N}^{8,1}_+$.
Similarly, we obtain $\mathcal{N}_1^-({\bf O})\cap
\mathcal{J}^1_{L^{\times}(2E_3),-1}=\mathcal{N}^{8,1}_-$.
\end{proof}
\medskip

\begin{proposition}\label{odp-13}
Assume that $X\in\mathcal{J}^1$ 
admits characteristic roots $\lambda_1$ 
of multiplicity $1$ and $\lambda_2$ of multiplicity $2.$
Then we have
\begin{align*}
{\rm (i)}&~E_{X,\lambda_1}\in\mathcal{H}({\bf O})
\coprod \mathcal{H}'({\bf O}),~~
{\rm (ii)}~
W_{X,\lambda_1}\in \{0\}\coprod \mathcal{N}_1^+({\bf O})
\coprod \mathcal{N}_1^-({\bf O}),\\
{\rm (iii)}&~E_{X,\lambda_1}\in\mathcal{H}({\bf O})
\Rightarrow W_{X,\lambda_1}=0
\quad
{\rm (iv)}~
W_{X,\lambda_1}\ne 0 
\Rightarrow E_{X,\lambda_1}\in\mathcal{H}'({\bf O})
\end{align*}
and the following assertions hold:

{\rm (1)} $E_{X,\lambda_1}\in\mathcal{H}({\bf O})\Rightarrow 
X\in Orb_{\mathrm{F}_{4(-20)}}
({\rm diag}(\lambda_1,\lambda_2,\lambda_2))$,
\smallskip

{\rm (2)} $E_{X,\lambda_1}
\in\mathcal{H}'({\bf O}) 
,~W_{X,\lambda_1}=0\Rightarrow 
X\in Orb_{\mathrm{F}_{4(-20)}}
({\rm diag}(\lambda_2,\lambda_2,\lambda_1))$,
\smallskip

{\rm (3)} $W_{X,\lambda_1}
\in\mathcal{N}_1^+({\bf O})\Rightarrow
X\in Orb_{\mathrm{F}_{4(-20)}}
({\rm diag}(\lambda_2, \lambda_2, \lambda_1)+P^+)$,
\smallskip

{\rm (4)} $W_{X,\lambda_1}\in\mathcal{N}_1^-({\bf O})\Rightarrow
X\in Orb_{\mathrm{F}_{4(-20)}}
({\rm diag}(\lambda_2, \lambda_2, \lambda_1)+P^-)$.
\end{proposition}
\begin{proof}
(i) follows from (\ref{odp-03}.b).
Since $\lambda_1$
is a characteristic root in $\mathbb{R}$ of multiplicity $1$,
form (\ref{odp-07}), we see
$Q(W_{X,\lambda_1})=4^{-1}(\lambda_2-\lambda_2)^2
=0$.
Put $Z=\lambda_1E-X$.
From (\ref{odp-02}.b)(iii) and (\ref{odp-02}.c),
we see
\[W_{X,\lambda_1}^{\times 2}
=-(Q(W_{X,\lambda_1})/\mathrm{tr}(Z^{\times 2}))
Z^{\times 2}=0.\]
Now, because of 
$\mathrm{tr}(E_{X,\lambda_1})=1$ 
and $\mathrm{tr}(E)=3$, we see 
\[\mathrm{tr}(W_{X,\lambda_1})
=\mathrm{tr}
\left(X-\left(\lambda_1E_{X,\lambda_1}
+2^{-1}(\mathrm{tr}(X)-\lambda_1)
(E-E_{X,\lambda_1})\right)\right)=0.\]
Thus $W_{X,\lambda_1}\in\{0\}
\coprod\mathcal{N}_1({\bf O})$.
Hence (ii) follows from (\ref{hp-01}.b),
and (iii) follows from
$Q(W_{X,\lambda_1})=0$ and Lemma~\ref{odp-04}(i).
Moreover, 
it follows from (i) that 
(iv) is the contrapositive proposition
of (iii).

(1) 
By (iii), $W_{X,\lambda_1}=0$.
Hence
it follows from Lemma~\ref{odp-04}(i). 

(2) It follows from Lemma~\ref{odp-04}(ii).

(3) From (iv), we see
$E_{X,\lambda_1}\in\mathcal{H}'({\bf O})$
and from Lemma~\ref{odp-04},
there exists $g_1\in \mathrm{F}_{4(-20)}$ such that
\[g_1X=\lambda_1 E_3+2^{-1}(\mathrm{tr}(X)-\lambda_1)(E-E_3)
+g_1 W_{X,\lambda_1}\]
where
$g_1 W_{X,\lambda_1}\in\mathcal{J}^1_{L^{\times}(2E_3),-1}$.
By (\ref{hp-np}.d),
$g_1 W_{X,\lambda_1}
\in g_1\mathcal{N}_1^+({\bf O})
=\mathcal{N}_1^+({\bf O})$
and by Lemma~\ref{odp-12},
$g_1 W_{X,\lambda_1}\in\mathcal{N}_1^+({\bf O})
\cap\mathcal{J}^1_{L^{\times}(2E_3),-1}=\mathcal{N}^{8,1}_+$.
From (\ref{spn-06}.d)(i),
there exists $g_2\in (\mathrm{F}_{4(-20)})_{E_3}$
such that $g_2g_1 W_{X,\lambda_1}=P^+$.
Now by (\ref{odp-06})(i),
$\mathrm{tr}(X)=\lambda_1+2\lambda_2$.
Thus
\[g_2g_1X=\lambda_1 E_3
+2^{-1}(\lambda_1+2\lambda_2-\lambda_1)(E-E_3)
+P^+=\mathrm{diag}(\lambda_2,\lambda_2,\lambda_1)+P^+.\] 
Hence (3) follows and
similarly, (4) follows.
\end{proof}
\medskip

\begin{proposition}\label{odp-14}
Assume that $X\in\mathcal{J}^1$ 
admits a characteristic root of multiplicity $3$.
Then
\[p(X)\in\{0\}\coprod
\mathcal{N}_1^+({\bf O})\coprod\mathcal{N}_1^-({\bf O})
\coprod\mathcal{N}_2({\bf O})\]
and the following assertions hold:

{\rm (1)} $p(X)=0\Rightarrow 
X\in Orb_{\mathrm{F}_{4(-20)}}(3^{-1}\mathrm{tr}(X)E)$,
\smallskip

{\rm (2)} $p(X)\in\mathcal{N}_1^+({\bf O}) 
\Rightarrow 
X\in Orb_{\mathrm{F}_{4(-20)}}(3^{-1}\mathrm{tr}(X)E+P^+)$,
\smallskip

{\rm (3)} $p(X)\in\mathcal{N}_1^-({\bf O}) 
\Rightarrow 
X\in Orb_{\mathrm{F}_{4(-20)}}
(3^{-1}\mathrm{tr}(X)E+P^-)$,
\smallskip

{\rm (4)} $p(X)\in\mathcal{N}_2({\bf O}) 
\Rightarrow 
X\in Orb_{\mathrm{F}_{4(-20)}}(3^{-1}\mathrm{tr}(X)E+Q^+(1))$.
\end{proposition}
\begin{proof}
Put $Z=p(X)$.
Because of $\Phi_X(\lambda)=(\lambda-3^{-1}\mathrm{tr}(X))^3$,
we see 
$\Phi_Z(\mu)
=\mathrm{det}((\mu+3^{-1}\mathrm{tr}(X))E-X)
=\Phi_X(\mu+3^{-1}\mathrm{tr}(X))
=\mu^3$.
From (\ref{prl-03}.c), we see $\mathrm{tr}(Z)
=\mathrm{tr}(Z^{\times 2})=\mathrm{det}(Z)=0$.
Then by (\ref{hp-03}),
$Z\in\mathcal{N}=
\{0\}\coprod
\mathcal{N}_1^+({\bf O})\coprod\mathcal{N}_1^-({\bf O})
\coprod\mathcal{N}_2({\bf O})$. 
Now 
$X=3^{-1}\mathrm{tr}(X)E+p(X)$. 
Since $E$ 
is invariant under the action of $\mathrm{F}_{4(-20)}$,
(1),(2),(3) and (4) follows from
(\ref{hp-np}.c), (\ref{hp-np}.d)
and (\ref{hp-np}.e).
\end{proof}
\bigskip

\begin{proof}[{\bf Proof 
of Main Theorem~\ref{orb-decomposition}}]
Fix $X\in\mathcal{J}^1$.
$\Phi_X(\lambda)$ is 
a $\mathbb{R}$-coefficient polynomial of 
$\lambda$ with degree $3$, 
so that we obtain the following cases 
(1)-(4) 
by means of the set of all characteristic 
roots with multiplicities.

(1) $X\in\mathcal{J}^1$ 
admits characteristic roots 
$\lambda_1>\lambda_2>\lambda_3$.

(2) $X\in\mathcal{J}^1$ 
admits characteristic roots $\lambda_1\in\mathbb{R}$
and $p\pm \sqrt{-1}q$ with $p\in\mathbb{R}$  and $q>0$.

(3) $X\in\mathcal{J}^1$ 
admits characteristic roots  $\lambda_1$ 
of multiplicity $1$ and $\lambda_2$ of multiplicity $2$.

(4) $X\in\mathcal{J}^1$ 
admits a characteristic root of multiplicity $3$.

Because the set of all characteristic roots with multiplicities 
are invariant under the action of $\mathrm{F}_{4(-20)}$, 
the difference of set of all characteristic roots 
with multiplicities induces 
the difference of $\mathrm{F}_{4(-20)}$-orbits 
in $\mathcal{J}^1$.
Therefore cases (1)-(4) 
are different cases of orbits from each other.

In case (1), by Proposition~\ref{odp-10},
we obtain the propositions (1)(i)-(ii) 
and the canonical forms.
By Lemma~\ref{odp-09}, 
the three cases are different orbits from each other.

In case (2), by Proposition~\ref{odp-11}, 
we obtain  the canonical form of $X$.

In case (3), from Proposition~\ref{odp-13}, 
we obtain the propositions (3)(i)-(iv)
and
the canonical forms.
Put 
$\mathcal{O}=\mathcal{N}_1^+({\bf O})$, 
$\mathcal{N}_1^+({\bf O})$ or $\{0\}$.
From (\ref{prl-10}) and Proposition~\ref{hp-np},
it follows that
if $W_{X,\lambda_1}\in \mathcal{O}$ then
\[W_{gX,\lambda_1}=gW_{X,\lambda_1}\in 
g\mathcal{O}=\mathcal{O}\quad
\text{for~all}~g\in \mathrm{F}_{4(-20)}.\]
Thus the condition $W_{X,\lambda_1}\in \mathcal{O}$ 
is invariant
under the action of $\mathrm{F}_{4(-20)}$.
Similarly, 
from (\ref{prl-10}) and Proposition~\ref{hp-np},
the condition 
$E_{X,\lambda_1}\in \mathcal{H}({\bf O})$ 
or $E_{X,\lambda_1}\in \mathcal{H}'({\bf O})$
is invariant
under the action of $\mathrm{F}_{4(-20)}$.
It implies that these four cases are 
different orbits from each other.

In case (4), by Proposition~\ref{odp-14},
we obtain the proposition
and the canonical forms.
Put 
$\mathcal{O}'=\mathcal{N}_1^+({\bf O})$, 
$\mathcal{N}_1^+({\bf O})$,
$\mathcal{N}_2({\bf O})$
or $\{0\}$.
From (\ref{prl-10}) and Proposition~ \ref{hp-np},
it follows that
if $p(X)\in \mathcal{O}'$ then
\[p(gX)=gp(X)\in g\mathcal{O}'=\mathcal{O}'\quad
\text{for~all}~g\in \mathrm{F}_{4(-20)}.\]
Thus the condition $p(X)\in \mathcal{O}'$ is invariant
under the action of $\mathrm{F}_{4(-20)}$. 
It implies that 
these four cases are different orbits 
from each other.

Hence we obtain 
a concrete orbit decomposition of $\mathcal{J}^1$
under the action of $\mathrm{F}_{4(-20)}$.
\end{proof}

\section{The construction 
of nilpotent subgroup.}\label{nlp}
In this section, we explain the construction of nilpotent
groups $N^+$ and $N^-$.
The differential 
$d\tilde{\sigma}_i
\in{\rm Aut}_{\mathbb{R}}(\mathfrak{f}_{4(-20)})$
of the involutive automorphism $\tilde{\sigma}_i$
is written by same letter $\tilde{\sigma}_i:$
$\tilde{\sigma}_i\phi=\sigma_i\phi\sigma_i$
for $\phi\in\mathfrak{f}_{4(-20)}.$

\begin{lemma}\label{nlp-01} 
Let $i \in \{1,2,3\}$ and
indexes $i,i+1,i+2$ be counted modulo $3$.
Let
$D \in\mathfrak{d}_4$ 
and $a,b \in {\bf O}$.

{\rm (1)} The following equations hold.
\[
\tag{\ref{nlp-01}.a}
\left
\{\begin{array}{rlll}
{\rm (i)}&\tilde{\sigma}_iD=D,&
{\rm (ii)}&\tilde{\sigma}_i\tilde{A}_i^1(a)
=\tilde{A}_i^1(a),\\
\smallskip
{\rm (iii)}&
\tilde{\sigma}_i\tilde{A}_j^1(a)=-\tilde{A}_j^1(a)
&\multicolumn{2}{l}{{\rm for}~j=i+1,i+2.}
\end{array}\right.\]
{\rm (2)} The following equations hold.
\[\tag{\ref{nlp-01}.b}
\left\{
\begin{array}{rlll}{\rm (i)}&
[ D, \tilde{A}_i^1(a) ] = \tilde{A}_i^1(D_ia),&
{\rm (ii)}&
[ \tilde{A}_i^1(a), \tilde{A}_i^1(b) ] 
\in \mathfrak{d}_4,\\
\smallskip
{\rm (iii)}&\multicolumn{3}{l}{
[ \tilde{A}_i^1(a), \tilde{A}_{i+1}^1(b) ]
= \epsilon(i+2)\tilde{A}_{i+2}^1(\overline{ab})}
\end{array}\right.
\]
where $D = d\varphi_0(D_1,D_2,D_3)\in\mathfrak{d}_4$
{\rm (see Lemma~\ref{cce-05}(3))}.
\end{lemma}
\begin{proof}
Fix $X \in \{E_i,F_i^1(x)~|~
x\in{\bf O},~i=1,2,3\}$.

(1) From Lemma~\ref{cce-05}(3),
$D$ can be expressed by
$D = d\varphi_0(D_1,D_2,D_3)$.
Using (\ref{cce-05}), 
we can show 
$\sigma_i D \sigma_i X=D X$
on $\mathcal{J}^1$.
Then from (\ref{prl-05}.a),
(\ref{nlp-01}.a)(i) follows.
From (\ref{cce-09}.b), showing
$\sigma_i \tilde{A}_i(a)\sigma_i X=\tilde{A}_i(a) X$
on $\mathcal{J}^1$,
(\ref{nlp-01}.a)(ii) follows.
Similarly, (\ref{nlp-01}.a)(iii) follows.

(2)
From
(\ref{cce-09}.b) and (\ref{cce-05}),
showing
$[D, \tilde{A}_i^1(a)] X=\tilde{A}_i^1(D_ia)X$
on $\mathcal{J}^1$,
(\ref{nlp-01}.b)(i) follows.
From
(\ref{cce-09}.b),  showing
$[\tilde{A}_i^1(a), \tilde{A}_{i+1}^1(b)] X
=\epsilon(i+2)\tilde{A}_{i+2}^1(\overline{ab})X$ 
on $\mathcal{J}^1$
and
$[\tilde{A}_i^1(a), \tilde{A}_i^1(b)]E_k=0$
with $k\in\{1,2,3\}$,
we obtain
(\ref{nlp-01}.b)(ii)(iii).
\end{proof}
\medskip

\begin{lemma}\label{nlp-02}
The following assertions hold.

{\rm (1) (\cite[Theorem 2.5.3]{Yi_arxiv})}.
The Killing form $B$ of $\mathfrak{f}_4^{\mathbb{C}}$
is given by
\[
\tag{\ref{nlp-02}.a}
B(\phi_1,\phi_2)=3\mathrm{tr}(\phi_1\phi_2)
\quad\text{for}~
\phi_i\in\mathfrak{f}_4^{\mathbb{C}}.\]
Especially, the Killing form 
$B$ of $\mathfrak{f}_{4(-20)}=
(\mathfrak{f}_4^{\mathbb{C}})_{\widetilde{\tau\sigma}}$
is given by
the restriction 
$B|(\mathfrak{f}_{4(-20)}\times \mathfrak{f}_{4(-20)}
)$.
\smallskip

{\rm (2)} Let $\phi=d\varphi_0(D_1,D_2,D_3)
+\sum_{i=1}^3\tilde{A}_i^1(a_i)$
where $d\varphi_0(D_1,D_2,D_3)\in\mathfrak{d}_4$
and $a_i\in{\bf O}.$ Then
\[
\tag{\ref{nlp-02}.b}
B(\phi,\tilde{\sigma}\phi)
=-3\left(\sum{}_{i=1}^3\left((\sum{}_{j=0}^7
(D_ie_j|D_ie_j))+24(a_{i}|a_{i})\right)\right).\]
Furthermore, $\tilde{\sigma}$ is a Cartan involution 
of $\mathfrak{f}_{4(-20)}.$
\end{lemma} 
\begin{proof} (2)  By (\ref{nlp-01}.a),
$\tilde{\sigma}\phi=d\varphi_0(D_1,D_2,D_3)+\sum_{i=1}^3
\epsilon(i)\tilde{A}_i^1(a_i)$.
Since $\{E_i,F_i^1(e_j)|~i=1,2,3,~j=0,\cdots, 7\}$
is a basis of $\mathcal{J}^1$,
using (\ref{nlp-02}.a), 
$B(\phi,\tilde{\sigma}\phi)
=3\left(\sum{}_{i=1}^3(\phi(\tilde{\sigma}\phi)E_i|E_i) 
+\sum{}_{i=1}^3\sum{}_{j=0}^7
((\phi(\tilde{\sigma}\phi)F_i^1(e_j))_{F_i^1}|e_j)\right)$
and from 
(\ref{cce-05}) and (\ref{cce-09}.b), 
we see that 
$(\phi(\tilde{\sigma}\phi)E_i|E_i)
=-2((a_{i+1}|a_{i+1})+(a_{i+2}|a_{i+2}))$ and
$((\phi(\tilde{\sigma}\phi)F_i^1(e_j))_{F_i^1}|e_j)
=-(D_ie_j|D_ie_j)-4(a_i|e_j)^2
-\sum_{j=i+1}^{i+2}(a_j|a_j)$
where indexes $i,i+1,i+2$ are counted modulo $3$.
Thus (\ref{nlp-02}.b) follows.
Moreover,
the bilinear form $B(\phi_1,\tilde{\sigma}\phi_2)$
$(\phi_1,\phi_2\in \mathfrak{f}_{4(-20)})$
is negative definite.
Hence the result follows.
\end{proof}
\medskip

Denote
$\mathfrak{k}:=\{\phi\in\mathfrak{f}_{4(-20)}
|~\tilde{\sigma}\phi=\phi\}=Lie(K)$ and
$\mathfrak{p}:=\{\phi\in\mathfrak{f}_{4(-20)}
|~\tilde{\sigma}\phi=-\phi\}$.
By Lemma~\ref{nlp-02}(2),
\[\mathfrak{f}_{4(-20)}=\mathfrak{k}\oplus \mathfrak{p}
\quad
 \text{(Cartan~decomposition)}.\]
From (\ref{nlp-01}.a)(iii), we see
$\tilde{A}_3^1(1)\in\mathfrak{p}$.
The abelian subspace $\mathfrak{a}$ of $\mathfrak{p}$ 
and the linear functional $\alpha$ on $\mathfrak{a}$ 
are defined as
\[\mathfrak{a}:=\{t\tilde{A}_3^1(1)
|~t\in\mathbb{R}\},\quad \alpha(\tilde{A}_3^1(1)):=1\]
respectively.
Let $\mathfrak{a}_{\mathfrak{p}}$ be
a maximal 
abelian subspace of $\mathfrak{p}$
such that $\mathfrak{a}
\subset\mathfrak{a}_{\mathfrak{p}},$
and denote 
the dual space of $\mathfrak{a}$ 
({\it resp}. $\mathfrak{a}_{\mathfrak{p}}$)
as $\mathfrak{a}^*$ 
({\it resp}. $\mathfrak{a}_{\mathfrak{p}}^*$).
For $\lambda\in \mathfrak{a}^*$
({\it resp}. $\mathfrak{a}_{\mathfrak{p}}^*$),
denote
\begin{gather*}
\mathfrak{g}_{\lambda}:=\{\phi\in \mathfrak{f}_{4(-20)}|~
[H,\phi]=\lambda(H)\phi
~~{\rm for~all}~H\in\mathfrak{a}\}\\
(resp.\quad
\mathfrak{g}_{\lambda}(\mathfrak{a}_{\mathfrak{p}})
:=\{\phi\in \mathfrak{f}_{4(-20)}|~
[H,\phi]=\lambda(H)\phi
~~{\rm for~all}~H\in\mathfrak{a}_{\mathfrak{p}}\})
\end{gather*}
and denote
\begin{gather*}
\Sigma:=\{\lambda\in \mathfrak{a}^*|~
\lambda\ne 0,~\mathfrak{g}_{\lambda}\ne \{0\}\}\\
(resp.\quad
\Sigma(\mathfrak{a}_{\mathfrak{p}})
:=\{\lambda\in \mathfrak{a}_{\mathfrak{p}}^*|~
\lambda\ne 0,
~\mathfrak{g}_{\lambda}(\mathfrak{a}_{\mathfrak{p}})
\ne \{0\}\}).
\end{gather*}
Moreover, the centralizer 
of $\mathfrak{a}$
of the group $K$ and its Lie algebra as
\begin{align*}
M&:=Z_K(\mathfrak{a})
=\{k\in K|~k \tilde{A}_3^1(1)k^{-1}=\tilde{A}_3^1(1)\},\\
\mathfrak{m}&:=Z_{\mathfrak{k}}(\mathfrak{a})=
\{\phi\in\mathfrak{k}~|~
[\phi,\tilde{A}_3^1(1)]=0\}
\end{align*}
respectively.
For all $p\in{\rm Im}{\bf O},$
the elements $l_p,r_p,t_p
\in{\rm End}_{\mathbb{R}}({\bf O})$ are defined by
\[l_px:=px,\quad r_px:=xp,\quad t_px:=(l_p+r_p)x=px+xp
\quad\quad\text{for}~x\in{\bf O}\]
respectively.
Because of $\overline{p}=-p$ and 
(\ref{prl-01}.e), we see
$(l_px|y)=-(x|l_py)$
so that $D_1\in\mathfrak{D}_4$.
Similarly $l_p,t_p\in\mathfrak{D}_4.$
By (\ref{prl-01}.j), 
\[l_p(x)y+xr_p(y)=(px)y+x(yp)=p(xy)+(xy)p
=\epsilon t_{-p}\epsilon(xy).\]
From Lemma~\ref{cce-05}(3), the element
$\delta(p)\in\mathfrak{d}_4$ is defined by
\[\delta(p)
:=d\varphi_0(l_p,r_p,t_{-p}).\]
For $p\in{\rm Im}{\bf O}$ and $x\in{\bf O},$ 
denote
\begin{align*}
\mathcal{G}_1(x)&:=
\tilde{A}_1^1(x)+\tilde{A}_2^1(-\overline{x}),
&\mathcal{G}_2(p)&:=\tilde{A}_3^1(-p)-\delta(p),\\
\mathcal{G}_{-1}(x)&:=
\tilde{A}_1^1(x)+\tilde{A}_2^1(\overline{x}),
&\mathcal{G}_{-2}(p)&:=\tilde{A}_3^1(p)-\delta(p)
\end{align*}
and
the subspaces $\mathfrak{g}_{\pm 1}$,
$\mathfrak{g}_{\pm 2}$ of  
$\mathfrak{f}_{4(-20)}$ as
\begin{align*}
\mathfrak{g}_1
&:=\{\mathcal{G}_1(x)|~x\in{\bf O}\},
&\mathfrak{g}_2
&:=\{\mathcal{G}_2(p)|~p\in{\rm Im}{\bf O}\},\\
\mathfrak{g}_{-1}
&:=\{\mathcal{G}_{-1}(x)|~x\in{\bf O}\},
&\mathfrak{g}_{-2}
&:=\{\mathcal{G}_{-2}(p)|~p\in{\rm Im}{\bf O}\}
\end{align*}
respectively.
\begin{lemma}\label{nlp-03} 
Let $p\in{\rm Im}{\bf O}$
and $x\in{\bf O}$.

{\rm (1)} 
$\mathfrak{g}_i\subset \mathfrak{g}_{i\alpha}$
for $i\in\{\pm 1,\pm 2\}.$
Especially, $\{\pm \alpha,\pm 2\alpha\}\subset
\Sigma.$
\smallskip

{\rm (2)} 
$[\mathfrak{g}_{i\alpha},
\mathfrak{g}_{j\alpha}]
=\mathfrak{g}_{(i+j)\alpha}.$
\end{lemma}
\begin{proof}
(1) Using (\ref{nlp-01}.b)(iii),
we calculate that
\[[\tilde{A}_3^1(1),\tilde{A}_1^1(x)+
\tilde{A}_2^1(\mp\overline{x})]=
\pm(\tilde{A}_1^1(x)+
\tilde{A}_2^1(\mp \overline{x}))\quad
(resp)\]
with $x\in{\bf O}$.
Thus  $\mathfrak{g}_{\pm 1}
\subset \mathfrak{g}_{\pm \alpha}~(resp)$.
Fix $p\in{\rm Im}{\bf O}$.
By (\ref{nlp-01}.b)(ii),
$[\tilde{A}_3^1(1),\tilde{A}_3^1(p)]E_k
=0$
with
$k \in \{1,2,3\}$,
and because of (\ref{cce-09}.b),
$\overline{p}=-p$
and $4(p|x)-4(1|x)p=2(x\overline{p}+p\overline{x})
-2p(x+\overline{x})=-2(px+xp)$,
we see 
\begin{gather*}
[\tilde{A}_3^1(1),\tilde{A}_3^1(p)]F_1^1(x)
=F_1^1(2px),\quad
[\tilde{A}_3^1(1),\tilde{A}_3^1(p)]F_2^1(x)
=F_2^1(2xp),\\
[\tilde{A}_3^1(1),\tilde{A}_3^1(p)]F_3^1(x)
=F_3^1(4(p|x)-4(1|x)p)
=F_3^1(-2(px+xp)).
\end{gather*}
Then from (\ref{prl-05}.a),
we see
$[\tilde{A}_3^1(1),\tilde{A}_3^1(p)]
=2\delta(p)$ on $\mathcal{J}^1$.
Thus, from (\ref{nlp-01}.b)(i), we have
\begin{gather*}
[\tilde{A}_3^1(1),\tilde{A}_3^1(-p)
-\delta(p)]=
-2\delta(p)+\tilde{A}_3^1(t_{-p}1)=
2(\tilde{A}_3^1(-p)
-\delta(p)),\\
[\tilde{A}_3^1(1),\tilde{A}_3^1(p)
-\delta(p)]=
2\delta(p)+\tilde{A}_3^1(t_{-p}1)=
-2(\tilde{A}_3^1(p)
-\delta(p))
\end{gather*}
and so 
$\mathfrak{g}_{\pm 2}
\subset \mathfrak{g}_{\pm 2\alpha}~(resp)$.
Hence (1) follows.

(2) It follows from the Jacobi identity.
\end{proof}
\medskip

\begin{proposition}\label{nlp-04}
The following equation hold.
\[
\tag{\ref{nlp-04}}
M=\mathrm{B}_3\cong{\rm Spin}(7).\]
\end{proposition}
\begin{proof}
From 
(\ref{cce-01}.b),
the subgroup $\mathrm{B}_3
\cong {\rm Spin}(7)$
in $K$
is the stabilizer
$(\mathrm{F}_{4(-20)})_{E_1,F_3^1(1)}$.
Fix $g\in{\rm Spin}(7)$.
Then $g$ can be expressed by 
$g=(g_1,\epsilon g_1\epsilon,g_3)$
such that
$g=(g_1,\epsilon g_1\epsilon,g_3)\in 
{\rm Spin}(8)$ 
and
$g_3 1=1$.
Put $\phi=\varphi_0(g)\tilde{A}_3^1(1)
\varphi_0(g)^{-1}.$ 
Using (\ref{cce-03}) and
(\ref{cce-09}.b), we calculate that
 \begin{gather*}
\phi (-E_1+E_2)=2F_3^1(1)
=\tilde{A}_3^1(1)(-E_1+E_2),\\
\phi P^-=2P^-=\tilde{A}_3^1(1)P^-,
~\phi E=0=\tilde{A}_3^1(1)E,
~\phi E_3=0=\tilde{A}_3^1(1)E_3,\\
\phi F_3^1(p)=2(g_3^{-1} p|1)(-E_1+E_2)
=2(p|1)(-E_1+E_2)
=\tilde{A}_3^1(1)F_3^1(p),\\
\phi Q^+(x)
=Q^+(\overline{x})
=\tilde{A}_3^1(1)Q^+(x),
~~\phi Q^-(x)
=-Q^-(\overline{x})
=\tilde{A}_3^1(1)Q^-(x)
\end{gather*}
where $x\in{\bf O}$ 
and $p\in{\rm Im}{\bf O}$.
From (\ref{prl-05}.b), we see
$\phi=\tilde{A}_3^1(1)$ on $\mathcal{J}^1$.
Thus  $\varphi_0(g)\in M$
and so $\mathrm{B}_3\subset M$.

Conversely, take $k\in M\subset K$.
Then 
$k,k^{-1} \in (\mathrm{F}_{4(-20)})_{E_1}$
by (\ref{spn-14}.b)
and
$k \tilde{A}_3^1(1)k^{-1}
=\tilde{A}_3^1(1)$. 
Now
from (\ref{cce-09}.b), 
we see that 
$kF_3^1(1)=-k \tilde{A}_3^1(1)E_1
=-k \tilde{A}_3^1(1)k^{-1}E_1=-\tilde{A}_3^1(1)E_1
=F_3^1(1)$.
Thus we obtain
$k\in (\mathrm{F}_{4(-20)})_{E_1,F_3^1(1)}=\mathrm{B}_3$
and so
$M\subset \mathrm{B}_3$.
Hence $M=\mathrm{B}_3$.
\end{proof}
\medskip

\begin{lemma}\label{nlp-05}
Let $i\in\{\pm 1,\pm 2\}$.
The following equations hold.
\begin{gather*}
\tag{\ref{nlp-05}.a}
\begin{array}{rlrl}
{\rm (i)} &\mathfrak{g}_i=\mathfrak{g}_{i\alpha},&
{\rm (ii)}& \mathfrak{a}_{\mathfrak{p}}=\mathfrak{a}. 
\end{array}
\\
\tag{\ref{nlp-05}.b}
\Sigma(\mathfrak{a}_{\mathfrak{p}})
=\Sigma=\{\pm \alpha,\pm 2\alpha\}.\\
\tag{\ref{nlp-05}.c}
\mathfrak{f}_{4(-20)}
=
\mathfrak{g}_{-2\alpha}\oplus
\mathfrak{g}_{-\alpha}
\oplus\mathfrak{a}\oplus \mathfrak{m}
\oplus \mathfrak{g}_{\alpha}\oplus
\mathfrak{g}_{2\alpha}.
\end{gather*}
\end{lemma}
\begin{proof}
From the definitions of
$\mathfrak{m}$ and $\mathfrak{a}$,
we see
$\mathfrak{m}\subset \mathfrak{g}_0\cap \mathfrak{k}$
and
$\mathfrak{a}\subset \mathfrak{g}_0\cap \mathfrak{p}$.
Then from Lemma~\ref{nlp-03}(1), we see that
$\mathfrak{g}_{-2}+
\mathfrak{g}_{-1}
+\mathfrak{a}+ \mathfrak{m}
+ \mathfrak{g}_{1}+
\mathfrak{g}_{2}$ is 
a direct sum 
$\mathfrak{g}_{-2}\oplus
\mathfrak{g}_{-1}
\oplus\mathfrak{a}\oplus \mathfrak{m}
\oplus \mathfrak{g}_{1}\oplus
\mathfrak{g}_{2}$ and that
\[\mathfrak{g}_{-2}\oplus
\mathfrak{g}_{-1}
\oplus\mathfrak{a}\oplus \mathfrak{m}
\oplus \mathfrak{g}_{1}\oplus
\mathfrak{g}_{2}
\subset \mathfrak{g}_{-2\alpha}\oplus
\mathfrak{g}_{-\alpha}
\oplus\mathfrak{a}\oplus \mathfrak{m}
\oplus \mathfrak{g}_{\alpha}\oplus
\mathfrak{g}_{2\alpha}
\subset \mathfrak{f}_{4(-20)}.\]
Now
${\rm dim}~\mathfrak{a}=1$,
${\rm dim}~\mathfrak{g}_{2}
={\rm dim}~\mathfrak{g}_{-2}
={\rm dim}~{\bf O}=8$,
${\rm dim}~\mathfrak{g}_{1}
={\rm dim}~\mathfrak{g}_{-1}
={\rm dim}~({\rm Im}{\bf O})=7$
and by (\ref{nlp-04}),
${\rm dim}~\mathfrak{m}
={\rm dim}~(\mathfrak{so}(7))=21$
where $\mathfrak{so}(7)=Lie({\rm SO}(7))$.
By (\ref{cce-09}.a),
${\rm dim}~
\mathfrak{f}_{4(-20)}=52$.
Thus
\[{\rm dim}~(
\mathfrak{g}_{-2}\oplus
\mathfrak{g}_{-1}
\oplus\mathfrak{a}\oplus \mathfrak{m}
\oplus \mathfrak{g}_{1}\oplus
\mathfrak{g}_{2})=52={\rm dim}~
\mathfrak{f}_{4(-20)}\]
and it is clear that
\[\mathfrak{g}_{-2}\oplus
\mathfrak{g}_{-1}
\oplus\mathfrak{a}\oplus \mathfrak{m}
\oplus \mathfrak{g}_{1}\oplus
\mathfrak{g}_{2}
= \mathfrak{g}_{-2\alpha}\oplus
\mathfrak{g}_{-\alpha}
\oplus\mathfrak{a}\oplus \mathfrak{m}
\oplus \mathfrak{g}_{\alpha}\oplus
\mathfrak{g}_{2\alpha}
= \mathfrak{f}_{4(-20)}\]
and that
$\mathfrak{g}_i=\mathfrak{g}_{i\alpha}$
for $i\in\{\pm 1,\pm 2\}$.
Because of $\mathfrak{a}\oplus \mathfrak{m}
\subset\mathfrak{g}_0$, the decomposition
$\mathfrak{f}_{4(-20)}=\mathfrak{g}_{-2\alpha}\oplus
\mathfrak{g}_{-\alpha}
\oplus(\mathfrak{a}\oplus \mathfrak{m})
\oplus \mathfrak{g}_{\alpha}\oplus
\mathfrak{g}_{2\alpha}$ implies
the eigendecomposition of ${\rm ad}(\tilde{A}_3^1(1))$
and the root space decomposition 
of $(\mathfrak{f}_{4(-20)},\mathfrak{a})$
(cf. \cite[Ch~V]{Knp2001}). 
Thus 
$\Sigma=\{\pm \alpha,\pm 2\alpha\}$ 
and $\mathfrak{g}_0=\mathfrak{a}\oplus \mathfrak{m}$.
Because of $\mathfrak{a}
\subset\mathfrak{a}_{\mathfrak{p}}
\subset\mathfrak{g}_0\cap \mathfrak{p}=\mathfrak{a}$,
we have
$\mathfrak{a}_{\mathfrak{p}}=\mathfrak{a}$
and $\Sigma(\mathfrak{a}_{\mathfrak{p}})=
\Sigma=\{\pm \alpha,\pm 2\alpha\}$.
\end{proof}
\medskip

The nilpotent subalgebras $\mathfrak{n}^{\pm}$ 
are defined by
\begin{align*}
\mathfrak{n}^+&:=\mathfrak{g}_{2\alpha}
\oplus \mathfrak{g}_{\alpha}
=\{\mathcal{G}_2(p)+\mathcal{G}_1(x)|
~p\in{\rm Im}{\bf O},x\in{\bf O}\},\\
\mathfrak{n}^-&:=\mathfrak{g}_{-2\alpha}
\oplus \mathfrak{g}_{-\alpha}
=\{\mathcal{G}_{-2}(p)+\mathcal{G}_{-1}(x)|
~p\in{\rm Im}{\bf O},x\in{\bf O}\}
\end{align*}
respectively.
In fact, 
by Lemma~\ref{nlp-03}(2) 
and (\ref{nlp-05}.b),
\[[\mathfrak{n}^+,[
\mathfrak{n}^+,\mathfrak{n}^+]]=0,\quad\quad
[\mathfrak{n}^-,[
\mathfrak{n}^-,\mathfrak{n}^-]]=0
\]
and
the nilpotent subgroups $N^{\pm}$ of $\mathrm{F}_{4(-20)}$ 
are defined as
\begin{align*}
N^+&:=\exp \mathfrak{n}^+=
\{
\exp(\mathcal{G}_2(p)+\mathcal{G}_1(x))~|
~p\in{\rm Im}{\bf O},x\in{\bf O}\},\\
N^-&:=\exp \mathfrak{n}^-=\{
\exp(\mathcal{G}_{-2}(p)+\mathcal{G}_{-1}(x))~|
~p\in{\rm Im}{\bf O},x\in{\bf O}\}
\end{align*}
respectively.

\begin{lemma}\label{nlp-06}
Let $x\in {\bf O}$ and $p\in{\rm Im}{\bf O}.$
\[
\tag{\ref{nlp-06}}
\exp\mathcal{G}_2(p)\exp\mathcal{G}_1(x)
=\exp(\mathcal{G}_2(p)+\mathcal{G}_1(x))
=\exp\mathcal{G}_1(x)\exp\mathcal{G}_2(p).\\
\]
\end{lemma}
\begin{proof}
By Lemma~\ref{nlp-03}(2) and (\ref{nlp-05}.b), 
$[\mathfrak{g}_{\alpha},\mathfrak{g}_{2\alpha}]=
[\mathfrak{g}_{2\alpha},\mathfrak{g}_{\alpha}]=0.$
Hence (\ref{nlp-06}) follows.
\end{proof}
\medskip

Denote
$T(x,y,z):=(x\overline{y})z-(z\overline{y})x$
for $x,y,z\in{\bf O}$ (cf. \cite[(6.55)]{Hfr1990}).

\begin{lemma}\label{nlp-07} 
{\rm (\cite[Lemma 6.56]{Hfr1990}).}
$T(x,y,z)$ 
is alternating on ${\bf O}:$
$T(x,x,z)=T(z,x,x)=T(x,z,x)=0$
for~all $x,z \in {\bf O}$.
Especially, for all $x_1,x_2,x_3 \in {\bf O}$,
\[
\tag{\ref{nlp-07}}
T(x_1,x_2,x_3)=T(x_i,x_{i+1},x_{i+2})
=-T(x_i,x_{i+2},x_{i+1})\]
where $i \in \{1,2,3\}$ and the indexes $i,i+1,i+2$ are counted
modulo $3$.
\end{lemma}
\begin{proof} 
It follows from (\ref{prl-01}.h).
\end{proof}
\medskip

\begin{lemma}\label{nlp-08}
Let $p,q\in{\rm Im}{\bf O}$ and $x,y\in{\bf O}.$
Then
\[
\tag{\ref{nlp-08}}
[\mathcal{G}_2(p)+\mathcal{G}_1(x),
~\mathcal{G}_2(q)+\mathcal{G}_1(y)]
=\mathcal{G}_2(2{\rm Im}(x\overline{y})).\]
\end{lemma}
\begin{proof}
First, by Lemma~\ref{nlp-03}(2)
and (\ref{nlp-05}.b),
\[[\mathcal{G}_2(p),\mathcal{G}_2(q)]
=[\mathcal{G}_2(p),
\mathcal{G}_1(x)]
=0.
\]
Second, we show $[\mathcal{G}_1(x),
\mathcal{G}_1(y)]
=\mathcal{G}_2(2{\rm Im}(x\overline{y}))$.
Put $f=[\mathcal{G}_1(x),
\mathcal{G}_1(y)]
-\mathcal{G}_2(2{\rm Im}(x\overline{y}))$. 
From (\ref{nlp-01}.b),
we calculate that
\begin{gather*}
[\mathcal{G}_1(x),\mathcal{G}_1(y)]
=([\tilde{A}_1^1(x),\tilde{A}_1^1(y)]
+[\tilde{A}_2^1(\overline{x}),
\tilde{A}_2^1(\overline{x})])
-\tilde{A}_3^1(2{\rm Im}(x\overline{y})),\\ 
\mathcal{G}_2(2{\rm Im}(x\overline{y}))
=\tilde{A}_3^1(-2{\rm Im}(x\overline{y}))-
\delta(2{\rm Im}(x\overline{y})).
\end{gather*}
Then 
$f=[\tilde{A}_1^1(\overline{x}),\tilde{A}_1^1(\overline{y})]
+[\tilde{A}_2^1(x),\tilde{A}_2^1(y)]
+\delta(2{\rm Im}(x\overline{y}))$.
By (\ref{nlp-01}.b)(ii), $f\in \mathfrak{d}_4$
and from Lemma~\ref{cce-05}(3),
we can write
$f=d\varphi_0(D_1,D_2,D_3)$
where $(D_1,D_2,D_3) \in (\mathfrak{D}_4)^3$
satisfying the infinitesimal triality.
Fix $z\in{\bf O}.$
Then
\begin{align*}
F_1^1(D_1z)&=d\varphi_0(D_1,D_2,D_3)F_1^1(z)=fF_1^1(z)\\
&=([\tilde{A}_1^1(x),
\tilde{A}_1^1(y)]
+[\tilde{A}_2^1(\overline{x}),\tilde{A}_2^1(\overline{y})]
+\delta(2{\rm Im}(x\overline{y})))F_1^1(z)
\end{align*}
and using 
(\ref{cce-09}.b), (\ref{prl-01}.g),
and (\ref{nlp-07}),
we see that
\begin{align*}
D_1z&=-4(y|z)x+4(x|z)y+(z\overline{y})x
-(z\overline{x})y+
(x\overline{y})z-(y\overline{x})z\\
&=-2(y\overline{z})z-2(z\overline{y})x
+2(x\overline{z})y+2(z\overline{x})y\\
&~~+(z\overline{y})x-(z\overline{x})y
+(x\overline{y})z-(y\overline{x})z\\
&=(z\overline{x})y-(y\overline{x})z
+(x\overline{y})z-(z\overline{y})x
+2(x\overline{z})y-2(y\overline{z})x\\
&=T(z,x,y)+T(x,y,z)+2 T(x,z,y)\\
&=T(x,y,z)+T(x,y,z)-2 T(x,y,z)=0.
\end{align*}
Then $D_1=0$
and from Lemma~\ref{cce-05}(1),
we see
$D_2=D_3=0$.
Thus $f=0$
and so $[\mathcal{G}_1(x),
\mathcal{G}_1(y)]
=\mathcal{G}_2(2{\rm Im}(x\overline{y}))$.
Hence the result follows.
\end{proof}
\medskip

\begin{lemma}\label{nlp-09} 
There exists a neighborhood $U$ of $0$ 
in ${\rm Im}{\bf O}\times {\bf O}$
such that  
\begin{align*}
\tag{\ref{nlp-09}}
&\exp(\mathcal{G}_2(p)+\mathcal{G}_1(x))
\exp(\mathcal{G}_2(q)+\mathcal{G}_1(y))\\
=&\exp(\mathcal{G}_2(p+q+{\rm Im}(x\overline{y}))
+\mathcal{G}_1(x+y))
\end{align*}
for all $(p,x),(q,y)\in U$.
\end{lemma}
\begin{proof}
Using Campbell-Hausdorff-Dynkin formulas
(cf. \cite[Theorem 3.4.4]{Fj2008}),
there exists  
a neighborhood $U_1$ of $0$ 
in ${\rm End}_{\mathbb{R}}(\mathcal{J}^1)$
such that for all $X,Y\in U_1,$ 
\begin{align*}
\exp X\exp Y&=\exp\left(X+Y+2^{-1}[X,Y]
+12^{-1}[X,[X,Y]]\right.\\
&~\left.+12^{-1}[Y,[Y,X]]
+({\rm terms~of~degree}\geq 4)\right).
\end{align*}
Because of 
$[\mathfrak{n}^+,[\mathfrak{n}^+,\mathfrak{n}^+]]=0$,
we see
\[\exp X\exp Y=\exp\left(X+Y+2^{-1}[X,Y]\right)\]
for all $X,Y\in \mathfrak{n}^+\cap U_1$.
Then by (\ref{nlp-08}), 
there exists a neighborhood $U$ of $0$ 
in ${\rm Im}{\bf O}\times {\bf O}$
such that for all $(p,x),(q,y)\in U,$
\begin{align*}
&\exp(\mathcal{G}_2(p)+\mathcal{G}_1(x))
\exp(\mathcal{G}_2(q)+\mathcal{G}_1(y)))\\
&=\exp\left(\mathcal{G}_2(p+q)
+\mathcal{G}_1(x+y)
+2^{-1}[\mathcal{G}_2(p)+\mathcal{G}_1(x),
\mathcal{G}_2(q)+\mathcal{G}_1(y)
]\right)\\
&=\exp(\mathcal{G}_2(p+q+{\rm Im}(x\overline{y}))
+\mathcal{G}_1(x+y)).
\end{align*}
\end{proof}

\begin{lemma}\label{nlp-10} 
Let $p,q\in{\rm Im}{\bf O}$ and $x,y\in{\bf O}$.
\begin{gather*}
\tag{\ref{nlp-10}.a}
\left\{\begin{array}{rl}
{\rm (i)}&
\mathcal{G}_2(p)(-E_1+E_2)=-2F_3^1(p), \\
\smallskip
{\rm (ii)}&
\mathcal{G}_2(p)P^-=0,\quad
{\rm (iii)}~~~
\mathcal{G}_2(p)E=0,\quad
{\rm (iv)}~~~
\mathcal{G}_2(p)E_3=0,\\
\smallskip
{\rm (v)}&\mathcal{G}_2(p)F_3^1(q)=-2(p|q)P^-,\\
\smallskip
{\rm (vi)}&
\mathcal{G}_2(p)Q^+(y)=0,\quad
{\rm (vii)}~~~\mathcal{G}_2(p)Q^-(y)
=-2Q^+(py).
\end{array}\right.\\
\tag{\ref{nlp-10}.b}
\left\{\begin{array}{rl}
{\rm (i)}&\mathcal{G}_1(x)(-E_1+E_2)=-Q^-(x),\\
\smallskip
{\rm (ii)}&\mathcal{G}_1(x)P^-=0,\quad
{\rm (iii)}~~~\mathcal{G}_1(x)E=0,
\quad
{\rm (iv)}~~~\mathcal{G}_1(x)E_3=Q^+(x),\\
\smallskip
{\rm (v)}&\mathcal{G}_1(x)F_3^1(q)
=-Q^+(qx),\quad
{\rm (vi)}~~~
\mathcal{G}_1(x)Q^+(y)
=2(x|y)P^-,\\
\smallskip
{\rm (vii)}&
  \mathcal{G}_1(x)Q^-(y)
  =2(x|y)(E-3E_3)+F_3^1(2{\rm Im}(x\overline{y})).\\
\end{array}\right.
\end{gather*}
\end{lemma}

\begin{proof}
Using (\ref{cce-09}.b) and the definition of $\delta(p)$,
it follows from
direct calculations. 
\end{proof}
\medskip

\begin{lemma}\label{nlp-11} 
Let $p,q\in{\rm Im}{\bf O}$ and $x,y\in{\bf O}$.
\begin{gather*}
\tag{\ref{nlp-11}.a}
\left\{\begin{array}{rl}
{\rm (i)}& \exp\mathcal{G}_2(p)(-E_1+E_2)
=(-E_1+E_2)-2F_3^1(p)+2(p|p)P^-,\\
\smallskip
{\rm (ii)}
& \exp\mathcal{G}_2(p)P^-
=P^-,\quad
{\rm (iii)}~~~\exp\mathcal{G}_2(p)E=E,
\\
\smallskip
{\rm (iv)}&\exp\mathcal{G}_2(p)E_3=E_3,
\\
\smallskip
{\rm (v)}&
\exp\mathcal{G}_2(p)F_3^1(q)=F_3^1(q)-2(p|q)P^-,\\
\smallskip
{\rm (vi)}& 
\exp\mathcal{G}_2(p)Q^+(y)
=Q^+(y),\\
\smallskip
{\rm (vii)}&
\exp\mathcal{G}_2(p)Q^-(y)=Q^-(y)
-2Q^+(py).
\end{array}\right.
\\
\left\{\begin{array}{rl}
\tag{\ref{nlp-11}.b}
{\rm (i)}&
\exp\mathcal{G}_1(x)(-E_1+E_2)
  =(-E_1+E_2)-Q^-(x)\\
\smallskip
&~~~-(x|x)(E-3E_3)+(x|x)Q^+(x)
+2^{-1}(x|x)^2P^-,\\
\smallskip
{\rm (ii)}&\exp\mathcal{G}_1(x)P^-=P^-,\quad
{\rm (iii)}~~~\exp\mathcal{G}_1(x)E=E,\\
\smallskip
{\rm (iv)}& \exp\mathcal{G}_1(x)E_3
=E_3+Q^+(x)+(x|x)P^-,\\
\smallskip
{\rm (v)}&
\exp\mathcal{G}_1(x)F_3^1(q)=F_3^1(q)-Q^+(qx),\\
\smallskip
{\rm (vi)}&
\exp\mathcal{G}_1(x)Q^+(y)=Q^+(y)+2(x|y)P^-,\\
\smallskip
{\rm (vii)}&
\exp\mathcal{G}_1(x)Q^-(y)=Q^-(y)
+2(x|y)(E-3E_3)\\
&~~~+F_3^1(2{\rm Im}(x\overline{y}))
-Q^+(3(x|y)x+{\rm Im}(x\overline{y})x)-2(x|y)(x|x)P^-.
\end{array}\right.
\end{gather*}
\end{lemma}
\begin{proof}
Using (\ref{nlp-10}.a) and (\ref{nlp-10}.b),
it follows from
direct calculations. 
\end{proof}
\medskip

\begin{lemma}\label{nlp-12} 
Let $p\in \mathrm{Im}{\bf O}$, $x\in {\bf O}$ and
$X\in \mathfrak{n}^+$. Then
$\mathcal{G}_2(p)^3=0$ and $\mathcal{G}_1(x)^5=0$
Especially,
\[
\tag{\ref{nlp-12}}
\exp\mathcal{G}_2(p)=
\sum{}_{n=0}^2 (n!)^{-1}\mathcal{G}_2(p)^n ,\quad
\exp\mathcal{G}_1(x)=
\sum{}_{n=0}^4 (n!)^{-1}
\mathcal{G}_1(x)^n .
\]
\end{lemma}
\begin{proof}
Let $S=\{-E_1+E_2,~P^-,~E,~E_3,~F_3^1(p),~Q^+(x),
~Q^-(y)|~p\in{\rm Im}{\bf O},~x,y\in{\bf O}\}$.
From (\ref{nlp-10}), we see
that
$\mathcal{G}_2(p)^3Y=0$ and $\mathcal{G}_1(x)^5Y=0$
for all $Y \in S$.
Thus it follows from (\ref{prl-05}.b) that
$\mathcal{G}_2(p)^3=0$ and $\mathcal{G}_1(x)^5=0$
on $\mathcal{J}^1$.
\end{proof}

\section{The stabilizers of 
semidirect product group type.}\label{sdp}
Consider
${\rm Spin}(7)\times
{\rm Im}{\bf O}\times{\bf O}$
in which multiplication is defined by 
\[(g,p,x)(h,q,y):=
(gh,p+g_3 q
+{\rm Im}(x\overline{(g_1 y)}),
x+g_1 y)\]
where $p,q\in{\rm Im}{\bf O}$,
$x,y\in{\bf O}$ and
$g=(g_1,g_2,g_3),
h
\in {\rm Spin}(7)$.
By Lemma~\ref{trl-03}(1),
$g_3 q\in{\rm Im}{\bf O}$
and it is clear that
the multiplication is closed.
Denote $G:={\rm Spin}(7)\times
{\rm Im}{\bf O}\times{\bf O}$,
$G_0:=\{(g,0,0)|~g\in{\rm Spin}(7)\}$
and define
the subset ${\rm H}_{{\rm Im}{\bf O},{\bf O}}$
of $G$ by 
\[
{\rm H}_{{\rm Im}{\bf O},{\bf O}}
:=\{(1,p,x)|~p\in{\rm Im}{\bf O},x\in{\bf O}\}.
\]
Then, for all $(1,p,x),(1,q,y) \in 
{\rm H}_{{\rm Im}{\bf O},{\bf O}}$,
\[(1,p,x)(1,q,y)=
(1,p+q
+{\rm Im}(x\overline{y}),
x+y).\]
${\rm H}_{{\rm Im}{\bf O},{\bf O}}$ 
has a group structure
from the following lemma
and is called
the {\it Heisenberg group of} {\bf O}.

\begin{lemma}\label{sdp-01} 
{\rm (1)}  $G$ is a group with respect to the multiplication.

{\rm (2)}  $G_0$ and ${\rm H}_{{\rm Im}{\bf O},{\bf O}}$ 
are subgroups of $G;$~
$G_0 \cong {\rm Spin}(7)$.
\smallskip

{\rm (3)}  
$G$ is the semi-direct product 
$G_0\ltimes {\rm H}_{{\rm Im}{\bf O},{\bf O}}$.
\end{lemma}

\begin{proof}
(1) Let $g=(g_1,g_2,g_3),
h=(h_1,h_2,h_3),
f
\in {\rm Spin}(7)$,
$p,q,r\in{\rm Im}{\bf O}$ and $x,y,z\in{\bf O}.$
Because of $g_3 1=1$ and (\ref{trl-03}.c)(ii),
we see
$g_3({\rm Im}(y\overline{(h_1 z)})
={\rm Im}((g_1 y)\overline{(g_1h_1 z)})$.
Then
\begin{align*}
~&((g,p,x)(h,q,y))(f,r,z)\\
&=(ghf,
p+g_3 q+g_3h_3r
+{\rm Im}(x\overline{(g_1 y)})\\
~&\quad+{\rm Im}(x\overline{(g_1h_1 z)})
+{\rm Im}((g_1 y)\overline{(g_1h_1 z)}),
x+g_1 y+g_1h_1 z)\\
&=(ghf,
p+g_3 q+g_3h_3 r
+{\rm Im}(x\overline{(g_1 y)})\\
&\quad+{\rm Im}(x\overline{(g_1h_1 z)})
+g_3({\rm Im}(y\overline{(h_1 z)})),
x+g_1 y+g_1h_1 z)\\
&=(g,p,x)((h,q,y)(f,r,z)).
\end{align*}
Thus  the associativity hold.
The identity element is $(1,0,0)$ 
and the inverse element of $(g,p,x)$ is 
$(g^{-1},-g_3^{-1} p,-g_1^{-1} x)$.
Hence (1) follows.

(2) Because
$(g,0,0)^{-1}(h,0,0)=(g^{-1}h,0,0)$
and
$(1,p,x)^{-1}(1,q,y)=(1,*,*)$,
we see that
$G_0$ and ${\rm H}_{{\rm Im}{\bf O},{\bf O}}$ 
are subgroups of $G$ and
obviously, $G_0 \cong {\rm Spin}(7)$.
Hence (2) follows.

(3) 
Let $(g,p,x) \in {\rm H}_{{\rm Im}{\bf O},{\bf O}}$
where $g=(g_1,g_2,g_3) \in \mathrm{Spin}(7)$.
Then $G=G_0{\rm H}_{{\rm Im}{\bf O},{\bf O}}$ follows from 
$(g,p,x)=(g,0,0)(1,g_3^{-1} p,g_1^{-1} x)$.
Obviously $G_0\cap {\rm H}_{{\rm Im}{\bf O},{\bf O}}
=\{(1,0,0)\}$.
Because of $(g,p,x)(1,q,y)(g,p,x)^{-1}
=(1,*,*)$, we see
$(g,p,x)({\rm H}_{{\rm Im}{\bf O},{\bf O}})(g,p,x)^{-1}
\subset {\rm H}_{{\rm Im}{\bf O},{\bf O}}$.
Thus
${\rm H}_{{\rm Im}{\bf O},{\bf O}}$ 
is a normal subgroup of $G$.
Hence $G=G_0\ltimes {\rm H}_{{\rm Im}{\bf O},{\bf O}}$.
\end{proof}

\noindent
Hereafter, we identify $G_0$ with ${\rm Spin}(7)$.    
Then we can write  $G={\rm Spin}(7)
\ltimes {\rm H}_{{\rm Im}{\bf O},{\bf O}}$
and use the notation ${\rm Spin}(7)
\ltimes {\rm H}_{{\rm Im}{\bf O},{\bf O}}$.
Next, we denote
$G':=\{(g,p,0)|~g
\in {\rm Spin}(7),
p\in{\rm Im}{\bf O}\}$
and
$N':=\{(1,p,0)|
~p\in{\rm Im}{\bf O}\}\subset G'$.
Moreover, we denote
$G'':=\{(g,p,x)|~g
\in \mathrm{G}_2,~
p,x\in{\rm Im}{\bf O}\}$
and
$G_0'':=\{(g,0,0)~|~g\in \mathrm{G}_2\}\subset G''$
and define
\[
{\rm H}_{{\rm Im}{\bf O},{\rm Im}{\bf O}}
:=\{(1,p,q)~|~p,q\in{\rm Im}{\bf O}\}
\subset G''.
\]
Easily,  we can prove the following lemma.

\begin{lemma}\label{sdp-02} 
{\rm (1)}
$G'$ and $G''$ are subgroups of 
${\rm Spin}(7)\ltimes 
{\rm H}_{{\rm Im}{\bf O},{\bf O}}$.
\smallskip

{\rm (2)} ${\rm Spin}(7)$ and $N'$ are subgroups of $G';$~
$N' \cong {\rm Im}{\bf O}.$
\smallskip

{\rm (3)} $G_0''$ and ${\rm H}_{{\rm Im}{\bf O},{\rm Im}{\bf O}}$ 
are subgroups of $G'';$
$G_0''\cong \mathrm{G}_2$.
\smallskip

{\rm (4)}
$G'$ is the semi-direct product 
${\rm Spin}(7)\ltimes N.$
\smallskip

{\rm (5)}  
$G''$ is the semi-direct product 
$G_0''\ltimes {\rm H}_{{\rm Im}{\bf O},{\rm Im}{\bf O}}$.
\end{lemma}

\noindent
Hereafter, we identify 
$N'$ with ${\rm Im}{\bf O}$ and
$G_0''$ with
$\mathrm{G}_2$. Then we can write
$G'={\rm Spin}(7)\ltimes {\rm Im}{\bf O}$ and
$G''=\mathrm{G}_2\ltimes 
{\rm H}_{{\rm Im}{\bf O},{\rm Im}{\bf O}}$,
respectively,
and use the notations
${\rm Spin}(7)\ltimes {\rm Im}{\bf O}$
and $\mathrm{G}_2\ltimes 
{\rm H}_{{\rm Im}{\bf O},{\rm Im}{\bf O}}$.
The map
$\varphi:
{\rm Spin}(7)\ltimes {\rm H}_{{\rm Im}{\bf O},{\bf O}}
\to (\mathrm{F}_{4(-20)})_{P^-}$ is defined by
\begin{align*}
&\varphi(g,p,x):
=\exp(\mathcal{G}_2(p)+\mathcal{G}_1(x))\varphi_0(g)\\
=&\exp\mathcal{G}_2(p)\exp\mathcal{G}_1(x)\varphi_0(g)
=\exp\mathcal{G}_1(x)\exp\mathcal{G}_2(p)\varphi_0(g)
\end{align*}
for $(g,p,x)\in
{\rm Spin}(7)\ltimes {\rm H}_{{\rm Im}{\bf O},{\bf O}}$
(see (\ref{nlp-06})).
From (\ref{cce-03}),
(\ref{nlp-11}.a) and (\ref{nlp-11}.b),
we see 
$\varphi(g,p,x)P^-=P^-$. 
So $\varphi$ is well-defined.

\begin{lemma}\label{sdp-03}
Let $g
=(g_1,g_2,g_3)
\in {\rm Spin}(7),$ 
$p,q\in{\rm Im}{\bf O}$
and $x,y\in{\bf O}$.
\[
\tag{\ref{sdp-03}}
\varphi_0(g)\exp(\mathcal{G}_2(p)
+\mathcal{G}_1(x))
\varphi_0(g)^{-1}
=\exp(\mathcal{G}_2(g_3p)
+\mathcal{G}_1(g_1x)).\]
\end{lemma}
\begin{proof}  
Put 
$S=\{-E_1+E_2,~P^-,~E,~E_3,~F_3^1(q),
~Q^+(y),~Q^-(z)
|~
q\in{\rm Im}{\bf O},~y,z\in{\bf O}\}$
and
$A=\varphi_0(g)$.
Because of
$A\exp(\mathcal{G}_2(p)
+\mathcal{G}_1(x))
A^{-1}
=\exp(A(\mathcal{G}_2(p)
+\mathcal{G}_1(x))A^{-1})$
and (\ref{prl-05}.b),
it is enough to show
that 
\[{\rm (i)}~A\mathcal{G}_2(p)A^{-1}X
=\mathcal{G}_2(g_3p)X,\quad
\text{(ii)}~A\mathcal{G}_1(x)
A^{-1}X
=\mathcal{G}_2(g_1x)X\]
for all $X \in S$.

(Step 1) We show
(i) by using (\ref{cce-03}) and  
(\ref{nlp-10}.a).

Case
$X=P^-,E,E_3,Q^+(y)$.
Then $A\mathcal{G}_2(p)A^{-1}X=0
=\mathcal{G}_2(g_3p)X$.

Case $X=-E_1+E_2$.
Then
$A\mathcal{G}_2(p)A^{-1}(-E_1+E_2)=-2F_3^1(g_3p)
=\mathcal{G}_2(g_3p)(-E_1+E_2)$.

Case $X=F_3^1(q)$.
Because $(g_1,g_2,g_3)
\in {\rm Spin}(7)$,
$A\mathcal{G}_2(p)A^{-1}F_3^1(q)=-2(p|g_3^{-1}q)P^-
=-2(g_3p|q)P^-=\mathcal{G}_2(g_3p)F_3^1(q)$.

Case $X=Q^-(z).$ 
From (\ref{trl-03}.c)(iii),
we see that
$A\mathcal{G}_2(p)A^{-1}Q^-(z)
=-2Q^+(g_1(p(g_1^{-1}z)))
=-2Q^+((g_3p)z)=
\mathcal{G}_2(g_3p)Q^-(z)$.

Thus 
(i) follows.
\smallskip

\noindent
(Step 2) We show 
by (ii) using (\ref{cce-03}) 
and (\ref{nlp-10}.b).

Case
$X=E,P^-$. Then $A\mathcal{G}_1(x)A^{-1}X=0
=\mathcal{G}_1(g_2x)X$.

Case
$X=-E_1+E_2$. Then
$A\mathcal{G}_1(x)A^{-1}(-E_1+E_2)=-Q^-(g_1x)
=\mathcal{G}_1(g_1x)(-E_1+E_2)$.

Case
$X=E_3$. Then
$A\mathcal{G}_1(x)A^{-1}E_3=Q^+(g_1x)
=\mathcal{G}_1(g_1x)E_3$.

Case
$X=Q^+(y)$. 
Because $(g_1,g_2,g_3)
\in {\rm Spin}(7)$, 
$A\mathcal{G}_1(x)A^{-1}Q^+(y)
=-2(x|g_1^{-1}y)P^-
=-2(g_1x|y)P^-
=\mathcal{G}_2(g_1p)Q^+(y)$.

Case
$X=F_3^1(q)$. 
From (\ref{trl-03}.c)(iii),
we see that
$A\mathcal{G}_1(x)A^{-1}F_3^1(q)
=F_3^1(g_1((g_3^{-1}q)x))
=F_3^1(q(g_1x))
=\mathcal{G}_1(g_1x)F_3^1(q)$.

Case
$X=Q^-(z)$. 
Because of $(g_1,g_2,g_3)
\in {\rm Spin}(7)$ and (\ref{trl-03}.c)(ii),
$A\exp\mathcal{G}_1(x)A^{-1}Q^-(z)
=2(x|g_1^{-1}z)(E-3E_3)
+F_3^1(2g_3{\rm Im}(x\overline{(g_1^{-1}z)}))
=2(g_1x|z)(E-3E_3)
+F_3^1(2(g_1x)\overline{z})
=\mathcal{G}_1(g_1x)Q^-(z)$.

Thus (ii) follows.
Hence the result follows.
\end{proof}
\medskip

\begin{lemma}\label{sdp-04}
Let $g
=(g_1,g_2,g_3)
\in {\rm Spin}(7),$ 
$p,q\in{\rm Im}{\bf O}$
and $x,y\in{\bf O}$.
\begin{align*}
\tag{\ref{sdp-04}}
&\exp(\mathcal{G}_2(p)+\mathcal{G}_1(x))
\exp(\mathcal{G}_2(q)+\mathcal{G}_1(y))\\
=&\exp(\mathcal{G}_2(p+q+{\rm Im}(x\overline{y}))
+\mathcal{G}_1(x+y)).
\end{align*}
\end{lemma}
\begin{proof}
Put  $f(p,x,q,y)
\in{\rm End}_{\mathbb{R}}(\mathcal{J}^1)$ as
\begin{align*}
f(p,x,q,y)
=&\exp(\mathcal{G}_2(p)+\mathcal{G}_1(x))
\exp(\mathcal{G}_2(q)+\mathcal{G}_1(y))\\
&-\exp(\mathcal{G}_2(p+q+{\rm Im}(x\overline{y}))
+\mathcal{G}_1(x+y)).
\end{align*}
Let $p_i,q_i,x_i,y_i$ be variables defined as
$p=\sum_{i=1}^7p_ie_i,$ $q=\sum_{i=1}^7q_ie_i,$
$x=\sum_{i=0}^7x_ie_i$ and $y=\sum_{i=0}^7y_ie_i.$
Fix $X,Y\in\mathcal{J}^1.$
Put the function 
$F_{X,Y}(p,x,q,y)=(f(p,x,q,y)X|Y).$
From (\ref{nlp-06}) and (\ref{nlp-12}), we see
\begin{gather*}
f(p,x,q,y)=(\sum{}_{i=0}^2(i!)^{-1}\mathcal{G}_2(p)^i)
(\sum{}_{i=0}^4(i!)^{-1}\mathcal{G}_1(x)^i)\\
(\sum{}_{i=0}^2(i!)^{-1}\mathcal{G}_2(q)^i)
(\sum{}_{i=0}^4(i!)^{-1}\mathcal{G}_1(y)^i)\\
-(\sum{}_{i=0}^2(i!)^{-1}\mathcal{G}_2(
p+q+{\rm Im}(x\overline{y}))^i)
(\sum{}_{i=0}^4(i!)^{-1}\mathcal{G}_1(x+y)^i)
\end{gather*}
and it is clear that
$F_{X,Y}(p,x,q,y)$
is a polynomial function.
Now from Lemma~\ref{nlp-09}, 
there exists a neighborhood $U$
of $0$ in $({\rm Im}{\bf O}\times {\bf O})^2$ 
such that
$f(p,x,q,y)=0$ for all $(p,x,q,y)\in U$.
Then  $F_{X,Y}(p,x,q,y)=0$ for all $(p,x,q,y)\in U$.
Since $F_{X,Y}(p,x,q,y)$
is a polynomial function,  $F_{X,Y}(p,x,q,y)=0$ 
for all $(p,x,q,y)
\in({\rm Im}{\bf O}\times {\bf O})^2$.
Moving $X,Y\in \mathcal{J}^1$,
we obtain $f(p,x,q,y)=0$.
Hence the result follows.
\end{proof}
\medskip

\begin{lemma}\label{sdp-05}
$\varphi$ is a group homomorphism.
Furthermore, \[
\tag{\ref{sdp-05}}
\varphi(
{\rm Spin}(7)\ltimes \mathrm{H}_{{\rm Im}{\bf O},{\bf O}})
= N^+M\subset (\mathrm{F}_{4(-20)})_{P^-}.\]
\end{lemma}
\begin{proof}
Let $F(p_0,x_0)
=\exp(\mathcal{G}_2(p_0)+\mathcal{G}_1(x_0))$
for $(p_0,x_0)\in\mathrm{H}_{{\rm Im}{\bf O},{\bf O}}$.
From (\ref{sdp-03}), (\ref{sdp-04}) 
and Lemma~\ref{cce-02}(2), we see
\begin{align*}
&\varphi(g,p,x)\varphi(h,q,y)
=F(p,x)(\varphi_0(g)
F(q,y)
\varphi_0(g)^{-1})\varphi_0(gh)\\
=&F(p,x)F(g_3q,g_1y)\varphi_0(gh)
=F(p+g_3 q+{\rm Im}(x\overline{(g_1y)}),x+g_1y)
\varphi_0(gh)\\
=&\varphi(
gh,
p+g_3 q+{\rm Im}(x\overline{(g_1y)}),x+g_1y).
\end{align*}
Thus $\varphi$ is a group homomorphism.
Moreover,
(\ref{sdp-05}) follows from
the definition of $\varphi$
and (\ref{nlp-04}).
\end{proof}
\medskip

Let $V$ be a $\mathbb{R}$-linear space and $N$ a
nilpotent subgroup of 
$\mathrm{GL}_{\mathbb{R}}(V).$
For $v\in V,$ the subset $Orb_N(v)$ of $V$ 
is called a {\it parabolic type plane}.  
Denote 
nilpotent subgroups $N_1$ and $N_2$ 
in $N^+=\varphi(\mathrm{H}_{{\rm Im}{\bf O},{\bf O}})$
as \[N_1:=\varphi({\rm Im}{\bf O}),\quad
N_2:=\varphi(
\mathrm{H}_{{\rm Im}{\bf O}, {\rm Im}{\bf O}})
\]
respectively, and the subsets 
$\mathcal{P}_{E_3,P^-},~\mathcal{P}_{P^-}$
and $\mathcal{P}_{Q^+(1)}$ of $\mathcal{J}^1$ as
\begin{align*}
\mathcal{P}_{E_3,P^-}&
:=\{X \in \mathcal{J}^1_{L^{\times}(2E_3),-1}
|~P^-\times X=-E_3,
~Q(X)=1\},\\
\mathcal{P}_{P^-}&:=\{X \in \mathcal{J}^1
|~P^-\times X=-2^{-1}P^-,~X^{\times 2}=0,
~\mathrm{tr}(X)=1\},\\
\mathcal{P}_{Q^+(1)}&:=\{X \in \mathcal{P}_{P^-}
|~Q^+(1)\times X=0\}
\end{align*}
respectively.

\begin{lemma}\label{sdp-06} 
The following equations hold.
\begin{gather*}
\tag{\ref{sdp-06}.a}
\begin{array}[t]{rcl}
\mathcal{P}_{E_3,P^-}&=&\{(-E_1+E_2)+2F_3^1(p)
+2(p|p)P^-
|~p\in {\rm Im}{\bf O}\}\\
\smallskip
&=&\{\exp\mathcal{G}_2(p)(-E_1+E_2)
|~p\in{\rm Im}{\bf O}\}\\
\smallskip
&=&Orb_{N_1}(-E_1+E_2)\\
\smallskip
&=&Orb_{(\mathrm{F}_{4(-20)})_{E_3,P^-}}(-E_1+E_2).
\end{array}\\
\tag{\ref{sdp-06}.b}
\begin{array}[t]{rcl}
\mathcal{P}_{P^-}&=&\{E_3+Q^+(x)+(x|x)P^-
|~x\in {\bf O}\}\\
\smallskip
&=&\{\exp\mathcal{G}_1(x)E_3
|~x\in{\bf O}\}\\
\smallskip
&=&Orb_{N^+}(E_3)\\
\smallskip
&=&Orb_{(\mathrm{F}_{4(-20)})_{P^-}}(E_3).
\end{array}\\
\tag{\ref{sdp-06}.c}
\begin{array}[t]{rcl}
\mathcal{P}_{Q^+(1)}&=&\{E_3+Q^+(x)+(x|x)P^-
|~x\in {\rm Im}{\bf O}\}\\
\smallskip
&=&\{\exp\mathcal{G}_1(x)E_3
|~x\in{\rm Im}{\bf O}\}\\
\smallskip
&=&Orb_{N_2}(E_3)\\
\smallskip
&=&Orb_{(\mathrm{F}_{4(-20)})_{Q^+(1)}}(E_3).
\end{array}
\end{gather*}
\end{lemma}
\begin{proof} (a) First, set $\mathcal{P}
=\{(-E_1+E_2)+F_3^1(2p)+2(p|p)P^-
~|~p\in {\rm Im}{\bf O}\}.$
By (\ref{nlp-11}.a)(i),
$(-E_1+E_2)+F_3^1(2p)+2(p|p)P^-=\exp\mathcal{G}_2(-p)(-E_1+E_2)$.
Thus $\mathcal{P}
=\{\exp\mathcal{G}_2(p)(-E_1+E_2)
~|~p\in{\rm Im}{\bf O}\}
\subset Orb_{N_1}(-E_1+E_2)$.

Second,
fix $X\in\mathcal{P}_{E_3,P^-}$.
Because of
$\mathcal{J}^1_{L^{\times}(2E_3),-1}=\mathbb{R}(-E_1+E_2)\oplus
\mathbb{R}P^-\oplus F_3^1({\rm Im}{\bf O})$, 
$X$ can be expressed by 
$X=r(-E_1+E_2)+F_3^1(2p)+sP^-$
for some $r,s\in \mathbb{R}~{\rm and}~p
\in {\rm Im}{\bf O}$.
By direct calculations,
$P^-\times X=-rE_3$.
Then
$r=1$
and $X=(-E_1+E_2)+F_3^1(2p)+sP^-$.
Because of $1=Q(X)=(1+s)^2-s^2-4(p|p)$,
we see
$s=2(p|p)$
and $X=(-E_1+E_2)+2F_3^1(p)+2(p|p)P^-$.
Thus  $X\in \mathcal{P}$ and so
$\mathcal{P}_{E_3,P^-} 
\subset \mathcal{P}$.

Third,
since $N_1$ is a subgroup of $(\mathrm{F}_{4(-20)})_{E_3,P^-}$,
we see $Orb_{N_1}(-E_1+E_2)
\subset Orb_{(\mathrm{F}_{4(-20)})_{E_3,P^-}}(-E_1+E_2)$. 

Last,
by virtue of
the definition of
$\mathcal{P}_{E_3,P^-},$ $(\mathrm{F}_{4(-20)})_{E_3,P^-}$ 
acts on $\mathcal{P}_{E_3,P^-}$.
Because of $-E_1+E_2\in\mathcal{P}_{E_3,P^-}$,
we see 
$Orb_{(\mathrm{F}_{4(-20)})_{E_3,P^-}}(-E_1+E_2)
\subset \mathcal{P}_{E_3,P^-}$.
Consequently 
$\mathcal{P}_{E_3,P^-} \subset \mathcal{P}
\subset Orb_{N_1}(-E_1+E_2)
\subset Orb_{(\mathrm{F}_{4(-20)})_{E_3,P^-}}(-E_1+E_2)
\subset \mathcal{P}_{E_3,P^-}$.
Hence (\ref{sdp-06}.a) follows.

(b) First, set $\mathcal{P}'=\{E_3+Q^+(x)+(x|x)P^-
~|~x\in {\bf O}\}$.
By (\ref{nlp-11}.b)(iv),
$E_3+Q^+(x)+(x|x)P^-=\exp\mathcal{G}_1(x)
E_3$.
Thus
$\mathcal{P}'=\{\exp\mathcal{G}_1(x)E_3
~|~x\in{\bf O}\}\subset Orb_{N^+}(E_3)$.

Second,
fix  $X\in\mathcal{P}_{P^-}$. 
By (\ref{prl-05}.b),
$X$ can be expressed by
$X=r(-E_1+E_2)+sP^-+uE+vE_3+F_3^1(p)
+Q^+(x)+Q^-(y)$
for some $r,s,u,v\in\mathbb{R}$
$p\in{\rm Im}{\bf O}$ and $x,y\in{\bf O}$.
By direct calculations, 
$P^-\times X
=-2^{-1}(u+v)P^--rE_3+Q^+(y)$
and therefore $r=y=0$, $u+v=1$
and
$X=sP^-+(1-v)E+vE_3+F_3^1(p)+Q^+(x)$. 
Because of $1=\mathrm{tr}(X)=3(1-v)+v$,
we see
$v=1$
and
$X=E_3+sP^-+F_3^1(p)+Q^+(x)$.
Because of $0=X^{\times 2}
=((x|x)-s)P^--F_3^1(p)+(p|p)E_3+Q^-(px)$,
we see
$p=0$ and $s=(x|x)$.
Thus  $X=E_3+(x|x)P^-+Q^+(x)
\in \mathcal{P}'$ 
and so $\mathcal{P}_{P^-}\subset \mathcal{P}'$.

Third,
since $N^+$ is a subgroup of $(\mathrm{F}_{4(-20)})_{P^-}$,
we see
$Orb_{N^+}(E_3)\subset Orb_{(\mathrm{F}_{4(-20)})_{P^-}}(E_3)$.

Last,
by virtue of
the definition of
$\mathcal{P}_{P^-},$ $(\mathrm{F}_{4(-20)})_{P^-}$ 
acts on $\mathcal{P}_{P^-}$.
Because of $E_3\in\mathcal{P}_{P^-}$, we see
$Orb_{(\mathrm{F}_{4(-20)})_{P^-}}(E_3)\subset \mathcal{P}_{P^-}$.
Consequently, we obtain 
$\mathcal{P}_{P^-} \subset \mathcal{P}'
\subset Orb_{N^+}(E_3)
\subset Orb_{(\mathrm{F}_{4(-20)})_{P^-}}(E_3)
\subset \mathcal{P}_{P^-}$.
Hence (\ref{sdp-06}.b) follows.

(c) First, set $\mathcal{P}''=\{E_3+Q^+(x)+(x|x)P^-
|x\in {\rm Im}{\bf O}\}.$
By (\ref{nlp-11}.b)(iv),
$E_3+Q^+(x)+(x|x)P^-=\exp\mathcal{G}_1(x)
E_3$.
Thus
$\mathcal{P}''
=\{\exp\mathcal{G}_1(x)E_3
|~x\in{\rm Im}{\bf O}\}
\subset Orb_{N_2}(E_3)$.

Second, fix $X\in\mathcal{P}_{Q^+(1)}$.
Because of $X\in \mathcal{P}_{P^-}$ and (2),
$X$ can be expressed by 
$X=E_3+Q^+(x_0)+(x_0|x_0)P^-$
for some $x_0\in{\bf O}$.
By direct calculations,
$0=Q^+(1)\times X={\rm Re}(x_0)P^-$
and $x_0\in {\rm Im}{\bf O}$.
Thus  $X\in \mathcal{P}''$ 
and so $\mathcal{P}_{Q^+(1)}\subset \mathcal{P}''$.

Third,
by (\ref{nlp-11}.a) and (\ref{nlp-11}.b)
$\exp\mathcal{G}_1(x)\exp\mathcal{G}_2(p)Q^+(1)
=Q^+(1)+(x|1)P^-=Q^+(1)$
for all $x,p\in{\rm Im}{\bf O}$.
Thus
$N_2$ 
is a subgroup of $(\mathrm{F}_{4(-20)})_{Q^+(1)}$ 
and so
$Orb_{N_2}(E_3)
\subset Orb_{(\mathrm{F}_{4(-20)})_{Q^+(1)}}(E_3)$.

Last, by virtue of the definition of
$\mathcal{P}_{Q^+(1)}$, 
$(\mathrm{F}_{4(-20)})_{Q^+(1)}$ 
acts on $\mathcal{P}_{Q^+(1)}$.
Thus $Orb_{(\mathrm{F}_{4(-20)})_{Q^+(1)}}(E_3)
\subset \mathcal{P}_{Q^+(1)}$
follows from $E_3\in\mathcal{P}_{Q^+(1)}$.
Consequently  $\mathcal{P}_{Q^+(1)}
\subset \mathcal{P}''\subset Orb_{N_2}(E_3)
\subset Orb_{(\mathrm{F}_{4(-20)})_{Q^+(1)}}(E_3)
\subset \mathcal{P}_{Q^+(1)}$.
Hence (\ref{sdp-06}.c) follows.
\end{proof}
\medskip

It follows from Lemma~\ref{sdp-06} 
that $\mathcal{P}_{E_3,P^-},$
$\mathcal{P}_{P^-}$ and $\mathcal{P}_{Q^+(1)}$
are parabolic type planes.
Let $f_i$ be the mappings from the suitable $\mathbb{R}^n$ 
to the parabolic type planes
defined as
\begin{align*}
f_1&:{\rm Im}{\bf O}\to\mathcal{P}_{E_3,P^-};
&&f_1(p):=\exp\mathcal{G}_2(p)(-E_1+E_2)
&&{\rm for}~p\in{\rm Im}{\bf O},\\
f_2&:{\bf O}\to\mathcal{P}_{P^-};
&&f_2(x):=\exp\mathcal{G}_1(x)E_3
&&{\rm for}~x\in{\bf O},\\
f_3&:{\rm Im}{\bf O}\to
\mathcal{P}_{Q^+(1)};
&&f_3(x):=\exp\mathcal{G}_1(x)E_3
&&{\rm for}~x
\in{\rm Im}{\bf O}
\end{align*}
respectively.

\begin{lemma}\label{sdp-07}
$f_i$ is a bijection for any $i\in\{1,2,3\}.$
\end{lemma}

\begin{proof}
Case $f_1$. By (\ref{sdp-06}.a),
$
\mathcal{P}_{E_3,P^-}
=\{f_1(p)~|~p\in{\rm Im}{\bf O}\}
$.
Then $f_1$
is onto.
Now if $p,q \in {\rm Im}{\bf O}$ and $p\ne q$, then 
$(-E_1+E_2)+2F_3^1(p)+2(p|p)P^-
\ne (-E_1+E_2)+2F_3^1(q)+2(q|q)P^-$.
Thus  $f_1$
is one-to-one.

Case $f_2$.
By (\ref{sdp-06}.b),
$
\mathcal{P}_{P^-}
=\{f_2(x)~|~x\in{\bf O}\}
$.
Then $f_2$
is onto.
Now if $x,y \in {\bf O}$ and $x\ne y$, then 
$f_2(x)=E_3+Q^+(x)+(x|x)P^-
\ne E_3+Q^+(y)+(y|y)P^-=f_2(y)$.
Thus  $f_2$
is one-to-one.

Case $f_3$.
By (\ref{sdp-06},c),
$
\mathcal{P}_{Q^+(1)}
=\{f_3(x)~|~x\in{\rm Im}{\bf O}\}
$.
Then $f_3$
is onto.
Now if $x,y \in {\rm Im}{\bf O}$ and $x\ne y$, then 
$f_3(x)=E_3+Q^+(x)+(x|x)P^-
\ne E_3+Q^+(y)+(y|y)P^-=f_3(y).$
Thus  $f_3$
is one-to-one.
\end{proof}
\medskip

The homomorphism
$\varphi_1:{\rm Spin}(7)\ltimes {\rm Im}{\bf O}
\to (\mathrm{F}_{4(-20)})_{E_3,P^-}$ 
is defined by the restriction 
$\varphi_1:=\varphi|{\rm Spin}(7)
\ltimes {\rm Im}{\bf O}$:
\[\varphi_1(g,p)=\varphi(g,p,0)
=\exp\mathcal{G}_2(p)\varphi_0(g)
\quad\text{for}~
(g,p)\in {\rm Spin}(7)\ltimes{\rm Im}{\bf O}.\]
From (\ref{cce-03})
and (\ref{nlp-11}.a)(iv), 
we see
$\varphi(g,p)E_3=E_3$,
$\varphi(g,p)P^-=P^-$.
So $\varphi_1$ is well-defined.
The homomorphism 
$\varphi_2:
\mathrm{G}_2\ltimes 
\mathrm{H}_{{\rm Im}{\bf O},{\rm Im}{\bf O}}
\to (\mathrm{F}_{4(-20)})_{Q^+(1)}$ is 
defined by the restriction
$\varphi_2:=\varphi
|\mathrm{G}_2\ltimes \mathrm{H}_{{\rm Im}{\bf O},
{\rm Im}{\bf O}}$:
\[\varphi_2(g,p,q)
=\exp(\mathcal{G}_2(p)+\mathcal{G}_1(q))\varphi_0(g)
\quad\text{for}~
(g,p,q)
\in \mathrm{G}_2\ltimes \mathrm{H}_{{\rm Im}{\bf O}, 
{\rm Im}{\bf O}}.\]
From (\ref{cce-03}),
(\ref{nlp-11}.a) and (\ref{nlp-11}.b),
we see
$\varphi_2(g,p,x)Q^+(1)=Q^+(1)$. 
So $\varphi_2$ is well-defined.
\medskip

\begin{proposition}\label{sdp-08} 
{\rm (1)}
$\varphi_1$ is an isomorphism
onto $(\mathrm{F}_{4(-20)})_{E_3,P^-}$.
\smallskip

{\rm (2)} $\varphi$ is an isomorphism
onto $(\mathrm{F}_{4(-20)})_{P^-}$.
\smallskip

{\rm (3)} $\varphi_2$ is an isomorphism
onto $(\mathrm{F}_{4(-20)})_{Q^+(1)}$.
\end{proposition}

\begin{proof} (1) We show that $\varphi_1$
is an onto and one-to-one.
Let  $\tilde{g}\in(F_{4(-20)})_{E_3,P^-}.$
From (\ref{sdp-06}.a),
we see
$
\tilde{g}(-E_1+E_2)=
\exp\mathcal{G}_2(p)(-E_1+E_2)
$
for some $p\in{\rm Im}{\bf O}$.
Because of
$\exp\mathcal{G}_2(-p)
\tilde{g}(-E_1+E_2)=-E_1+E_2$
and (\ref{nlp-11}.a),
we see
$\exp\mathcal{G}_2(-p)
\tilde{g}Y=Y$
where $Y=E_3$ or $P^-$
and therefore
$\exp\mathcal{G}_2(-p)\tilde{g}
\in (F_{4(-20)})_{-E_1+E_2,E_3,P^-}
=(F_{4(-20)})_{E_1,E_2,E_3,F_3^1(1)}
=B_3$.
Thus 
$\exp\mathcal{G}_2(-p)\tilde{g}
=\varphi_0(g)$
for some $g\in{\rm Spin}(7)$
and so $\tilde{g}
=\exp\mathcal{G}_2(p)\varphi_0(g)
=\varphi_1(g,p)$.
Hence $\varphi_1$ is onto.

Take $(g,p)\in {\rm Ker}(\varphi_1).$
By (\ref{cce-03}),  
$-E_1+E_2=\varphi_1(g,p)(-E_1+E_2)
=\exp\mathcal{G}_2(p) \varphi_0(g)(-E_1+E_2)
=\exp\mathcal{G}_2(p)(-E_1+E_2)
=f_1(p)$.
From Lemma~\ref{sdp-07}, we see $p=0$
and  $\varphi_1(g,p)=\varphi_0(g)$, 
and Lemma~\ref{cce-02}(2), 
$g=1$
where $1$ denotes the identity element of
$\mathrm{Spin}(7)$.
Thus  ${\rm Ker}(\varphi_1)
=\{(1,0)\}$
and so $\varphi_1$ is  one-to-one.
Hence (1) follows.

(2) 
Let $\tilde{g}\in(\mathrm{F}_{4(-20)})_{P^-}$.
From (\ref{sdp-06}.b), we see 
$\tilde{g}E_3=\exp\mathcal{G}_1(x)E_3$
for some $x\in{\bf O}$.
Because of
$\exp\mathcal{G}_1(-x)\tilde{g}E_3=E_3$
and 
$\exp\mathcal{G}_1(-x)\tilde{g}P^-=P^-$
(see (\ref{nlp-11}.b)),
we see
$\exp\mathcal{G}_1(-x)\tilde{g}
\in (\mathrm{F}_{4(-20)})_{E_3,P^-}$.
Thus by (1),
$\exp\mathcal{G}_1(-x)\tilde{g}
=\exp\mathcal{G}_2(p)\varphi_0(g)$
for some $(g,p)\in{\rm Spin}(7)
\ltimes{\rm Im}{\bf O}$
and so  $\tilde{g}
=\exp\mathcal{G}_1(x)\exp\mathcal{G}_2(p)\varphi_0(g)
=\varphi(g,p,x)$.
Hence  $\varphi$ is onto.

Next, take $(g,p,x)\in {\rm Ker}(\varphi)$.
From (\ref{cce-03}) 
and (\ref{nlp-11}.b), we see
$E_3=\varphi(g,p,x)E_3
=\exp\mathcal{G}_1(x)E_3
=f_2(x)$.
By Lemma~\ref{sdp-07}, $x=0$ 
and  $\varphi(g,p,0)
=\varphi_1(g,p)\in(\mathrm{F}_{4(-20)})_{E_3,P}$.
Then by (1),
$(g,p)\in {\rm Ker}(\varphi_1)
=\{(1,0)\}$.
Thus $ {\rm Ker}(\varphi)=\{(1,0,0)\}$ 
and so $\varphi$ is  one-to-one.
Hence (2) follows.

(3) By (2), the map $\varphi_2$ 
is a restriction map of isomorphism $\varphi$.
Thus $\varphi_2$ is a mono-morphism.
Take $\tilde{g}\in (\mathrm{F}_{4(-20)})_{Q^+(1)}.$
Because of $P^-=Q^+(1)^{\times 2}$, 
we easily
see that $(\mathrm{F}_{4(-20)})_{Q^+(1)}$
is a subgroup of $(\mathrm{F}_{4(-20)})_{P^-}$.
By (2), $\tilde{g}=\varphi(g,p,x)$
for some
$(g,p,x)\in{\rm Spin}(7)\ltimes
\mathrm{H}_{{\rm Im}{\bf O}, {\rm Im}{\bf O}}$
with
$g=(g_1,g_2,g_3)$.
From (\ref{cce-03}),
(\ref{nlp-11}.a)(vi) and (\ref{nlp-11}.b)(vi),
we see
$Q^+(1)=\varphi(g,p,x)Q^+(1)
=Q^+(g_1 1)+2(x|g_1 1)P^-$.
Then $g_1 1=1$ and $0=(x|g_1 1)$.
Because of $0=(x|g_1 1)=(x|1)$, we see 
$x\in{\rm Im}{\bf O}$, and
because of $g\in {\rm Spin}(7)$, 
$g_1 1=1$
and (\ref{trl-04}.b),
we see 
$g\in \mathrm{G}_2$.
Thus  $(g,p,x)\in \mathrm{G}_2\ltimes 
\mathrm{H}_{{\rm Im}{\bf O}, {\rm Im}{\bf O}}$
and so $\varphi_2$ is onto.
Hence (3) follows.
\end{proof}
\medskip

\begin{corollary}\label{sdp-09} 
\[\tag{\ref{sdp-09}}
(\mathrm{F}_{4(-20)})_{P^-}=N^+M=MN^+.\]
\end{corollary}
\begin{proof}
Because of (\ref{sdp-05}) 
and Proposition~\ref{sdp-08}(2),
$(\mathrm{F}_{4(-20)})_{P^-}=N^+M$.
By (\ref{sdp-03}), 
$\varphi_0(g)\exp(\mathcal{G}_2(p)
+\mathcal{G}_1(x))
=\exp(\mathcal{G}_2(g_3p)
+\mathcal{G}_1(g_1x))\varphi_0(g)$
for all
$(g,p,x)\in{\rm Spin}(7)\ltimes
\mathrm{H}_{{\rm Im}{\bf O}, {\bf O}}$ where 
$g=(g_1,g_2,g_3)$.
Thus $N^+M=MN^+$.
\end{proof}
\medskip

\begin{proposition}\label{sdp-10} 
Let $p,q\in\mathbb{R}$ and $p\ne q.$

{\rm (1)} Let $Y=P^-+q(E-E_3)+pE_3$
$\in\mathcal{J}^1$.
Then
$(\mathrm{F}_{4(-20)})_Y=(\mathrm{F}_{4(-20)})_{E_3,P^-}
\cong{\rm Spin}(7)
\ltimes {\rm Im}{\bf O}$.
\smallskip

{\rm (2)} Let  $Y'=P^++q(E-E_3)+pE_3$
$\in\mathcal{J}^1$.
Then
$(\mathrm{F}_{4(-20)})_{Y'}
=\tilde{\sigma}((\mathrm{F}_{4(-20)})_{E_3,P^-}
\cong{\rm Spin}(7)
\ltimes{\rm Im}{\bf O}$.
\end{proposition}

\begin{proof} (1) Obviously 
$(\mathrm{F}_{4(-20)})_{E_3,P^-}\subset(\mathrm{F}_{4(-20)})_Y.$
Conversely, take $g\in (\mathrm{F}_{4(-20)})_Y$.
Because of $pE-Y=(p-q+1)E_1+(p-q-1)E_2+F_1^1(-1)$
and (\ref{prl-06}.d), we see
$(pE-Y)^{\times 2}=(p-q)^2E_3$
and 
$\mathrm{tr}((pE-Y)^{\times 2})=(p-q)^2\ne 0$.
Then 
$E_{Y,p}\in\mathcal{J}^1$ is well-defined 
and $E_{Y,p}=E_3$.
By (\ref{prl-10})(iii),
$g E_3=gE_{Y,p}=E_{g Y,p}= E_{Y,p}=E_3$. 
Then 
$g P^-=g(Y-pE_3-q(E-E_3)))
=Y-pE_3-q(E-E_3)=P^-$.
Thus $g \in (\mathrm{F}_{4(-20)})_{E_3,P^-}$ 
and so $(\mathrm{F}_{4(-20)})_Y
\subset (\mathrm{F}_{4(-20)})_{E_3,P^-}$.
Hence  $(\mathrm{F}_{4(-20)})_Y=(\mathrm{F}_{4(-20)})_{E_3,P^-}
\cong{\rm Spin}(7)
\ltimes {\rm Im}{\bf O}$
follows from Proposition~\ref{sdp-08}(1). 

(2) Obviously
$(\mathrm{F}_{4(-20)})_{Y'}= (\mathrm{F}_{4(-20)})_{-Y'}$. 
Put $Z=P^--q(E-E_3)-pE_3.$
Because of $\sigma(-Y')=P^--q(E-E_3)-pE_3=Z$
and (1), we see that  
$(\mathrm{F}_{4(-20)})_{-Y'}
=\tilde{\sigma}((\mathrm{F}_{4(-20)})_{Z})
=\tilde{\sigma}((\mathrm{F}_{4(-20)})_{E_3,P^-})$.
By Proposition~\ref{sdp-08}(1), we have
$(\mathrm{F}_{4(-20)})_{Y'}
=\tilde{\sigma}((\mathrm{F}_{4(-20)})_{E_3,P^-})
\cong (\mathrm{F}_{4(-20)})_{E_3,P^-}\cong{\rm Spin}(7)
\ltimes {\rm Im}{\bf O}$.
\end{proof}
\medskip

\begin{proposition}\label{sdp-11} 
Let $r\in\mathbb{R}.$

{\rm (1)} Let $Y=P^-+rE$
$\in\mathcal{J}^1$.
Then
$(\mathrm{F}_{4(-20)})_Y
=(\mathrm{F}_{4(-20)})_{P^-}
\cong{\rm Spin}(7)
\ltimes \mathrm{H}_{{\rm Im}{\bf O},{\bf O}}$.
\smallskip

{\rm (2)} Let  $Y'=P^++rE$ 
$\in\mathcal{J}^1$.
Then
$(\mathrm{F}_{4(-20)})_{Y'}
=\tilde{\sigma}((\mathrm{F}_{4(-20)})_{P^-})
\cong{\rm Spin}(7)
\ltimes \mathrm{H}_{{\rm Im}{\bf O},{\bf O}}$.
\end{proposition}

\begin{proof} (1) 
Since the element $E$ is invariant
under the $\mathrm{F}_{4(-20)}$-action, 
$(\mathrm{F}_{4(-20)})_Y=(\mathrm{F}_{4(-20)})_{P^-}.$
By Proposition~\ref{sdp-08}(2),
we have
 $(\mathrm{F}_{4(-20)})_Y
\cong{\rm Spin}(7)
\ltimes \mathrm{H}_{{\rm Im}{\bf O},{\bf O}}$.

(2)
Obviously
$(\mathrm{F}_{4(-20)})_{Y'}= (\mathrm{F}_{4(-20)})_{-Y'}$.
Put $Z=P^--rE.$
Because of $\sigma(-Y')=P^--rE=Z$ 
and (1), we see
$(\mathrm{F}_{4(-20)})_{-Y'}
=\tilde{\sigma}((\mathrm{F}_{4(-20)})_{Z})=
\tilde{\sigma}((\mathrm{F}_{4(-20)})_{P^-})$.
By Proposition~\ref{sdp-08}(2),
we obtain that
$(\mathrm{F}_{4(-20)})_{Y'}
=\tilde{\sigma}((\mathrm{F}_{4(-20)})_{P^-})
\cong (\mathrm{F}_{4(-20)})_{P^-}
\cong{\rm Spin}(7)
\ltimes \mathrm{H}_{{\rm Im}{\bf O},{\bf O}}$.
\end{proof}
\medskip

\begin{proposition}\label{sdp-12}
Let  $Y=Q^+(1)+rE$
$\in\mathcal{J}^1$
where $r\in\mathbb{R}.$
Then
$(\mathrm{F}_{4(-20)})_Y
=(\mathrm{F}_{4(-20)})_{Q^+(1)}
\cong \mathrm{G}_2
\ltimes \mathrm{H}_{{\rm Im}{\bf O},{\rm Im}{\bf O}}$.
\end{proposition}

\begin{proof} Since the element $E$ is invariant
under the $\mathrm{F}_{4(-20)}$-action, we see
$(\mathrm{F}_{4(-20)})_Y=(\mathrm{F}_{4(-20)})_{Q^+(1)}$. 
Hence  it
follows from Proposition~\ref{sdp-08}(3).
\end{proof}
\medskip

\begin{proof}[{\bf Proof of Main Theorem~\ref{stabilizer-main}}]
By Main Theorem~\ref{orb-decomposition}, 
we have  a concrete orbit decomposition of $\mathcal{J}^1$
under the group $\mathrm{F}_{4(-20)}$.
Because of Propositions~\ref{cce-04}, 
\ref{spn-09}, 
\ref{spn-12},
\ref{sdp-10}, 
\ref{sdp-11}, 
\ref{sdp-12} 
and $g E=E$ for all $g \in \mathrm{F}_{4(-20)}$,
we determine the Lie group structure of
the stabilizer for each $\mathrm{F}_{4(-20)}$-orbit 
on $\mathcal{J}^1$.
\end{proof}
\medskip

\begin{remark}\label{remark-spd-11}{\rm
Denote the quaternions
${\bf H}:=\{\sum_{i=0}^3 a_i e_i|~a_i\in \mathbb{R}\}$
and ${\bf F}$  
a real division algebra $\mathbb{R}$, 
${\bf C}$, ${\bf H}$
or ${\bf O}$.
J.A.~Wolf (\cite{{Wja2007},{Wja1975}}) gave
{\it Heisenberg groups} 
$H_{p,q,{\bf F}}$ and $G_{p,q,{\bf F}}.$
Then $H_{1,0,{\bf O}}$ 
is equal to the group 
$\mathrm{H}_{{\rm Im}{\bf O}, {\bf O}}$
and $G_{1,0,{\bf O}}$ is equal to the group 
${\rm Spin}(7)\ltimes
\mathrm{H}_{{\rm Im}{\bf O}, {\bf O}}$.
F.W.~Keene showed $MN^+\cong G_{1,0,{\bf O}}$
in his thesis (cf. \cite{Ken1978}, \cite{Wja1975}).
In Propositions~\ref{sdp-08} 
and \ref{sdp-09}, 
it appears that the subgroup
$MN^+\cong G_{1,0,{\bf O}}$ 
in $\mathrm{F}_{4(-20)}$ is the stabilizer
of the element $P^-$.
}\end{remark}

{\bf Acknowledgment.}\quad\quad
The author would like to thank Professor Osami Yasukura 
for his advices and encouragements.

\end{document}